%% file: ejs-template.tex
\documentclass[ejs]{imsart}

\usepackage[utf8]{inputenc}
\usepackage{amsfonts, amsmath, amssymb, amsthm, bbm}

\usepackage{color}
\usepackage{ulem}

\usepackage{graphicx}
\usepackage{array,multirow,makecell}
\usepackage{caption}
\usepackage{bbm}
\usepackage{wrapfig}
\usepackage{subcaption}
\usepackage{hyperref}

\usepackage{textcomp}
\usepackage{comment}

\usepackage[dvipsnames]{xcolor}

\RequirePackage[shortlabels]{enumitem}

\newcommand{\pa}[1]{\left(#1\right)}
\newcommand{\cro}[1]{\left[#1\right]}
\newcommand{\ac}[1]{\left\{#1\right\}}

\newcommand{\bPi}{\boldsymbol{\Pi}}
\newcommand{\E}{\operatorname{\mathbb{E}}}
\renewcommand{\P}{\operatorname{\mathbb{P}}}
\newcommand{\R}{\mathbb{R}}

\newcommand{\xbf}{\mathbf{x}}
\newcommand{\hxbf}{\hat{\mathbf{x}}}
\newcommand{\hxsp}{\hat{\mathbf{x}}^{\textrm{VSA}}}
\newcommand{\bS}{\overline{S}}
\newcommand{\eps}{\varepsilon}

\newcommand{\Ccal}{\mathcal{C}}
\newcommand{\Dcal}{\mathcal{D}}

\newcommand{\Fcal}{\mathcal{F}}

\newcommand{\Ocal}{\mathcal{O}}

\newcommand{\Rcal}{\mathcal{R}}
\newcommand{\Scal}{\mathcal{S}}

 \newcommand{\mucal}{\mu_{{\hat{\bf x}}^{(1)}_S, {\bf x}^*_{S}}(Q)}

\newcommand{\si}{\sigma}
\newcommand{\de}{\delta}

\newcommand{\ep}{\varepsilon}

\DeclareMathOperator*{\argmax}{argmax}
\DeclareMathOperator*{\argmin}{argmin}

\newcommand{\eref}[1]{(\ref{#1})}

\newtheorem{thm}{Theorem}[section]
\newtheorem{defi}[thm]{Definition}
\newtheorem{setting}{Setting}
\newtheorem{prop}[thm]{Proposition}

\newtheorem{lem}[thm]{Lemma}
\newtheorem{cl}[thm]{Claim}

\newtheorem{cor}[thm]{Corollary}

\renewcommand{\contentsname}{}

\doi{10.1214/154957804100000000}
\pubyear{0000}
\volume{0}
\firstpage{1}
\lastpage{1}

\startlocaldefs
\endlocaldefs

\begin{document}

\begin{frontmatter}

\title{Localization in 1D non-parametric latent space models from pairwise affinities}
\runtitle{Localization in 1D latent space}


\author{\fnms{Christophe} \snm{Giraud}\ead[label=e1]{christophe.giraud@universite-paris-saclay.fr}}

\address{Laboratoire de Math\'ematiques d'Orsay, Universit\'e Paris-Saclay, CNRS, France.\\
\printead{e1}}

\author{\fnms{Yann} \snm{Issartel}\ead[label=e2]{yann.issartel@telecom-paris.fr}}
\address{CREST - ENSAE, \\ Télécom Paris, Institut Polytechnique de Paris, France.\\
\printead{e2}}

\author{\fnms{Nicolas} \snm{Verzelen}\ead[label=e3]{nicolas.verzelen@inrae.fr}}

\runauthor{C. Giraud et al.}

\address{INRAE, Montpellier SupAgro, MISTEA,
Univ. Montpellier, France.\\
\printead{e3}}

\begin{abstract}
 We consider the problem of estimating latent positions in  a one-dimensional torus from pairwise affinities. The observed affinity between a pair of items is modeled as a noisy observation of a function $f(x^*_{i},x^*_{j})$ of the latent positions $x^*_{i},x^*_{j}$ of the two items on the torus. 
 The affinity function $f$ is unknown, and it is only assumed to fulfill some shape constraints ensuring that $f(x,y)$ is large when the distance between $x$ and $y$ is small, and vice-versa. This non-parametric modeling offers a good flexibility to fit data.
 We introduce an estimation procedure that provably  localizes all the latent positions with a maximum error of the order of $\sqrt{\log(n)/n}$, with high-probability. This rate is proven to be minimax optimal. A computationally efficient variant of the procedure is also analyzed under some more restrictive assumptions. Our general results can be instantiated  to the problem of statistical seriation, leading to new bounds for the maximum error in the ordering. 
\end{abstract}





\end{frontmatter}



 
\section{Introduction}
\subsection{1D latent localization problem}\label{sect11:intro}
We consider the 1D latent localization problem, where we seek to recover the 1D latent positions of $n$ objects from pairwise similarity measurements. Such problems arise in archeology for relative dating of objects or graves \cite{Robinson51}, in 2D-tomography for angular synchronization \cite{2D-tomo, singer2011angular},  in bioinformatics for reads alignment in {\it de novo} sequencing \cite{bioinfo17}, in computer science for time synchronization in distributed networks \cite{Clock-Synchro04,Clock-Synchro06}, or in matchmaking problems \cite{BradleyTerry1952}. The data are collected as a $n \times n$ symmetric matrix $[A_{ij}]_{1\leq i,j \leq n}$, called affinity matrix, which provides similarity measurements between pairs of objects. These similarity measurements $A_{ij}$ can be real valued scores, or they can be binary pieces of information, as when the matrix $A$ encodes a network structure.

In 1D latent space models \cite{hoff2002latent}, the affinity matrix is assumed to be sampled as follows. The distribution is parametrized by a 1D metric space $(\mathcal{X},d)$, some (possibly random) latent positions $(x^*_1,\ldots, x^*_n)\in \mathcal{X}^n$ and an affinity function 
$f:\mathcal{X}\times \mathcal{X}\to \mathbb{R}$. Then, conditionally on   $(x^*_1,\ldots, x^*_n)$, the upper-diagonal entries $A_{ij}$ of the affinity matrix are sampled independently, with conditional mean $f(x^*_{i},x^*_{j})$. The affinity $f(x^*_i,x^*_j)$ is typically assumed to  decrease as the metric distance $d(x^*_i,x^*_j)$ increases. In particular, close points $x^*_i$ and $x^*_j$ share a high affinity, whereas distant points share a small affinity. These latent space models encompass many classical models, as exemplified in the next paragraphs. \medskip

\noindent 
\textit{Example 1: Random Geometric Graph}~\cite{Gilbert61,penrose2003random,diaz2020learning,de2017adaptive}. We observe a random graph with $n$ nodes labelled by $\ac{1,\ldots,n}$. The graph is encoded into an adjacency matrix $A\in \{0,1\}^{n \times n}$,  by setting $A_{ij}=1$ if there is an edge between nodes $i$ and $j$, and $A_{ij}=0$ otherwise. Let  $\mathcal{C}$ denote the unit sphere in $\mathbb{R}^2$ endowed with the geodesic distance $d$. In the circular random geometric graph model, the edges are sampled are sampled independently, with probability $\P[A_{ij}=1]= g(d(x^*_i,x^*_j))$, where $g: [0,\pi]\mapsto [0,1]$ is a non-increasing function and $x^*_1,\ldots, x^*_n\in \Ccal$ are the  latent positions of the nodes on the sphere. This random graph model is therefore an instance of 1D latent space model where $A_{ij}\in \{0,1\}$, $\mathcal{X}=\Ccal$ and $f(x^*_i,x^*_j) = g(d(x^*_i,x^*_j))$.
\medskip 

\noindent 
\textit{Example 2:  Graphons and $f$-Random Graphs} \cite{diaconis, lovasz2012large}. 
The class of  $f$-random graph models, also called graphon models, 
encompasses all the distributions on random graphs that are invariant by permutation of nodes.
It is parametrized by the set of measurable functions $f:[0,1]\times [0,1]\to [0,1]$. The adjacency matrix
$A$ of the graph is sampled as follows. First, $n$ latent positions $x^*_{1},\ldots, x^*_{n}$ are sampled i.i.d.\ uniformly on $[0,1]$.
Then, conditionally on $x^*_{1},\ldots, x^*_{n}$, the edges are sampled independently, 
with conditional probability $\P[A_{ij}=1|x^*_{1},\ldots, x^*_{n}]= f(x^*_i,x^*_j)$.
The $f$-random graph model is then an instance of 1D latent space model where $A_{ij}\in \{0,1\}$, and $\mathcal{X}=[0,1]$. Unless some additional constraints are imposed on the shape of $f$, the affinity $f(x^*_i,x^*_j)$ may vary arbitrarily with the distance $|x_i^*-x_j^*|$. 
\medskip 

\noindent 
\textit{Example 3: R-Matrices and Statistical Seriation}. 
A Robinson matrix (R-matrix) is any symmetric matrix $B\in \mathbb{R}^{n\times n}$ whose entries decrease when moving away from the diagonal, i.e.\ such that $B_{i,j} \geq  B_{i+1,j}$ and  $B_{i,j} \geq  B_{i,j-1}$, for all $1\leq j \leq i\leq n$. A matrix $F$ is called a pre-R matrix, when there exists a permutation  $\sigma\in \Sigma_n$ of $\ac{1,\ldots,n}$, such that $F_{\sigma}= [F_{\sigma(i),\sigma(j)}]_{i,j}$ is an R-matrix. The  noisy seriation problem \cite{fogel2013convex} amounts to find, from a  noisy observation of a pre-R matrix 
$F$, a permutation $\sigma^*$ such that $F_{\sigma^*}$ is a R-matrix. This problem
appears in genomic sequencing \cite{garriga2011banded}, in interval graph identification \cite{fulkerson1965incidence}, and in envelope reduction for sparse matrices \cite{barnard1995spectral}. 
This problem can be recast in the latent space terminology using $\mathcal{X}=\ac{1,\ldots,n}$, $x^*_i=\sigma^*(i)$ and the affinity function $f(x^*_i,x^*_j)= F_{\sigma^*(i),\sigma^*(j)}$. Since  $F_{\sigma^*}$ is a R-matrix, the function $f(x_i^*,x_j^*)$ is decreasing with the distance $|x_i^*-x_j^*|$. 
\medskip 

\noindent 
\textit{Example 4: Toroidal R-Matrices and Toroidal Seriation}. Consider the set $\{1,\ldots,$ $n\}$ as a torus with the corresponding distance $d(i,j)=\min(|j-i|, |n+i-j|)$ for any $1\leq i, j\leq n$. A toroidal R-matrix is any symmetric matrix $B$ whose entries decrease when moving away from the diagonal with respect to the toroidal distance: 
 $B_{i,j} \geq B_{i+1,j}$ when $d(i,j)< d(i+1,j)$ and $B_{i,j} \geq B_{i,j+1}$ when $d(i,j)< d(i,j+1)$. As in Example 3 above, a pre-toroidal R-matrix is defined as a permutation of a toroidal R-matrix and the statistical seriation model is defined analogously \cite{recanati2018reconstructing}. Again, we can recast this model as a latent space model on $\mathcal{X}=\ac{1,\ldots,n}$ endowed with the toroidal distance. Alternatively, we can also rewrite it as a latent space model on the regular grid $\mathcal{C}_n$ of the unit sphere $\mathcal{C}$ corresponding to the $n$-th unit roots, endowed with the geodesic distance on $\Ccal$. 
\medskip

In the following, we assume that we observe a symmetric matrix $[A_{ij}]_{1\leq i ,j \leq n}$ of pairwise affinity measurements, with 
$A_{ii}=0$ (by convention) and 
$$A_{ij}=f(x^*_{i},x^*_{j})+E_{ij},\quad \textrm{for}\ 1\leq i < j \leq n,$$
where\\
(i) $x^*_{1},\ldots,x^*_{n}$ are  $n$ unobserved latent positions spread on the unit sphere $\mathcal{C}$  in $\mathbb{R}^2$, \\
(ii) $f:\Ccal \times \Ccal \to [0,1]$ is unobserved, symmetric, decreasing with the geodesic distance $d(x,y)$, and\\
(iii) $[E_{ij}]_{1\leq i < j \leq n}$ are some independent sub-Gaussian random variables. \smallskip\\
 This non-parametric framework is very flexible for fitting pairwise affinity data.
 It encompasses the circular random geometric graph model (Example 1) and the toroidal statistical seriation model (Example 4). 
 
Our overall goal is to recover the $n$-tuple  of latent positions $\xbf^*=(x^*_{1},\ldots,x^*_{n})$ $\in \Ccal^n$, with some high-confidence, simultaneously for all individual positions $x^*_{i}$. As the global error of an estimator $\hat \xbf$, say $d_{2}(\hat\xbf,\xbf^*)=\sqrt{\sum_{i=1}^n d(\hat x_{i},x^*_{i})^2}$, provides limited information on each  individual error $d(\hat x_{i},x^*_{i})$, we focus instead on the  maximum error  
\begin{equation}\label{eq:maxd}
d_{\infty}(\hat \xbf,\xbf^*)= \max_{i=1,\ldots n} d(\hat x_{i},x^*_{i}).
\end{equation}
We propose some estimators $\hat \xbf$ achieving, with high-probability, a maximum error $d_{\infty}(\hat \xbf,\xbf^*)$ of the order of $\sqrt{\log(n)/n}$, under the assumptions that the latent positions $x^*_{1},\ldots,x^*_{n}$  are sufficiently spread on $\Ccal$ and some shape conditions relative to the decreasing of $f(x,y)$ with $d(x,y)$. The  $\sqrt{\log(n)/n}$-rate of estimation is shown to be optimal. To the best of our knowledge, these are the first optimal results  on maximum error $d_{\infty}(\hat \xbf,\xbf^*)$ in  latent space models with unknown and non-parametric affinity function $f$.

\subsection{Our contribution}
As explained above, our overall goal is to recover the $n$-tuple of latent positions $\xbf^*=(x^*_{1},\ldots,x^*_{n})$ with a control on the maximum error \eref{eq:maxd}. Unfortunately, this program cannot be carried out literally, as the latent positions are not identifiable from the distribution of the data. Indeed, for any bijective map $\varphi:\Ccal\to \Ccal$, we have $f(x,y)=f\circ \varphi^{-1}(\varphi(x),\varphi(y))$ for all $x,y\in \Ccal$, with the notation   $f\circ \varphi^{-1}(x,y):=f(\varphi^{-1}(x),\varphi^{-1}(y))$. Even if we would enforce some strong shape constraints, like $f(x,y)=1-\alpha d(x,y)$ with $\alpha>0$, since $f(x,y)=f(Qx,Qy)$ for any orthogonal transformation $Q$ of $\Ccal$, the distribution of the data would still be invariant by orthogonal transformation of the latent positions. Hence, we face a delicate identifiability issue. This identifiability issue is fully explained and tackled in Section \ref{sec:identif}. Informally, our remedy is to provide
some estimators $\hat \xbf$ which are, under some assumptions, at the distance $d_{\infty}(\hat \xbf,\xbf^*)=O(\sqrt{\log(n)/n})$ of some specific representative $\xbf^*$ of the latent positions. 

Our shape assumption (ii) on the affinity function $f$ ensures that the matrix $[f(x_{i},x_{j})]_{i,j=1,\ldots,n}$ is (approximately) a toroidal pre-R matrix. We observe that the constant function $f(x,y)=1$ fulfills assumption (ii), and that for this specific function there is no hope to recover any information on the latent positions, even in the noiseless case. To circumvent this issue,  we introduce a {\it bi-Lipschitz} assumption, detailed in Section~\ref{section:setting}, constraining the decay of $f$ with $d$. In the specific case of the random geometric graph model with $g$ continuously  differentiable, this condition merely amounts to require $g'(x)$ to be bounded away from 0. 

Our estimation procedures proceed in two main stages:\\
(1) we start with an initial localization with a global control in $d_{1}(\xbf,{\bf y}):=\sum_{i=1}^n d(x_{i},y_{i})$ distance,\\
(2) then, for each point,  we refine  this first estimator to get a control in $d_{\infty}$ distance.\\
In order to avoid some nasty statistical dependencies between the two stages, we use a sample splitting scheme ensuring that, at the second stage, the refinement  uses data independent from those used at the first stage. 

Let $S\subset \ac{1,\ldots,n}$ be a subset of indices sampled uniformly at random, and $\bS= \ac{1,\ldots,n}\setminus S$. 
At the second step, the refined estimator $\hat\xbf^{(2)}_{\bS}$ of $\xbf^*_{\bS}:=(x^*_{i})_{i\in \bS}$, can take as input any initial estimator  $\hat\xbf^{(1)}_{S}$ of $\xbf^*_{S}=(x^*_{i})_{i\in S}$. This second step has a polynomial computational complexity and, under appropriate assumptions, it fulfills with high-probability
$$d_{\infty}(\hat \xbf^{(2)}_{\bS},\xbf^*_{\bS}) \leq C\ \pa{d_{1}(\hat \xbf^{(1)}_{S},\xbf^*_{S})\over n}\vee \sqrt{\log(n)\over n},$$
for some specific representative $\xbf^*$ of the latent positions. 
Hence, in order to get the desired  bound $d_{\infty}(\hat \xbf^{(2)}_{\bS},\xbf^*_{\bS})=O\left(\sqrt{\log(n)/ n}\right)$, we need an initial control $d_{1}(\hat \xbf^{(1)}_{S},\xbf^*_{S})=O\left(\sqrt{n\log(n)}\right)$. We propose two estimators fulfilling this requirement:\\
(a) a first one, which requires no additional assumptions, but which has a super-polynomial computational complexity;\\
(b) a second one, adapted from~\cite{recanati2018reconstructing}, which has a polynomial computational complexity, but for which we prove a $O\big(\sqrt{n\log(n)}\big)$ control only for a class of random geometric graphs. 

Repeating the sampling of $S$ and merging the resulting estimators, we then get an estimator $\hat \xbf$ achieving, with high-probability and under appropriate assumptions, $d_{\infty}(\hat \xbf,\xbf^*)=O(\sqrt{\log(n)/n})$ for a specific representative $\xbf^*$ of the latent positions. 
A matching lower bound is also derived, proving the optimality of the $\sqrt{\log(n)/n}$ rate. 
The significance of the improvement offered by the refinement step, and the impact of the sample splitting on the localization error are investigated numerically.

\subsection{Related work}

In the last decade, the analysis of interaction data has given rise to numerous works in machine learning and statistics. Most of these works handle  cases where the affinity function $f$ is either known or belong to a known parametric model. There is a long standing debate on the validity of such a rigid modeling \cite{BallingerWilcox1997}. Our modeling assumptions, with only shape constraints on $f$, offers a more  flexible setting to fit data.

\paragraph{Latent points estimation in random geometric graphs.} 
Random geometric graphs have attracted a lot of attraction as a simple model for wireless communications or internet \cite{Gilbert61, penrose2003random}. In the most classical setting, $A_{ij}=f(x^*_{i},x^*_{j})={\bf 1}_{\|x^*_{i}-x^*_{j}\|\leq r}$ for some $r>0$.
The problem of estimating the latent positions $x_1^*,\ldots,x_n^*$ in a square of $\mathbb{R}^2$ has been tackled by  \cite{diaz2020learning}.  
Compared to us, they consider the noiseless setting, where $E_{ij}=0$, with the affinity map $f$ belongs to 1-dimensional parametric model. The problem of latent positions localization has also been investigated in the  random dot-product graph \cite{sussman2013consistent,Lyzinski17,athreya2021estimation}, where, conditionally to the latent positions, the entries $A_{ij}$ of the adjacency matrix are independent\  Bernoulli random variables with mean $f(x^*_{i},x^*_{j})=\langle x^*_{i},x^*_{j}\rangle$, where $\langle\cdot,\cdot\rangle$ is the Euclidean scalar product in $\R^d$. In this case,  the function $f$ is known and the results are of an asymptotic nature. 

\paragraph{Phase synchronization Problems.} The phase synchronization problem \cite{singer2011angular} amounts to estimating unknown angles $\theta_1,\ldots,\theta_n$ from noisy measurements of $\theta_i- \theta_j \textup{ mod } 2\pi$. 
A version of this problem is when we seek to retrieve $x^*_1,\ldots,x^*_n$, which are spread on the unit complex sphere $\mathbb{C}_1 =\{x \in \mathbb{C} : |x| = 1\}$, from   noisy observations of $x^*_i\overline{x}^*_j= e^{\iota (\theta_i - \theta_j)}$.
In this model, some minimax $\ell^2$-bounds on the localization error have been obtained by  \cite{gao2020exact},  without assumptions on the latent positions $x^*_1,\ldots,x^*_n$.
Such a model is close to the pairwise affinity model on $\Ccal$.
The main differences compared to our setting is that we focus on $d_{\infty}$-bounds, with an unknown function $f$, which does not have the affinity shape (ii).

\paragraph{Skills estimation in the Bradley-Terry model.} 
In the Bradley-Terry  model \cite{BradleyTerry1952}, the observations $A_{ij}$ are independent Bernoulli outcomes with mean  $f(x^*_{i},x^*_{j})=\sigma(x^*_{i}-x^*_{j})$, where  $x^*_{i}\in \R$ represents the skill of individual $i$ and $\sigma(x)=e^{x}/(1+e^{x})$ is the sigmoid function. The estimation of the skills by a spectral algorithm, or two-steps variants of it, has received a lot of attention recently   \cite{RankCentrality,TopK2015, BTL2019}. In particular, building on the structure of the problem, rate-minimax  $\ell_{\infty}$-bound have been derived for the spectral algorithm in the Bradley-Terry model when the skills belong to a compact set, possibly with missing at random observations.The Bradley-Terry model is a special instance of the 1D-latent space model. Compared to our setting, the function $f$ is known and it does not fulfill the affinity properties (for example, it is not symmetric).

\paragraph{Seriation from pairwise affinity.}
Given a pre-R matrix $F$, the seriation problem seeks to find the latent order $\si^*$ such that $F_{\si^*}$ is a R-matrix. For this noiseless version of Example 3, efficient algorithms have been proposed using convex optimization \cite{fogel2013convex}, or spectral methods \cite{atkins1998spectral}. The exact seriation problem has been solved on toroidal R-matrices in the noiseless case \cite{recanati2018reconstructing}, by using a spectral algorithm. A perturbation analysis has also been sketched in~\cite{recanati2018reconstructing}. As a byproduct of our analysis, we provide some more explicit recovery bounds in our specific setting with noisy observations. 
 Closer to our contribution, Jannssen and Smith~\cite{janssen2020reconstruction} observe a noisy version of a pre-R matrix and, under some assumptions on the affinity function $f$,  learn a permutation that satisfies $\max_{i\in [n]}|\widehat{\sigma}_i- \sigma^*_i|\leq C \sqrt{n}\log^{5}(n)$. Although their assumptions on $f$ are not directly comparable to ours, the localization rates are (up to logarithmic factors) comparable to ours. We refer to the discussion below Corollary~\ref{cor:seriation} for more details.

\paragraph{Two-step methods for latent space models.} Our work is related to the global-to-local estimation strategy, that was originally introduced in Stochastic Block Models and more generally in clustering analysis~\cite{lei2014generic,gao2017achieving,zhang2016minimax}, for the purpose of deriving sharp recovery bounds with polynomial time procedures. The general idea is to build upon an initial estimator that satisfies a certain (weak) consistency condition, and then apply greedy-type procedures (e.g. Lloyd's algorithm) to obtain minimax recovery bounds. This approach turned out to be fruitful in various latent space problems with discrete structure~\cite{chen2018network,gao2019iterative} and our procedure can be interpreted as one instance of this strategy in a non-parametric setting with a continuous latent space.

\subsection{Notation and organization of the paper}

\textit{Notation:} In the sequel,  $C, C', C''>0$ denote numerical constants that may change from line to line. For two functions or sequences $x$ and $y$, we write $x\lesssim y$ (resp. $x\gtrsim  y$) when, for some numerical constant $C>0$, we have $x\leq Cy$ (resp. $x\geq Cy$). The maximum (resp. minimum) of $x$ and $y$ is denoted by $x\vee y$ (resp. $x\wedge y$).
For any  $x>0$, we write $\lfloor x \rfloor$ for its integer part, and $[x]$ for the set of integers  $[x]=\ac{1,\ldots,\lfloor x \rfloor}$ . 
For $q\geq 1$, the entry-wise $l_q$ norm of a matrix  $F=(f_{ij})$ is denoted by $\|F\|_q= (\sum_{ij}|f_{ij}|^{q})^{1/q}$, the $i^{\textup{th}}$-row of $F$ is denoted by $F_i$, and the Frobenius scalar  product between two matrices $F$ and $G$ is denoted by $\langle F, G\rangle$.  Let $\Sigma_n$ be the collection of permutations of $[n]$. For any permutation 
$\si\in\Sigma_n$, and for any $n$-tuple ${\bf x}$ of size $n$, we define ${\bf x}_{\si}$ as the permuted $n$-tuple ${\bf x}_{\si}=(x_{\si(1)},\ldots,x_{\si(n)})$.

Assimilating  points in the  unit sphere $\Ccal$ of $\mathbb{R}^2$ to complex numbers with unit norm, we can represent  $x \in \Ccal$ by $x = e^{\iota \underline{x}}$, with $\underline{x}\in [0,2\pi)$. We call henceforth argument of $x\in\Ccal$ the real number $\underline{x}$. The geodesic distance $d(x,y)$ on $\mathcal{C}$ can be conveniently defined in terms of the arguments  of $x$ and $y$
\begin{equation}\label{eq:deodesic}
d(x,y) = \left|\underline{x}-\underline{y}\right| \wedge \left(2\pi - \left|\underline{x}-\underline{y}\right|\right).
\end{equation}
For any positive integer $k$, we define the regular grid $\mathcal{C}_k=\{1, e^{\iota 2\pi/k},\ldots,$ $e^{\iota 2\pi(k-1)/k}\}$, which plays an important role in our analysis and algorithms. We denote by $\mathcal{O}$  the orthogonal group  of $\mathbb{R}^2$ made of rotations and reflections, and for any $n$-tuple $\xbf=(x_{1},\ldots,x_{n})\in \Ccal^n$, and any $Q\in \mathcal{O}$, we define
$Q\xbf:= (Qx_{1},\ldots,Qx_{n})$. For two subsets $S,S'$ of $\ac{1,\ldots,n}$, a matrix $A\in \R^{n\times n}$ and a $n$-tuple $\xbf=(x_{1},\ldots,x_{n}) \in \Ccal^n$, we define $A_{SS'}=[A_{ij}]_{i\in S,j\in S'}$ and $\xbf_{S}=[x_{i}]_{i\in S}$. More generally, we denote by $\xbf_{S}\in\Ccal^{S}$ a $|S|$-tuple indexed by $S$.
The complement of a  set $S$, is denoted by $\overline{S}$. 
\medskip 

\textit{Organization:} In  Section~\ref{sec:problem}, we describe the statistical setting and we discuss thoroughly the identifiability issues. The main embedding procedure, called Localize-and-Refine, is presented in Section~\ref{seriation:section:resul}. A spectral variant of this procedure is introduced in Section~\ref{sec:geometric}, with an application to geometric models. 
In  Section \ref{sec:numerique}, we investigate numerically the usefulness of  the sample splitting, and the significance of the improvement offered by the refinement step. 
We summarize our findings and discuss an open problem in Section~\ref{sec:discussion}. All the proofs are postponed to the appendices.

\section{Model assumptions and identifiability issues}\label{sec:problem}

\subsection{Statistical setting}\label{section:setting}  

We observe a realization of a symmetric random matrix $A\in\R^{n\times n}$, whose values on the diagonal are $A_{ii}=0$. We denote by $F_{ij}=\E[A_{ij}]$ the mean value of $A_{ij}$ and by $E_{ij}=A_{ij}-F_{ij}$ the centered random fluctuation. We assume that $A$ has been generated by a latent space model on $\Ccal$: there exist a $n$-tuple $\xbf^*=(x_{1}^*,\ldots,x_{n}^*)\in \Ccal^n$ and a function $f:\Ccal\times \Ccal \to \R$ such that $F_{ij}=f(x^*_{i},x^*_{j})$, so
\begin{equation}\label{eq:model}
A_{ij}=F_{ij}+E_{ij}=f(x^*_{i},x^*_{j})+E_{ij},\quad \textrm{for} \quad 1\leq i<j\leq n.
\end{equation}
Both the function $f$ and the latent positions $x_{1}^*,\ldots,x_{n}^*$ are unknown.
We emphasize that the latent positions $\xbf^*=(x_{1}^*,\ldots,x_{n}^*)$ are assumed to be fixed\footnote{if they were random, our results would apply conditionally on the sampling of $\xbf^*$.} and we denote by $\P_{(\textup{{\bf x}}^*,f)}$ the distribution of $A$. Let us describe our assumptions on the spreading of the  latent positions $x_{1}^*,\ldots,x_{n}^*$, the shape of $f$, and the random fluctuations $E_{ij}$.

\paragraph{Spreading of the latent positions.} We have in mind that the latent positions are well spread over the unit sphere. We do not strictly enforce this condition, but our error bounds depend on how far the latent positions are from a regular position on $\Ccal$. 
More precisely, let us denote by  $\bPi_{ n}$ the set of regular positions on the unit sphere 
$\bPi_{n}=\ac{\xbf=(e^{ \iota 2\pi \si(j)/n})_{1\leq j \leq  n}:\sigma\in\Sigma_{ n}}$. 
Our results involve the $d_{\infty}$-distance of the $n$-tuple of latent positions $\xbf^*$ to  the set $\bPi_{n}$ of regular positions
\begin{equation}\label{alpha:def:new}
d_{\infty}(\xbf^*,\bPi_{n}) := \underset{{\bf y}\in \bPi_{ n}}{\textup{min}} \ \,  d_{\infty}({\bf x}^*,{\bf y})\enspace ,
\end{equation}
with $d_{\infty}({\bf x}^*,{\bf y})$ defined in \eref{eq:maxd}.

\paragraph{Bi-Lipschitz shape of $f$.} As explained in the introduction, we have in mind that $f(x,y)$ decreases with the distance $d(x,y)$. Since there is no hope to recover the latent positions $\xbf^*$ when the function $f$ is flat, we impose a minimal decreasing of $f(x,y)$ with the distance $d(x,y)$. We also require some Lipschitz continuity of $f$ for our analysis. 
These two conditions on $f$ are enforced by the Bi-Lipschitz condition described below.

\begin{defi}\label{defi:BL} {\bf Bi-Lipschitz functions.} \ For any fixed constants $c_e\geq 0$ and $0 < c_l\leq c_L$, we define 
 $\mathcal{BL}[c_l,c_L,c_e]$ as the set made of all functions $f: \Ccal^2 \rightarrow [0,1]$ that are symmetric (i.e. $f(x,y)= f(y,x)$ for all $x,y\in\mathcal{C}$) and that satisfy the two following conditions for all $x,y,y'\in \mathcal{C}$, 
\begin{align}\label{cond:lipsch}
    |f(x,y)-f(x,y')| &\leq c_L d(y,y') + \ep_n \enspace ; \\
\label{cond:lipschLower}
    f(x,y')-f(x,y) &\geq c_l \big{(}d(x,y) - d(x,y')\big{)} -\ep_n\quad \text{if }\,d(x,y)\geq d(x,y')\enspace,
\end{align}  
with $\ep_n= c_e \sqrt{\log(n)/n}$.
\end{defi} 

When $c_e=0$, Condition~\eref{cond:lipsch}  enforces Lipschitz continuity and Condition~\eref{cond:lipschLower} enforces 
a minimal decreasing of $f(x,y)$ with $d(x,y)$. In the geometric case $f=g\circ d$ with $g:[0:\pi]\to[0,1]$ continuously differentiable, these conditions hold when $-c_{L}\leq g'(t)\leq -c_{l}$ for all $t\in [0,\pi]$. 
For $c_{e}>0$, the term $\ep_n$ in (\ref{cond:lipsch}--\ref{cond:lipschLower}) can be interpreted as a possible small relaxation of a strict bi-Lipschitz property.
In the remaining of the paper, we will assume that $f\in \mathcal{BL}[c_l,c_L,c_e]$ for some $c_e\geq 0$ and $0 < c_l\leq c_L$.

\paragraph{SubGaussian errors.} We assume that the entries $E_{ij}$ for $1 \leq i < j\leq n$ of the noise matrix are independent and follow a subGaussian(1) distribution. It means that, for any matrix $B\in \R^{n\times n}$ and $t\geq 0$, we have
\begin{equation}\label{defi:subG}
\P\cro{\sum_{1\leq i < j \leq  n} B_{ij} E_{ij} > t \sqrt{\sum_{1\leq i < j \leq  n} B_{ij}^2} } \leq e^{-t^2/2}.
\end{equation}
Since centered random variables taking values in $[-1,1]$ have a subGaussian(1) distribution, this setting encompasses the case where $A\in \{0,1\}^{n\times n}$ is the adjacency matrix of a random graph,  whose distribution belongs to a latent space model on $\Ccal$.

To keep the notation and the presentation simple, we assume  henceforth that the sample size $n$ is a multiple of $4$, 
and we denote by $n_0$ the integer $n_0=n/4$.

\subsection{Identifiability issues}\label{sec:identif}
Our overall goal is to estimate the latent positions $\xbf^*$. Yet, in general, these latent positions are not identifiable from the distribution of $A$. Indeed, for any bijective $\varphi :\Ccal \to \Ccal$, we have $[F_{ij}]_{i<j}=\cro{(f\circ \varphi^{-1})(\varphi(x^*_{i}),\varphi(x^*_{j})}_{i<j}$, with the notation $f\circ \varphi^{-1}(x,y):=f(\varphi^{-1}(x),\varphi^{-1}(y))$. Hence, it is not possible to recover $\xbf^*$ from $F$ and, unless $\xbf^*$ is identifiable from the distribution of the noise $E$, the $n$-tuple of latent positions $\xbf^*$ is not identifiable. Worse, $F$ can be represented by many different couple $(\xbf,f)$. Hence, we face a serious identifiability issue. However, with the premise that the latent positions are well spread on $\Ccal$, we can give a sensible meaning to our estimation objective. We explain progressively the issues that we face, in order to clarify the problem. 

As a warm-up, let us assume in this paragraph that $f$ is known. Even in this favorable case, there might exist some 
bijective $\varphi :\Ccal \to \Ccal$ such that $f=f\circ \varphi$, and hence $f(x^*_{i},x^*_{j})=f(\varphi(x^*_{i}),\varphi(x^*_{j}))$. For example, when $f=g\circ d$, we have $f=f\circ Q$ for any orthogonal transformation $Q\in \Ocal$. In this last case, unless $\xbf^*$ is identifiable from the distribution of the noise $E$, the best that we can hope is to consistently estimate $\xbf^*$ in terms of the quasi-distance
$$\min_{Q\in \Ocal} d_{\infty}(\hat \xbf,Q\xbf^*),\quad \textrm{where }Q\xbf:= (Qx_{1},\ldots,Qx_{n}).$$

Let us come back to our setting where $f$ is unknown.
We define $\Rcal[F,c_{l},c_{L},c_{e}]$ (or simply $\Rcal[F]$) as the set of representations of $F$ by $n$-tuples in $\Ccal^n$ and bi-Lipschitz functions
\begin{equation}\label{defi:R}
\Rcal[F,c_{l},c_{L},c_{e}]:= \ac{(\xbf,f)\in\Ccal^n\times \mathcal{BL}[c_l,c_L,c_e]: f(x_{i},x_{j})=F_{ij}\ \textrm{for all } i<j}.
\end{equation}
We observe that for any $f\in \mathcal{BL}[c_l,c_L,c_e]$ and any $Q\in \Ocal$, we have $f\circ Q^{-1}\in \mathcal{BL}[c_l,c_L,c_e]$. 
Hence, if $(\xbf^*,f)\in \Rcal[F,c_{l},c_{L},c_{e}]$, then $\ac{(Q\xbf^*,f\circ Q^{-1}):Q\in\Ocal}\subset \Rcal[F,c_{l},c_{L},c_{e}]$. 
Are all the elements in $\Rcal[F,c_{l},c_{L},c_{e}]$ of the form $(Q\xbf^*,f\circ Q^{-1})$, as in the case discussed above?

Let us first focus on the case where $(\xbf,f),(\xbf',f')\in \Rcal(F)$ with $\xbf,\xbf'\in \bPi_{n}$ and $c_{e}=0$. The Proposition~\ref{prop:representative} below ensures that there exists $Q\in \Ocal$ such that $\xbf'=Q\xbf$. Hence, if regular representations $(\xbf,f)$ exist in $\Rcal[F]$, all the other regular representations are given by $(Q\xbf^*,f\circ Q^{-1})$ with $Q\in \Ocal$ letting $\bPi_{n}$ invariant. 

This property breaks down when we move away from regular latent positions in $\bPi_{n}$. Indeed, the next proposition shows that we can have $\min_{Q\in \Ocal} d_{\infty}(\xbf,Q\xbf')$ large even, when $\xbf \in \bPi_{n}$ and $\xbf'$ is well spread on $\Ccal$. More precisely, this property is shown for 
$$\xbf'\in \mathcal{S}_{ev}:=\Big\{ {\bf x} \in \Ccal^{n}: \, \sup_{\,z \in \Ccal} \, \min_{\,i\in[ n]} \, d(x_i, z) \leq 3\pi/ n \Big\}.$$
Such a $n$-tuple $\xbf'$ is well spread on $\Ccal$, since any $z\in \Ccal$ is at a distance at most $3\pi/n$ from one of the $x'_i$.
\begin{prop}\label{prop_identif}
Assume that $F_{ij}=f(x_{i},x_{j})$ for $1\leq i<j\leq n$, with 
 $f(x,y)= 1 - d(x,y)/(2\pi)$ and $x_k = e^{ \iota k 2\pi/ n}$ for all $k \in[ n]$. 
 Then, there exists another  representation $(\xbf',f')\in \Rcal[F,(3\pi)^{-1},\pi^{-1},0]$, with $\xbf'\in \mathcal{S}_{ev}$, such that
 $$\min_{Q\in \Ocal} d_{\infty}(\xbf,Q\xbf') \geq \pi /8\enspace .$$
\end{prop}
The representation $(\xbf',f')$ can be obtained from $(\xbf,f)$ by slightly stretching and contracting some pieces of the sphere $\Ccal$. The
 detailed proof of this proposition is postponed to Appendix~\ref{proof:lowerbound}.
This result shows that the set of latent positions $\xbf'$ in $ \Rcal[F,(3\pi)^{-1},\pi^{-1},0]$ is much richer than the set $\ac{Q\xbf:Q\in\Ocal}$ of orthogonal transformations of $\xbf$, since it includes some latent positions $\xbf'$ at constant $d_{\infty}$-distance from this set. Yet, the next proposition shows that for any $\xbf, \xbf'$ in $\Rcal[F]$, the $d_{\infty}$-distance of $\xbf'$ to 
$\ac{Q\xbf:Q\in\Ocal}$ is controlled in terms of the $d_{\infty}$-distance of $\xbf$ and $\xbf'$ to the set of regular positions $\bPi_{n}$.
\begin{prop}\label{prop:representative}
Let $F$ be a symmetric matrix given by $F_{ij}=f(x^*_{i},x^*_{j})$ for $1\leq i < j\leq n$, with $\xbf^*\in\Ccal^n$ and $f\in  \mathcal{BL}[c_l,c_L,c_e]$ for some $c_{e}\geq 0$ and  $0 < c_l\leq c_L$. 
Then, there exists a constant $C_{lLe}>0$, depending only on $c_{e}$, $c_{l}$ and $c_{L}$, and such that, for any $(\xbf,f),(\xbf',f')\in \Rcal[F,c_{l},c_{L},c_{e}]$, we have
\begin{equation}\label{eq:identif:Q}
\min_{Q\in \Ocal} d_{\infty}(\xbf,Q\xbf') \leq C_{lLe} \pa{d_{\infty}(\xbf,\bPi_{n})+d_{\infty}(\xbf',\bPi_{n})+\sqrt{\log(n)\over n}}.
\end{equation}
If, in addition,  $c_{e}=0$ and $\xbf,\xbf' \in\bPi_{n}$, then there exists $Q\in \Ocal$ such that $\xbf=Q\xbf'$.
\end{prop}
The proof of~\eqref{eq:identif:Q} can be found in Appendix~\ref{sec:prop_representative:deuxieme_preuve}, whereas the proof of the second statement can be found in Appendix~\ref{app:central}
. This result shows that if $\xbf$ and $\xbf'$ are close to $\bPi_{n}$, then $\xbf$ is close to an orthogonal transformation of $\xbf'$. Hence, when we restrict to representations $(\xbf,f)$ 
and $(\xbf',f')$, with latent positions $\xbf$ and $\xbf'$ close to $\bPi_{n}$, the identifiability issue becomes smoother.

\noindent Summarizing our discussion above, we have shown that: \\
(1) for any $(\xbf,f)\in \Rcal[F]$, we have $\ac{(Q\xbf,f\circ Q^{-1}):Q\in\Ocal}\subset \Rcal[F]$;\\
(2) for any $(\xbf,f), (\xbf',f')\in \Rcal[F]$, the $d_{\infty}$-distance of $\xbf'$ to $\ac{Q\xbf:Q\in\Ocal}$ is bounded in terms of the $d_{\infty}$-distance of $\xbf$ and $\xbf'$ to $\bPi_{n}$.

 Accordingly,  to contain the phenomenon described in Proposition~\ref{prop_identif}, we will focus henceforth on the representations $(\xbf,f)$ which are the closest to $\bPi_{n}$.

\paragraph{Problem formulation.} 
Let us explain our estimation strategy, in light of the above discussion. 
In order to circumvent the identifiability issues, we focus on the representations   
$(\xbf,f)\in \Rcal[F,c_{l},c_{L},c_{e}]$ whose latent positions are the closest to $\bPi_{n}$, i.e. we focus on the following set\footnote{$\Rcal_{\bPi_{n}}[F]$ is well-defined and non-empty because the set of all $\xbf$ such that, for some $f$, $(\xbf,f) \in \Rcal[F,c_{l},c_{L},c_{e}]$ is compact.
} of representations
\begin{equation}\label{defi:Rpin}
 \Rcal_{\bPi_{n}}[F]= \Rcal_{\bPi_{n}}[F,c_{l},c_{L},c_{e}]:= \argmin_{(\xbf,f)\in \Rcal[F,c_{l},c_{L},c_{e}]} d_{\infty}(\xbf,\bPi_{n}).
 \end{equation}
Our goal is then to build an estimator $\hat\xbf$, \emph{not depending on $c_{l}$, $c_{L}$ and $c_{e}$}, such that, with high-probability
$$d_{\infty}(\hat \xbf,\xbf^*) \leq C_{lLe} \left(\min_{(\xbf,f)\in \Rcal[F,c_{l},c_{L},c_{e}]} d_{\infty}(\xbf,\bPi_{n})+ \sqrt{\log(n)\over n}\right),$$
for some representation $(\xbf^*,f)\in \Rcal_{\bPi_{n}}[F,c_{l},c_{L},c_{e}]$, and some  constant $C_{lLe}>0$ depending only on $c_{e}$, $c_{l}$ and $c_{L}$. 
For such an estimator, under the premise that $d_{\infty}(\xbf^*,\bPi_{n})$ is small for $(\xbf^*,f)\in \Rcal_{\bPi_{n}}[F,c_{l},c_{L},c_{e}]$, we then estimate accurately a representation $(\xbf^*,f)$ of $F$ in $\Rcal_{\bPi_{n}}[F]$.

As a side remark, we notice that combining such a bound with \eref{eq:identif:Q}, we obtain that, for any $(\xbf,f)\in \Rcal_{\bPi_{n}}[F,c_{l},c_{L},c_{e}]$, we have with high-probability
$$ 
\min_{Q\in \Ocal}d_{\infty}(\hat\xbf,Q\xbf) \leq  C'_{lLe} \left(\min_{(\xbf',f')\in \Rcal[F,c_{l},c_{L},c_{e}]} d_{\infty}(\xbf',\bPi_{n})+ \sqrt{\log(n)\over n}\right).$$

\section{Localize-and-Refine algorithm}\label{seriation:section:resul}
The overall strategy for estimating a $n$-tuple of latent positions in  $\Rcal_{\bPi_{n}}[F]$, is to start with a first estimator $\hxbf^{(1)}\in \bPi_{n}$ with a control in  $d_{1}(\xbf,{\bf y})=\sum_{i=1}^n d(x_{i},y_{i})$ distance, and then to refine the estimation  of each point $x^*_{i}$. In order to avoid complex statistical dependencies between the two steps, we use a sample splitting of the data. 
We sample   $S\subset \ac{1,\ldots,n}$ with cardinality $|S|=n_0=n/4$ uniformly at random, and set $\bS= \ac{1,\ldots,n}\setminus S$. The first estimator $\hat\xbf^{(1)}_{S}$ of $\xbf^*_{S}=(x^*_{i})_{i\in S}$ is computed on $A_{SS}=[A_{ij}]_{i,j\in S}$, while the refined estimator $\hat\xbf^{(2)}_{\bS}$ of $\xbf^*_{\bS}:=(x^*_{i})_{i\in \bS}$ takes as input the estimator  $\hat\xbf^{(1)}_{S}$ and the matrix $A_{S\bS}$. This scheme avoids to have some statistical dependence between $\hat\xbf^{(1)}_{S}$ and $A_{S\bS}$. 
The estimator $\hat\xbf^{(2)}_{\bS}$ provides a localization for points indexed by $\bS$, with an error bound in $d_{\infty}$-norm.
In order to localize all the points, we repeat the process and the final estimator is obtained by carefully merging the estimations. 
These three steps are precisely described in Sections~\ref{linftyEstim}--\ref{subsection:LTSalgo}, 
after the statement of our main results.

\subsection{Main result}
In this section, we consider the following setting for the data.

\begin{setting}\label{setting:general}
Assume that the matrix $A$ of observations is given by \eref{eq:model},  with $\xbf^*\in\Ccal^n$, and $f\in  \mathcal{BL}[c_l,c_L,c_e]$, for some $c_{e}\geq 0$ and  $0 < c_l\leq c_L$ (the set of  Bi-Lipschitz functions $\mathcal{BL}[c_l,c_L,c_e]$ is introduced in Definition~\ref{defi:BL}, page~\pageref{defi:BL}). Assume also that the noise matrix $E$ follows the subGaussian errors assumption (\ref{defi:subG}).
\end{setting}

In this setting, the estimator \eref{def:x:gluage} described in the next subsections fulfills the following risk bound.

\begin{thm}\label{thm:principal}
Assume that the data are generated according to the Setting \ref{setting:general} above.
Then, there exists a constant $C_{lLe}>0$ depending only on $c_{e}$, $c_{l}$ and $c_{L}$ such that, with probability at least $1-5/n^2$, there exists a representation $(\xbf,f)$ in the set  $\Rcal_{\bPi_{n}}[F,c_{l},c_{L},c_{e}]$ defined in (\ref{defi:Rpin}) such that the estimator \eref{def:x:gluage} fulfills
\begin{equation}\label{eq:main:thm}
d_{\infty}(\hat \xbf,\xbf) \leq C_{lLe} \left(\min_{(\xbf',f')\in \Rcal[F,c_{l},c_{L},c_{e}]} d_{\infty}(\xbf',\bPi_{n})+ \sqrt{\log(n)\over n}\right)\enspace ,
\end{equation}
with $\Rcal[F,c_{l},c_{L},c_{e}]$ defined by (\ref{defi:R}).
\end{thm}

We emphasize that the estimator \eref{def:x:gluage}  has no tuning parameter. In particular, it does not depend on the unknown constants $c_{e}\geq 0$ and  $0 < c_l\leq c_L$. The first term in the right-hand side of \eref{eq:main:thm} can be assimilated to a bias term, which stems from the bias of the estimator \eref{def:x:gluage} towards regular positions on $\Ccal$. The second term in the right-hand side of \eref{eq:main:thm} is a variance-type term. We observe that, if there exists $c_{a}>0$ such that
\begin{equation}\label{assump1:unif}
\min_{(\xbf',f')\in \Rcal[F,c_{l},c_{L},c_{e}]} d_{\infty}(\xbf',\bPi_{n})\leq c_{a} \sqrt{\frac{\log(n)}{n}}\enspace ,
\end{equation}
then $d_{\infty}(\hat \xbf,\xbf)=O\pa{ \sqrt{{\log(n)}/{n}}}$ with high probability. 
Such a setting arises for example when the latent positions have been sampled uniformly on the sphere, see Corollary~\ref{coro:iid:perf}.
This $\sqrt{{\log(n)}/{n}}$ rate is shown to be minimax optimal in Section~\ref{sec:minimax}. 

We  also give an explicit control for any representative $(\xbf,f)\in$ $\Rcal[F,c_{l},c_{L},c_{e}]$.

\begin{thm}\label{cor:thm:bis}
Under the Setting \ref{setting:general},  there exists a constant $C'_{lLe}>0$ depending only on $c_{e}$, $c_{l}$ and $c_{L}$, such that, with probability at least $1-5/n^2$,  for any $(\xbf,f)\in  \Rcal[F,c_{l},c_{L},c_{e}]$, the estimator \eref{def:x:gluage} fulfills
\begin{equation}\label{eq:main:thm:bis}
\min_{Q\in\Ocal}d_{\infty}(\hat \xbf,Q\xbf) \leq C'_{lLe} \left(d_{\infty}(\xbf,\bPi_{n})+ \sqrt{\log(n)\over n}\right)\enspace .
\end{equation}
\end{thm}

We emphasize that the above statement holds for {\it any} representative $(\xbf,f)\in  \Rcal[F,c_{l},c_{L},c_{e}]$.
Before moving to the description of the estimator \eref{def:x:gluage}, let us give two important instantiations of Theorem~\ref{thm:principal} and \ref{cor:thm:bis}.

\paragraph{Latent model with uniform sampling on $\Ccal$.}
Let us consider the latent model $F_{ij}=f(x^*_{i},x^*_{j})$ when the positions $x^*_1,\ldots,x_{ n}^*$ have been sampled independently and uniformly on $\mathcal{C}$, as in the graphon model. In this case, the Assumption~\eqref{assump1:unif} is satisfied with high probability; see  Appendix~\ref{append:coro} for a proof. We then derive  the next result from Theorem~\ref{cor:thm:bis}.

\begin{cor}\label{coro:iid:perf} 
Assume that  the latent positions $x^*_1,\ldots,x^*_{ n}$ have been sampled i.i.d.\ uniformly  on $\mathcal{C}$ and that $f\in \mathcal{BL}[c_l,c_L,c_e]$.
Then, with probability higher than $1-7/n^{2}$, we have
 \[
  \min_{Q\in\Ocal} d_{\infty}(\hat{{\bf x}},Q{\bf x}^*) \leq C_{lLe} \,  \sqrt{\frac{\log(n)}{n}}\enspace,
 \]
 for some constant $C_{lLe}>0$ depending only on $c_{e}$, $c_{l}$ and $c_{L}$.
\end{cor}

\paragraph{Toroidal seriation.}
Let us consider the toroidal seriation problem introduced in Example 4 of Section~\ref{sect11:intro}.  In this setting,  the set $[n]$ is considered as a torus, endowed with the torus distance $d(i,j)=\min(|j-i|, |n+i-j|)$ for any $1\leq i \leq j\leq n$, and the matrix $F$ is a pre-toroidal R-matrix. Let  $\si^*\in \Sigma_{ n}$ be a permutation such that $[F_{\sigma^*(i)\sigma^*(j)}]_{i,j}$ is a toroidal R-matrix. 
Our goal is estimate $\si^*$ from the noisy observation $A=F+E$. 
As explained in the introduction, we can recast this problem as a localization problem in a latent space model on the regular grid $\Ccal_{n}$. 
Assimilating points on the sphere $\Ccal$ to unit norm complex numbers, we define the vector ${\bf x}^*\in \bPi_n$ by $x^*_j= \exp(\iota 2\pi \sigma^*(j)/n )$,  and we define $f_{n}:\Ccal_{n}\times \Ccal_{n}\to \R$ by $f(x^*_i,x^*_j)=F_{ij}$. The problem of estimating $\sigma^*$ then amounts to estimating $\xbf^*$ in the latent space model $A_{ij}=f(x^*_i,x^*_j)+E_{ij}$ on $\Ccal_{n}$. We can apply our estimator~\eref{def:x:gluage} and get an estimation
$\hat{\bf x}\in \Ccal^n$. From this estimation, we can derive the map  $\hat{\sigma}: [ n]\rightarrow [ n]$ by setting $\hat{\sigma}_i = \lceil n\underline{\hat{x}}_i/2\pi \rceil$ for $i=1,\ldots, n$, where  $\underline{\hat{x}}_i\in (0,2\pi]$ is the argument of $\hat{x}_i$ and $\lceil z\rceil$ is the upper integer part of $z$. 
While the map $\hat\sigma$ may not be a permutation in $\Sigma_{n}$, it is an estimation of $\sigma^*$ and we can translate Theorem~\ref{cor:thm:bis} into a $\ell^{\infty}$-error between the two. 

\begin{cor}\label{cor:seriation}
Assume that $[F_{\sigma^*(i)\sigma^*(j)}]_{i,j=1,\ldots,n}$ fulfills the  bi-Lipschitz condition with respect to the torus distance:
\begin{align*}
    |F_{\sigma^*(i)\sigma^*(j)}-F_{\sigma^*(i)\sigma^*(k)}| &\leq c_L \frac{2\pi}{n}d(j,k) + \ep_n \\
   F_{\sigma^*(i)\sigma^*(k)}-F_{\sigma^*(i)\sigma^*(j)} &\geq c_l \frac{2\pi}{n} \big{(}d(i,j) - d(i,k)\big{)} -\ep_n\quad \text{if }\,d(i,j)\geq d(i,k)\enspace . 
\end{align*}
Then, there exists a constant $C_{lLe}>0$, depending only on $c_{e}$, $c_{l}$ and $c_{L}$, such that, with probability at least 
$1-5/n^2$, we have
\begin{equation}\label{borne:seriation}
    \min_{\tau \in \Gamma_n} \ \max_{i \in [ n]}\ |\tau \circ \si^*(i)-\hat{\si}(i)| \, \leq C_{lLe} \sqrt{n\log(n)}\enspace ,
\end{equation}
where $\Gamma_n$ is the subgroup of permutations of $\ac{1,\ldots,n}$ generated by the circular and reverse permutations.
\end{cor}

To prove \eref{borne:seriation}, we extend $f_{n}$
 defined on $\Ccal_n \times \Ccal_n$  to  $f:\mathcal{C}\times \mathcal{C}\to \R$ belonging to $\mathcal{BL}[c_l,c_L,c_e]$,
 and apply Theorem~\ref{cor:thm:bis}.  The minimum over $\Gamma_n$ in the left-hand side of \eref{borne:seriation} cannot be avoided, since $\sigma^*$ is identifiable from $F$ only up to permutations in $\Gamma_n$.   Furthermore, the $\sqrt{n\log(n)}$ rate for the toroidal seriation problem can be shown to be minimax in the above set-up, by combining Theorem~\ref{thm:lowerBound} page \pageref{thm:lowerBound} and the correspondence between the $n$-tuples ${\bf x}^* \in \bPi_n$ and the permutations $\si^*$ of $[n]$.

In a recent work, Janssen and Smith~\cite{janssen2020reconstruction} consider the related seriation problem for R-matrices (Example 3 in the introduction),
in a geometric setting where $F_{\sigma^*_i,\sigma^*_j}= g(|i-j|/n)$ for some unknown permutation  $\sigma^*$, and some unknown function $g$. Hence, in addition to be a pre-R matrix, $F$  is also a Toeplitz matrix. Under additional assumptions on the squared matrix $(F_{\sigma^*_i,\sigma^*_j})^2$, they establish that an algorithm based on a (thresholded version) of the square matrix of observations  $A^2$, achieves, with high probability, an error bound 
$$ \min_{\tau\in \Gamma'_{n}} \ \ \max_{i\in[n]} \, |\tau \circ \si^*(i) - \hat{\si}(i)| \lesssim \sqrt{n} \,(\log(n))^5,$$ 
where $\Gamma_{n}'$ gathers the identity and the reverse permutations. Their assumptions are not comparable to ours, but their rates are similar (up to log factors). Our results then complement this work, by providing another set of conditions on $F$, under which the hidden permutation $\sigma^*$ can be recovered at the rate $\sqrt{n\log(n)}$.

\paragraph{Organization of Section \ref{seriation:section:resul}.}

The description of the estimator \eref{def:x:gluage} is organized as follows. The refinement step $\hat\xbf^{(2)}_{\bS}$ is described in  Section~\ref{linftyEstim}. This step can take as input any initial estimator $\hxbf^{(1)}_{S}$  based on $A_{SS}$ and taking values in $\bPi_{|S|}$. When this initial estimator fulfills with high-probability
\begin{equation}\label{eq:hx1:expl}
d_{1}(\hxbf^{(1)}_{S},\xbf_{S}) \leq C_{lLe} \pa{n\min_{(\xbf',f')\in \Rcal[F]}d_{\infty}(\xbf'_{S},\bPi_{|S|})+\sqrt{n\log(n)}},
\end{equation}
for some $(\xbf,f)\in \Rcal_{\bPi_{n}}[F]$, 
then the refined estimator $\hat\xbf^{(2)}_{\bS}$ is shown to fulfill with high-probability
$$d_{\infty}(\hxbf^{(2)}_{\bS},\xbf_{\bS}) \leq C'_{lLe} \pa{\min_{(\xbf',f')\in \Rcal[F]}d_{\infty}(\xbf'_{S},\bPi_{|S|})+\sqrt{\log(n)\over n}}.$$
In order to get an estimator satisfying the risk bound \eref{eq:main:thm}, we then need an initial estimator $\hxbf^{(1)}_{S}$ fulfilling \eref{eq:hx1:expl}. Such an estimator is provided in Section~\ref{estimL1}. A computationally efficient alternative, based on the spectral decomposition of $A$ is proposed in Section~\ref{sec:geometric}. To get an estimator of the whole $n$-tuple of latent positions, the data splitting is repeated and a final merging step is needed to build $\hxbf$. This final step is described in
Section~\ref{subsection:LTSalgo}.

\subsection{Step 2: refined estimation}\label{linftyEstim}
We start by describing the refinement step which converts an initial estimator with an error bound in $d_{1}$-distance into a refined  estimator with an error bound in $d_{\infty}$-distance. For a subset $S\subset \ac{1,\ldots,n}$ of cardinality $|S|=n_0=n/4$, the refinement step takes as input any initial estimator $\hxbf^{(1)}_{S}$ of $\xbf^*_{S}$  based on $A_{SS}$ and taking values in $\bPi_{n_0}$. It outputs an estimator $\hxbf^{(2)}_{\bS}$ of $\xbf^*_{\bS}$. 

We denote by $D(z, \hxbf^{(1)}_S)= [d(z , \hat x^{(1)}_j)]_{j\in S}$  the vector of distances between $z\in \Ccal$ and the components of the $n_0$-tuple $\hxbf^{(1)}_S$.
The refined estimator $\hat\xbf^{(2)}_{\bS}$ is obtained by solving 
\begin{equation}\label{def:estim:Linfini}
    \hat{x}^{(2)}_i \in \argmin_{z \in \Ccal_{n_0}} \,  \big\langle A_{i , S}, D(z, \hxbf^{(1)}_S)\big\rangle,\quad\textrm{for each}\  i\in\bS,
\end{equation}
where $\Ccal_{n_0}=\{1, e^{\iota 2\pi/n_0},\ldots, e^{\iota 2\pi(n_0-1)/n_0}\}$ is the regular grid of cardinality $n_0$ on $\Ccal$. The principle underlying the definition \eref{def:estim:Linfini} is that $d(\hat{x}_i^{(2)}, \hat{x}^{(1)}_j)$ should be small when $A_{ij}$ is large, and vice-versa. Hence, for any $x_{i}^*\in \Ccal_{n_0}$ and any matrix $A$ with $A_{ij}$ decreasing with $d(x^*_{i},x_{j}^*)$, the minimum is achieved in $x_{i}^*$ when $\hat \xbf_{S}^{(1)}=\xbf^*_{S}\in \bPi_{n_0}$, see Appendix~\ref{app:central} for details.  
Since $A$ is a noisy version of a such a matrix, and since $d(x_{i}^*,\Ccal_{n_0}):=\min_{y\in \Ccal_{n_0}} d(x_{i}^*,y) =O(1/n)$,  the estimator $ \hat{x}^{(2)}_i$ should remain close to $x^*_{i}$ when $\hat \xbf_{S}^{(1)}$  is close to $\xbf^*_{S}$.
The next proposition quantifies this statement, by providing a uniform error bound for $\hat{{\bf x}}^{(2)}_{\bS}$ in terms of the  error bound  $d_{1}(\hxbf^{(1)}_S, Q{\bf x}^*_{S})$ for the initial estimator.

\begin{prop}\label{prop:LHO:perf} 
Let $(\xbf^*,f)\in  \Rcal[F,c_{l},c_{L},c_{e}]$ for some $c_{e}\geq 0$ and  $0 < c_l\leq c_L$.
 Let $S\subset \ac{1,\ldots,n}$ be of cardinality $|S|=n_0=n/4$, and  $\hxbf^{(1)}_{S}$ be any initial estimator of $\xbf^*_{S}$  based on $A_{SS}$ and taking values in $\bPi_{n_0}$. Then, there exists a constant $C_{lLe}$ depending only on $c_{l}$, $c_{L}$ and $c_{e}$, such that, conditionally on $A_{SS}$, and for all $Q\in \Ocal$,  the estimator \eref{def:estim:Linfini}  fulfills 
 with probability at least $1-1/n^2$
\begin{equation}\label{eq:upper_d_infty}
  d_{\infty}(\hat{{\bf x}}^{(2)}_{\overline{S}},  Q{\bf x}^*_{\overline{S}})\leq C_{lLe}\pa{ d_{\infty}({\bf x}^*_S,\bPi_{n_0}) + \frac{  d_{1}(\hxbf^{(1)}_S, Q{\bf x}^*_{S})}{n} + \sqrt{\frac{\log(n)}{n}}}\enspace .  
\end{equation}
\end{prop}

The right-hand side of~\eqref{eq:upper_d_infty} is made of three terms. The first one is the uniform approximation error of ${\bf x}^*_{S}$ by $\bPi_{n_0}$, as defined in \eqref{alpha:def:new}. It is a bias-type term  which stems from the fact that we aim at estimating positions that are almost evenly spaced.   The second term accounts for the error of the preliminary estimator ${\bf x}^{(1)}_{S}$ in $d_1$-distance and the last-one is a variance-type term.  The proof of Proposition~\ref{prop:LHO:perf} can be found in Appendix~\ref{sect:proof:lem:perf:LHO:specificase}.

The time complexity for computing $\hat{x}^{(2)}_i$ is linear in $n_0$ and the algorithm can be parallelized for computing $\hxbf^{(2)}_{\bS}$. To decrease the time complexity, it is possible to restrict to $z\in \Ccal_{\sqrt{n_0}}$ in \eqref{def:estim:Linfini}  instead of $z\in \Ccal_{n_0}$. In that case, the proof of~\eqref{eq:upper_d_infty} still holds, with different constants.

\subsection{Step 1: initial  localization} \label{estimL1}
In view of the above Proposition \ref{prop:LHO:perf}, we seek to build  an initial estimator $\hat{{\bf x}}^{(1)}_S$ fulfilling 
\eref{eq:hx1:expl} for some $\xbf_{S}=Q\xbf_{S}^*$, with $Q\in\Ocal$. Such an estimator can be obtained by solving
\begin{equation}\label{def:estim:l1bound}
    \hat{{\bf x}}^{(1)}_S \in  \argmin_{{\bf x}_{S} \in \bPi_{n_0}} \,  \langle A_{SS}, D({\bf x}_{S})\rangle,\quad \textrm{with}\ D({\bf x}_{S})= \big{[} d(x_i , x_j) \big{]}_{i,j \in S}\enspace .
\end{equation}
The estimator $\hat{{\bf x}}^{(1)}_S$ is chosen in such a way that the distance $d(\hat{x}^{(1)}_i,\hat{x}^{(1)}_j)$ should be small when the signal $F_{ij} = f(x_i^*,x_j^*)$ is large,  and conversely, it should be large when $F_{ij}$ is  small. 
To grasp the principle underlying the definition  \eqref{def:estim:l1bound}, let us look at the noiseless geometric affine  setting, where the observations are $A_{ij}=1-\alpha d(x^*_{i},x^*_{j})$, and where the positions ${\bf x}^*_{S}$ are evenly spread, i.e.\ ${\bf x}^*_{S}\in \bPi_{n_0}$. 
Then, one readily checks that 
\[\argmin_{{\bf x}_{S} \in \bPi_{n_0}} \,  \langle A_{S S}, D({\bf x}_{S})\rangle = \underset{{\bf x}_{S} \in \bPi_{n_0}}{\textup{argmax}} \,  \langle D({\bf x}^*_{S}), D({\bf x}_{S})\rangle \enspace ,\]
whose  maximum is achieved at all ${\bf x}_{S} = Q{\bf x}^{*}_{S}$ with $Q$  any  orthogonal transformation preserving $\bPi_{n_0}$.
In other words, the estimator~\eqref{def:estim:l1bound} exactly recovers --up to distance preserving transformations-- the positions in this ideal setting.

The next proposition establishes a $d_{1}$-bound with the flavor of \eref{eq:hx1:expl} for the estimator~\eqref{def:estim:l1bound}. 

\begin{prop}\label{l1bound} Let $(\xbf^*,f)\in \Rcal[F,c_l,c_L,c_e]$ and let $S\subset\ac{1,\ldots,n}$ be a subset of cardinality $n_0=n/4$.
Then, there exists a constant $C_{lLe}$ depending only on $c_{l}$, $c_{L}$ and $c_{e}$, such that, the estimator  $\hat{{\bf x}}^{(1)}_S$ defined by \eqref{def:estim:l1bound} satisfies 
\begin{equation}\label{eq:thm:L1}
 \min_{Q\in \Ocal} d_{1}(\hat{{\bf x}}^{(1)}_S, Q{\bf x}^*_{S}) \leq C_{lLe} \pa{n  d_{\infty}({\bf x}^*_{S},\bPi_{n_0}) + \sqrt{n \log{n}}}\enspace ,
 \end{equation} 
with probability higher than $1-1/n^2$.
\end{prop} 

To prove Proposition \ref{l1bound} (in Appendix \ref{sec:proof:prop:l1bound:noconstr}), we first establish  that $\| D({\bf x}^*_{S}) -D(\hat{{\bf x}}^{(1)}_S)\|_2$ is small, meaning that the distances between the estimated positions $\big\{\hat{ x}_i^{(1)}:i\in S\big\}$ are close to the distances between the true positions $\ac{x^*_{i}:i\in S}$. Then, relying on a recent result on matrix perturbation from~\cite{arias2020perturbation}, we deduce that $\|\hat{{\bf x}}^{(1)}_S-Q{\bf x}^*_{S}\|_2$ is small, where $Q \in \Ocal$ is a distance preserving transformation and where we consider here $\hat{{\bf x}}^{(1)}_S$ and $Q{\bf x}^*_{S}$ as $2\times n_0$ matrices. The bound~\eref{eq:thm:L1} then follows by connecting the Euclidean distance in $\mathbb{R}^2$ to 
the $d_{1}$-distance. 

From a computational point of view, the minimization problem~\eqref{def:estim:l1bound} is an instance of the Quadratic Assignment Problem which is known to be NP-Hard and even hard to approximate~\cite{QAP1,QAP2}. In section~\ref{sec:geometric}, we propose a computationally  efficient alternative to  \eqref{def:estim:l1bound}, and we provide theoretical guarantees for this alternative under additional model assumptions.

\subsection{Final merging step} \label{subsection:LTSalgo}
For a given subset $S\subset \ac{1,\ldots,n}$ of cardinality $n_0=n/4$, combining the initial estimator~\eref{def:estim:l1bound} with the refined localisation~\eref{def:estim:Linfini}, we get an estimator $\hxbf^{(2)}_{\bS}$ with an error bound on 
$d_{\infty}(\hxbf^{(2)}_{\bS},Q \xbf^*_{\bS})$ for some orthogonal transformation $Q\in \Ocal$. 
In order to get an estimation of all the latent positions, we repeat the process by sampling $S'\subset \bS$ of cardinality $n_0=n/4$ and by computing 
 $\hat\xbf^{(2')}_{\bS'}$ with~\eref{def:estim:l1bound} and~\eref{def:estim:Linfini}. We then get an estimator   with an error bound on 
$d_{\infty}(\hat\xbf^{(2')}_{\bS'},Q' \xbf^*_{\bS'})$ for some orthogonal transformation $Q'\in \Ocal$. 
In order to get a final estimator $\hxbf$, we still have to deal with the fact that we may have $Q\neq Q'$, and hence the trivial merge $\hxbf=(\hxbf^{(2)}_{\bS},\hat\xbf^{(2')}_{S})$ may not be a good one. Hence, we need to synchronize the estimators $\hxbf^{(2)}_{\bS}$ and $\hat\xbf^{(2')}_{\bS'}$. This synchronization is obtained by solving
\begin{equation}\label{eq:merge}
\widehat{Q} \in \argmin_{Q \in \Ocal} \, d_{\infty}(\hat{{\bf x}}^{(2)}_{\overline{S} \cap \overline{S}'}, Q\hat{{\bf x}}^{(2')}_{\overline{S} \cap \overline{S}'}),
\end{equation}
and by defining the final estimator as $\hat{{\bf x}}=(\hat{{\bf x}}^{(2)}_{\bS}, \widehat{Q}\hat{{\bf x}}^{(2')}_{S})$.
Putting pieces together, we then get the following estimation procedure.

\noindent { \def\arraystretch{1.3}
\begin{tabular}{|l|}
\hline
{\bf Localize-and-Refine procedure} \\
\hline
\begin{minipage}{0.95\textwidth} \centering
\begin{minipage}{0.9\textwidth}

\medskip
\underline{Input:} Observations matrix $A$. 

\medskip
\noindent {\bf A) Localization of $3n/4$ points} 
\begin{enumerate}[topsep=3pt]
    \item Pick uniformly at random a subset $S \subset [n]$ of cardinality $|S|= n/4$.
\item Compute $\hxbf^{(2)}_{\bS}$ by solving
\begin{eqnarray*}
\hat{{\bf x}}^{(1)}_S &\in&  \argmin_{{\bf x}_{S} \in \bPi_{n_0}} \langle A_{SS}, D({\bf x}_{S})\rangle \,,\\
\hat{x}^{(2)}_i &\in& \argmin_{z \in \Ccal_{n_0}}  \big\langle A_{i , S}, D(z, \hxbf^{(1)}_S)\big\rangle,\ \ \textrm{for}\  i\in\bS\,,
\end{eqnarray*}
with $D({\bf x}_{S})= \big{[} d(x_i , x_j) \big{]}_{i,j \in S}$ and \ \ $D(z,\xbf_{S})= \big{[} d(z , x_j) \big{]}_{j \in S}$\,.
\end{enumerate}


\noindent {\bf B) Localization of $3n/4$ (other) points}

\begin{enumerate}[topsep=3pt]
 \item Pick uniformly at random a subset $S' \subset \overline{S}$ of size $|S'|= n/4 $.
\item 
Compute $\hat\xbf^{(2')}_{\bS'}$ by solving
\begin{eqnarray*}
\hat{{\bf x}}^{(1')}_{S'} &\in&  \argmin_{{\bf x}_{S'} \in \bPi_{n_0}}  \langle A_{S'S'}, D({\bf x}_{S'})\rangle,\\
\hat{x}^{(2')}_i &\in& \argmin_{z \in \Ccal_{n_0}}  \big\langle A_{i , S'}, D(z, \hat\xbf^{(1')}_{S'})\big\rangle,\ \ \textrm{for}\  i\in\bS'\,,
\end{eqnarray*}
with $D({\bf x}_{S'})= \big{[} d(x_i , x_j) \big{]}_{i,j \in S'}$ and $D(z,\xbf_{S'})= \big{[} d(z , x_j) \big{]}_{j \in S'}$.
\end{enumerate}


\noindent {\bf C) Merging the two localizations}

\begin{enumerate}[topsep=3pt]
\item Compute $\widehat{Q}$ by solving \eref{eq:merge}.
\item \underline{Output:} $\hat{{\bf x}}\in\Ccal^n$ defined by
\begin{equation}\label{def:x:gluage}
\hat{x}_j = \left\{
    \begin{array}{ll}
       \hat{x}^{(2)}_j \ \, & \mbox{ if } j \in \overline{S} \enspace, \\
       \widehat{Q}\hat{x}^{(2')}_j \ \, & \mbox{ if } j \in S \enspace. 
    \end{array}
\right.\end{equation}
\end{enumerate}
\end{minipage}%
\end{minipage} \vspace{3pt}\\
\hline
\end{tabular} }
\medskip

The minimization problem \eref{eq:merge} can be solved efficiently. For example, we can observe  that the minimum is achieved at some $\hat Q\in\Ocal$ preserving $\Ccal_{n_0}$, and since there are at most $2n_0$  such orthogonal transformations, we can enumerate them. We refer to Section~\ref{sec:numerique} for details.







\section{Spectral Localization in the geometric latent model}\label{sec:geometric}
The computation of the initial localization $\hxbf^{(1)}_{S}$ requires to minimize
 \eqref{def:estim:l1bound} over $\bPi_{n_0}$, which is an instance of the Quadratic Assignment Problem (QAP), which is known to be NP-hard and hard to approximate~\cite{QAP1,QAP2}. 
A spectral relaxation of the QAP has been shown to be successful for reordering a pre (toroidal) R-matrix~\cite{atkins1998spectral, recanati2018reconstructing}, and hence for solving the noiseless seriation problem for R-matrices. 
This vanilla spectral algorithm proposed in~\cite{atkins1998spectral}
takes as input any symmetric matrix $M\in\R^{N\times N}$ and output $N$ points in $\R^2$. \medskip

\noindent { \def\arraystretch{1.3}
\begin{tabular}{|l|}
\hline
{\bf Vanilla Spectral Algorithm (VSA)} \label{alg:spec:geo} \\
\hline
\begin{minipage}{0.95\textwidth} \centering
\begin{minipage}{0.9\textwidth}

\medskip
\noindent\underline{Input:} symmetric matrix $M\in\R^{N\times N}$ with eigenvalues $\hat{\lambda}_{0}\geq \ldots \geq \hat{\lambda}_{N-1}$.

\medskip

\noindent\underline{Compute:}  two orthonormal eigenvectors $\hat u,\hat v\in\R^{N}$ associated with the second and third eigenvalues $\hat{\lambda}_{1}$ and $\hat{\lambda}_{2}$  of $M$
\medskip

\noindent\underline{Output:} 
\begin{equation}\label{eq:VSA}
\hxsp=\pa{\sqrt{\frac{N}{2}}(\hat{u}_1, \hat{v}_1)^T,\ldots,\sqrt{\frac{N}{2}}(\hat{u}_N, \hat{v}_N)^T} \in \R^{2\times N}
\end{equation}
\medskip

\end{minipage}%
\end{minipage} \\
\hline
\end{tabular} }

\medskip

In this section, we adapt this vanilla spectral algorithm in order to get a computationally efficient initial estimator $\tilde\xbf^{(1)}_{S}$. 
In Section~\ref{sec:init:spectral}, we describe the estimator $\tilde\xbf^{(1)}_{S}$ and provide some error bounds in $d_{\infty}$-distance for the Localize-and-Refine algorithm based on $\tilde\xbf^{(1)}_{S}$. The main difference compared to Section~\ref{seriation:section:resul} is that our theory is limited to cases where the  function 
$f$ is \textit{geometric}, that is, 
\begin{equation}\label{def:geometricSetting}
 f(x,y) = g(d(x,y)) \quad \textrm{for all}\ x,y \in \Ccal,
\end{equation}
for some  $g: [0,\pi] \rightarrow [0,1]$. 
In Section~\ref{sec:vanilla:spectral}, we complement this result by providing an error bound in $\ell^1$-norm for the vanilla spectral algorithm (VSA) in the geometric case.

\subsection{Spectral Localization algorithm}\label{sec:init:spectral}

We observe that the output $\hxsp_{S}$ of the vanilla spectral algorithm applied to $A_{SS}$ does not belong to $\bPi_{n_0}$, and even not to $\Ccal^{n_0}$. Hence, we need an additional approximation step in order to get an estimator  $\tilde\xbf^{(1)}_{S}$ that can be plugged in our Localize-and-Refine procedure \eref{def:estim:Linfini}. In the description of the algorithm below, we identify points on the circle $\Ccal$ to unit norm complex numbers. Besides, for such a point $z$, we write $\|z\|_1$ for its $\ell^1$ norm in $\mathbb{R}^2$.

\medskip

\noindent { \def\arraystretch{1.3}
\begin{tabular}{|l|}
\hline
{\bf   Spectral Localization (LS)} \label{alg:LHO:geo} \\
\hline
\begin{minipage}{0.95\textwidth} \centering
\begin{minipage}{0.9\textwidth}

\medskip

 \underline{Input:} a subset $S=\ac{i_{1},\ldots,i_{n_0}}\subset[n]$ with $i_{1}<\ldots<i_{n_0}$, and data matrix $A_{SS}$. 

\medskip
\noindent \underline{Vanilla spectral localization:} compute
$\hxsp_{S}:= [\hat{x}^{\textup{VSA}}_{i_{\ell}}]_{l\in[n_0]} =\textrm{VSA}(A_{SS})$

\medskip 

\noindent \underline{Uniform Approximation (UA) in $\bPi_{n_0}$:}  \label{alg:findDiscAprox}

\begin{enumerate}
\item Set $z_{\ell}=\hat{x}^{\textup{VSA}}_{i_{\ell}}/\|\hat{x}^{\textup{VSA}}_{i_{\ell}}\|_2$, for $\ell=1,\ldots,n_0$.  

\item Pick any permutation $\si$ such that $z_{\sigma(1)},\ldots,z_{\sigma(n_0)}$ is in trigonometric order.

\item Set $\displaystyle{\tilde x^{(1)}_{i_{\sigma(\ell)}}=e^{\iota {2\pi (\widehat k+\ell)\over n_0}},\quad \textrm{where}\quad 
\widehat k\in\argmin_{k\in [n_0]}\sum_{\ell=1}^{n_0} \left\|e^{\iota {2\pi (k+\ell)\over n_0}}-z_{\si(\ell)}\right\|_1}.$

\end{enumerate}

\noindent \underline{Output:} $\tilde\xbf^{(1)}_{S}:=[\tilde x^{(1)}_{i_{\ell}}]_{l\in[n_0]} \in \bPi_{n_0}$.
\end{minipage}%
\end{minipage} \\
\hline
\end{tabular} }

 \medskip 
 
 The next theorem provides  an error bound in $d_{\infty}$-distance for the Localize-and-Refine procedure~\eref{def:x:gluage}, when we replace $\hxbf^{(1)}_{S}$ by $\tilde\xbf^{(1)}_{S}$. This bound involves the two spectral gaps $\Delta_1 = \lambda_{0}^* -\lambda_{1}^*$ and $\Delta_2= \lambda_{2}^* - \lambda_{3}^*$, where  $\lambda_{0}^*\geq \ldots \geq \lambda_{n-1}^*$ denote the eigenvalues of the signal matrix $F$.

 \begin{thm}\label{thm:graphgeo:new} Let $n\geq 16$,  $c_a,c_{b},c_{e} >0$, and $0<c_{l}\leq c_{L}$. Let $(\textup{\textbf{x}}^*,f)\in \Rcal[F,c_l,c_L,c_e]$ with $f$ a geometric function $f=g\circ d$. Assume that $\xbf^*$  fulfills 
 \begin{equation}\label{GraphGeo:assump:latentPosition}
 d_{\infty}(\xbf^*,\bPi_{n}) \leq c_a  \sqrt{\frac{\log(n)}{n}}\enspace ,
\end{equation}
and that  the spectral gaps satisfy $\Delta_1 \wedge \Delta_2\geq c_b n$.
Then, 
there exists a constant $C_{lLeab}>0$ depending only on $c_{e}$, $c_{l}$, $c_{L}$, $c_{a}$ and $c_{b}$, such that, with probability at least $1-9/n^2$, the spectral Localize-and-Refine procedure~\eref{def:x:gluage}  with $\hxbf^{(1)}_{S}$ and $\hxbf^{(1')}_{S'}$ replaced by $\tilde\xbf^{(1)}_{S}$ and $\tilde\xbf^{(1')}_{S'}$
satisfies the uniform bound 
\begin{equation*}
    \min_{Q\in\Ocal} d_{\infty}(\hat{{\bf x}},Q{\bf x}^*) \leq C_{lLeab} \,  \sqrt{\frac{\log(n)}{n}} \enspace .
\end{equation*}
\end{thm} 

Similarly as in Theorem~\ref{thm:principal} and \ref{cor:thm:bis}, we estimate the latent positions at the optimal $\sqrt{\log(n)/n}$ rate in $d_{\infty}$-distance, but under the additional assumptions that $f$ is a geometric function~\eref{def:geometricSetting} and $F$ fulfills the spectral gap condition  $\Delta_1 \wedge \Delta_2\geq c_b n$. The proof of Theorem~\ref{thm:graphgeo:new} is given in Appendix~\ref{sec:appendix:spectral}.
The proof mainly relies on controlling the $\ell^1$-norm between $\hxsp$ and $Q\textup{\textbf{x}}^*$. This result, which has its own interest, is presented in Proposition~\ref{conj:graphgeo:l1}, in Section~\ref{sec:vanilla:spectral}. 
Below, we exhibit  two cases where the spectral gap condition $\Delta_1 \wedge \Delta_2\geq c_b n$ holds.

\paragraph{Example: geometric model with Fourier gaps.} 
The eigenvalues of $F$ are closely related to the discrete Fourier transform of $g$, so that we can bound the spectral gaps $\Delta_1$ and $\Delta_2$ in terms of these Fourier coefficients. 
More precisely, the function $f$ is given by $f(x,y)= g(d(x,y))$, with $g$ defined on $[0,\pi]$. One can  extend $g$ to $[0,2\pi)$ by taking $g(x)= g(2\pi-x)$ for any  $x\in(\pi, 2\pi)$. 
Then, for any integer $n$, the discrete Fourier transform of $\ac{g(j \frac{2\pi}{n}):j=0,\ldots, n-1}$ is defined by    
\begin{equation}\label{spectre:fourier:paper}
\Fcal_{k,n}(g) =  \sum_{j=0}^{n-1} g\left(j \frac{2\pi}{n} \right) \cos\left(j \frac{2\pi k}{n}\right),\quad \textrm{for}\ k=0,\ldots, n-1.
\end{equation}
The following lemma bounds the spectral gaps $\Delta_1$ and $\Delta_2$ in terms of the gaps between the Fourier coefficients.

\begin{lem}\label{lem:fourier}  
 Let $c_a,c_{b},c_{e} >0$, and $0<c_{l}\leq c_{L}$. Let $(\textup{\textbf{x}}^*,f)\in \Rcal[F,c_l,c_L,c_e]$ with $f$ a geometric function $f=g\circ d$, and $\xbf^*$  fulfilling \eref{GraphGeo:assump:latentPosition}.

Let us set $\Phi_{1} = \Fcal_{0,n}(g) - \Fcal_{1,n}(g)$ and 
$\Phi_{2} =  \min_{j =2,\ldots,\lfloor n/2\rfloor} \ \, \Fcal_{1,n}(g) - \Fcal_{j,n}(g)$. Then, there exists a constant  $C_{lLea}>0$, depending only on $c_{e}$, $c_{l}$, $c_{L}$ and $c_{a}$, such that
\[
\left| \Delta_1 -\Phi_{1} \right| \vee  \left| \Delta_2 -\Phi_{2} \right| \, \leq C_{lLea} \sqrt{n \log(n)}\enspace .
\]
\end{lem}
Hence,  Theorem~\ref{thm:graphgeo:new} still holds when we replace the gap condition $\Delta_1 \wedge \Delta_2\geq c_b n$  by 
the condition $\Phi_1 \wedge \Phi_2\geq c_b n$. 
So, when the first discrete Fourier coefficients of $g$ are well separated from the other coefficients, the spectral version of the Localize-and-Refine algorithm estimates, in polynomial time, the latent positions at the optimal $\sqrt{\log(n)/n}$ rate in $d_{\infty}$-distance. 
Below, we give an  example where the Fourier coefficients can be explicitly computed and where $\Phi_{1},\Phi_{2}$ are proportional to $n$.

\paragraph{Example: affine geometric model.}\label{sec:example:spectral}

As a simple instantiation of Theorem~\ref{thm:graphgeo:new} and Lemma~\ref{lem:fourier}, let us consider the geometric function $f(x,y)= 1- d(x,y)/(2\pi)$. The corresponding univariate function $g(z)=1-z/(2\pi)$ is affine and its discrete Fourier coefficients can be computed explicitly in terms of trigonometric functions. In  Appendix \ref{proof:specGap}, we prove that  $\Phi_1 \wedge \Phi_2\geq c_b n$ for some numerical constant $c_{b}>0$. We then get the next corollary of Theorem \ref{thm:graphgeo:new}.

\begin{cor}\label{coro:gapspectral}
Let $f$ be the function defined as $f(x,y)=1- d(x,y)/(2\pi)$. Assume that the latent positions ${\bf x}^*\in \Ccal^{n}$ fulfill \eqref{GraphGeo:assump:latentPosition}. Then, there exist constants $C_{a}$ and $C'_a$ depending only on $c_{a}$ such that for all $n\geq C'_a$, with probability higher than $1-9/n^2$,  the spectral Localize-and-Refine procedure~\eref{def:x:gluage}, with $\hxbf^{(1)}_{S}$ and $\hxbf^{(1')}_{S'}$ replaced by $\tilde\xbf^{(1)}_{S}$ and $\tilde\xbf^{(1')}_{S'}$,
satisfies the uniform bound 
\begin{equation*}
    \min_{Q\in\Ocal} d_{\infty}(\hat{{\bf x}},Q{\bf x}^*) \leq C_{a} \,  \sqrt{\frac{\log(n)}{n}} \enspace .
\end{equation*}
\end{cor}

Theorem~\ref{thm:lowerBound} in the next section shows that this estimation rate is optimal.

\subsection{$\ell^1$-bound for the vanilla spectral algorithm}\label{sec:vanilla:spectral}
As a byproduct of our analysis, we provide an $\ell^{1}$-bound for the  estimation of the latent positions with the vanilla spectral algorithm (VSA), in the geometric latent model. Recanati et al.~\cite{recanati2018reconstructing} have already shown that VSA succeeds to recover the hidden permutation in the noiseless seriation problem with R-matrices. We extend their work 
 to the geometric latent model on $\Ccal$.

Starting from the noisy observation $A=F+E$ with $F_{ij}=g\circ d(x^*_{i},x^*_{j})$, we apply VSA to the whole matrix $A$ and get an estimation $\hxsp\in \R^{2\times n}$ of $\xbf^*$. 
The next proposition provides a bound in terms of the $\ell^{1}$-distance
$$\|\hxsp- {\bf x}\|_1:= \sum_{j=1}^n \sum_{i=1}^{2} \left|\hat{x}^{\textrm{VSA}}_{ij}-x_{ij}\right|,$$ 
and in terms of the spectral gaps $\Delta_1 = \lambda_{0}^* -\lambda_{1}^*$ and $\Delta_2= \lambda_{2}^* - \lambda_{3}^*$, where  $\lambda_{0}^*\geq \ldots \geq \lambda_{n_0-1}^*$ are the eigenvalues of the signal matrix $F$.

\begin{prop}\label{conj:graphgeo:l1} 
Let $n\geq 4$, and let $f=g\circ d \in\mathcal{BL}[c_l,c_L,c_e]$ be a bi-Lipschitz geometric function.  Assume that the  latent positions ${\bf x}^*$ fulfill the Assumption~ \eref{GraphGeo:assump:latentPosition} , with $c_a>0$.
 Then, there exist two constants $C_{lLea}$ and $C'_{lLea}$, depending only on $c_{l}$, $c_{L}$, $c_{e}$ and $c_{a}$, such that,
 with probability at least $1-1/n^2$, the vanilla spectral estimator $\hxsp$ satisfies 
\begin{align*}
  \min_{Q\in \Ocal}  \|\hxsp -Q{\bf x}^*\|_{1} &\leq C_{lLea} \, \frac{n \sqrt{n \log(n)}}{\left(\Delta_1 \wedge \Delta_2\right) \vee 1} \\
  &\leq C_{lLea}  \frac{n \sqrt{n \log(n)}}{\left[(\Phi_1 \wedge \Phi_2) - C'_{lLea} \sqrt{n \log(n)} \right] \vee 1}  \, .
\end{align*}
\end{prop}

Proposition~\ref{conj:graphgeo:l1}  is proved in Appendix \ref{proof:bidon:spectral:remaniement}. 
It provides an $\ell^{1}$-localization bound depending on the spectral gap $\Delta_1\wedge \Delta_2$ of the signal matrix $F$. Since there are only $n$ positions to be estimated in the bounded space $\Ccal$, this bound is uninformative when the spectral gaps  $\Delta_1 \wedge \Delta_2$ are smaller than $\sqrt{n\log(n)}$. Conversely,  when the spectral gaps are of the order of $n$, we get an $\ell^{1}$-bound of the desired scaling $\sqrt{n\log(n)}$.  

Proposition \ref{conj:graphgeo:l1}  is based on the fact that the signal matrix $F$ is well approximated by a circulant and circular-R matrix, which benefits from nice spectral properties, see Appendix \ref{appendix:spectral}. This type of R-matrices was already studied in \cite{recanati2018reconstructing} to derive some error bounds on the reconstruction of positions -- see Proposition D in \cite{recanati2018reconstructing}. Here, Proposition \ref{conj:graphgeo:l1} extends their result by providing some explicit bounds in the stochastic setting and also by considering some more general signals $F$, which are not assumed to be an exact circulant and circular R-matrix.





\section{Minimax lower bound}\label{sec:minimax}

In this section, we prove that the $\sqrt{\log(n)/n}$ rate in Theorem~\ref{thm:principal} is minimax optimal.
Let us consider the observation model $A = F + E$, where we assume that the entries $\ac{A_{ij}:i<j}$ follow independent Bernoulli distributions with parameters $f(x_i^*,x_j^*)$. We focus on this specific case of sub-Gaussian distributions in the lower bound, as  we have in mind random graph applications.  We emphasize that the same lower-bound  holds for Gaussian noise.

To prove the lower bound, we consider the simpler setting where $f_0$ is known to the statistician, and is an affine function of $d$, 
\[
 f_0(x,y) = (3/4)- d(x,y)/(4\pi)\,,\quad \textrm{for all}\ x,y\in \Ccal. 
\]
This function $f_0$ corresponds to a geometric latent model as discussed in the introduction, and it satisfies the bi-Lipschitz assumption~(\ref{cond:lipsch} - \ref{cond:lipschLower}) for $c_e=0$ and $c_l=c_L=(4\pi)^{-1}$. 
In this simple scenario, the latent  positions are identifiable up to the orthogonal transformations in $\Ocal$, so we derive a lower bound in terms of the quasi-metric $\min_{Q\in\Ocal}d_{\infty}(\hat{{\bf x}},Q{\bf x}^*)$. 
Recall that $\P_{(\textup{{\bf x}}^*,f_0)}$ denotes the distribution of $A$ with representation $({\bf x}^*,f_0)$.

\begin{thm}\label{thm:lowerBound} There exist two positive constants $C, C'$ such that  for any $n\geq C'$, we have the lower bound 
 $$\inf_{\hxbf}\ \sup_{{\bf x}^* \in \bPi_n} \  \P_{(\textup{{\bf x}}^*,f_0)} \cro{\min_{Q\in\Ocal}d_{\infty}(\hat{{\bf x}},Q{\bf x}^*) \geq C \,  \sqrt{\frac{\log(n)}{n}}\,} \geq \frac{1}{2}\enspace ,$$
 where the  infimum holds over all $\sigma(A)$-measurable functions $\hxbf$. 
 \end{thm}
 
 The proof of the Theorem \ref{thm:lowerBound} is given in Appendix~\ref{append:LB}.
 The lower bound is written over the collection of $n$-tuples ${\bf x}^* \in \bPi_n$, which is a subclass of the class considered in our upper bounds (since all ${\bf x}^* \in \bPi_n$ satisfy the condition \eqref{assump1:unif} for any $c_a \geq 0$). The lower bound matches the upper bound in Theorem \ref{cor:thm:bis} up to some multiplicative constants. Therefore, it implies the optimality of the $\sqrt{\log(n)/n}$ estimation rate of our estimator (in the minimax sense).   The fact that the lower bound holds even for a known function entails that the  rate $\sqrt{\log(n)/n}$ is not driven by the (absence of) knowledge of the affinity function in our setting. Moreover, since the affine function $f_0$ satisfies the bi-Lipschitz assumption~(\ref{cond:lipsch}--\ref{cond:lipschLower}) for $c_e=0$, i.e. $f_0\in \mathcal{BL}[(4\pi)^{-1},(4\pi)^{-1},0]$, this entails that the rate $\sqrt{\log(n)/n}$ is not due to the slack $\ep_n = c_e\sqrt{\log(n)/n}$ in the bi-Lipschitz assumption. In fact, we precisely allow this slack   $c_e\sqrt{\log(n)/n}$ in~(\ref{cond:lipsch}--\ref{cond:lipschLower}) because this generalization does not worsen the estimation rate compared to pure bi-Lipschitz functions ($c_e=0$). Finally, since the set of $n$-tuples ${\bf x}^* \in \bPi_n$ is in correspondence with the set of permutations $\si^*$ of $[n]$, Theorem \ref{thm:lowerBound} ensures
 that the bound~\eref{borne:seriation} is rate-optimal  for the bi-Lipschitz seriation problem.

\section{Numerical experiments}\label{sec:numerique}
\input{seriation_numeric.tex}

\section{Discussion}\label{sec:discussion}

Relying on observations of pairwise affinities  in a latent space model,  we studied the problem of uniformly localizing positions ${\bf x}^*=(x_1^*,\ldots,x_n^*)$ on the unit sphere $\Ccal\subset \mathbb{R}^2$. Under bi-Lipschitz assumptions on the affinity function, we established the  rate $\sqrt{\log(n)/n}$ for the uniform localization of balanced $n$-tuples ${\bf x}^*\in \bPi_{n}$.  We also proved that non-trivial estimation error is still possible when the latent points do not form a balanced $n$-tuple (${\bf x}^*\notin \bPi_{ n}$) to the price of  an additional bias $d_{\infty}({\bf x}^*,\bPi_{n})$. This bias remains small compared to the $\sqrt{\log(n)/n}$ rate when the points have been sampled uniformly at random on $\Ccal$. 

We also analyzed a spectral embedding alternative in Section~\ref{sec:geometric}, which benefits from a polynomial-time complexity. 
When the function $f$ is geometric and when the associated Fourier coefficients are suitably separated, this spectral method achieves the optimal rate $\sqrt{\log(n)/n}$ for uniform localization.  
{Yet, the spectral embedding takes advantage of  the structure of Toeplitz R-matrix. Since this structure disappears in the general case of bi-Lipschitz functions, there is no apparent reason for the spectral algorithm to work over the whole class of bi-Lipschitz functions. 

As our non-polynomial-time algorithm is based on an instance of the Quadratic Assignment Problem, which is known to be NP Hard and even hard to approximate, the existence of polynomial-time algorithms achieving the  $\sqrt{\log(n)/n}$ rate over the whole class of bi-Lipschitz functions remains an open question.}

{
The latent positions are not identifiable when $f$ is unknown, and our main hypothesis is that there exists a representation $(\textbf{x},f)$ with $\textbf{x}$ close to $\Pi_{n}$ and $f$ bi-Lipschitz. We use as reference the regular distribution $\Pi_{n}$, since regular and uniform distributions are the ones that appear in classical models like graphon, $f$-random graphs, or statistical seriation. Our algorithms builds on this hypothesis,  and consequently the bias
 $\textup{min}_{(\textup{\textbf{x}},f) \in \mathcal{R}(c_l,c_L,c_e)} d_{\infty}(\textup{\textbf{x}},\Pi_n)$ appears in our bounds, where the minimum is over  the set $\mathcal{R}(c_l,c_L,c_e)$ of bi-Lipschitz representatives $(\textup{\textbf{x}},f)$.  This minimum leaves room to handle situations where the latent positions do not match the regular grid $\Pi_n$ but are only more or less evenly spread. For instance, the minimal bias is zero for some representations $(\textbf{x},f)$, with $\textup{\textbf{x}}$ as far apart from $\Pi_n$ as $d_{\infty}(\textup{\textbf{x}},\Pi_n)\geq \pi/8$ (Proposition \ref{prop_identif}). 
Yet, there are many practical situations where the affinity matrix is clustered, that we cannot handle. In the case where the affinity matrix is clustered ($f$ bi-Lipschitz, but the $x_{i}$ are clustered), the problem becomes a clustering problem, rather than a seriation problem, and our algorithms are not suited for clustering data. The question of handling simultaneously clustering and seriation is very interesting, but it is beyond the scope of this paper.}

{In this manuscript, we focused our attention to symmetric pairwise affinity functions $f$. However, other one-dimensional localization models such as Bradley-Terry model or more generally ranking problems, do not satisfy the symmetry assumption. Still, we hope that our general two-step approach can leverage other structural assumptions.
In ranking, a natural counterpart of our model~\eqref{eq:model} is the so-called SST model introduced by~\cite{shah2016stochastically}, which is defined as follows. We observe $A_{ij}= f(x_i^*, x_j^*)+E_{ij}$ where the function $f:(x,y) \in [0,1]\times[0,1]\mapsto f(x,y) \in [0,1]$ is non-decreasing with respect to $x$ and non-increasing with respect to $y$ and satisfies the skew symmetry assumption, that is $f(x,y)= 1-f(y,x)$. In this setting, $f(x_i^*,x_j^*)$ stands for the probability that player $i$ wins a game against player $j$. Note that the latent space is now $[0,1]$ and not the torus $\mathcal{C}$ anymore. Although our methodology does not apply verbatim, we could adapt the Localize and Refine procedure for the latent space $[0,1]$. To exploit the bi-isotonic and skew-symmetric assumptions, the refinement estimator of~\eqref{def:estim:Linfini} could for instance be  replaced by  
\[
    \hat{x}^{(2)}_i \in \argmax_{z \in G_{n_0}} \,  \big\langle (A_{i , S}-\frac{1}{2}),  z-\hxbf^{(1)}_S\big\rangle,\quad\textrm{for each}\  i\in\bS\ , 
\]
where $G_{n_0}$ stands for the regular grid $\{\tfrac{1}{n_0}, \tfrac{2}{n_0},\ldots, 1\}$ and $\hxbf^{(1)}_S$ is a suitable first-step estimator. In comparison to 
\eqref{def:estim:Linfini}, $D(z,\hxbf^{(1)}_S)$ is replaced by $z -\hxbf^{(1)}_S$. We expect that, with a suitable initialization $\hxbf^{(1)}_S$ and under bi-Lipschitz assumptions, the resulting procedure achieves near-optimal localization rates. This is an interesting direction for future research.  
}

\input{seriation_appendix.tex}



\bibliographystyle{imsart-number} 
\bibliography{biblio_seriation}       


\end{document}

%% file: seriation_numeric.tex
 

\subsection{Optimal rate}

In Figure \eqref{fig0:optimality}, we study the ratio $r = \frac{ \textup{min}_{Q \in \mathcal{Q}} \, d_{\infty}(\hat{\textbf{x}},Q\textbf{x}^*)}{v_{opt}}$
of the maximum error  of the  Localize-and-Refine algorithm $\hat{\textbf{x}}$ (without data splitting) and the optimal rate $v_{opt} = \sqrt{\log(n)/n}$. For each sample size $n=100, 200, 300, 400$, a dot represents the average of $50$  ratios $r_1,\ldots,r_{50}$ obtained on independent data sets $A^{(1)}, \ldots,A^{(50)}$.  Each data matrix $A^{(j)}$, $j\in[50]$, has been generated as in the model (3), with the three following specifications. The latent points $x_1^*,\ldots,x_n^*$ are sampled independently and uniformly on $\mathcal{C}$.  The affinity function is  the affine geometric function $f(x,y) = 1 - d(x,y)/(2 \pi)$. The entries $E_{ij}$, $1 \leq i < j\leq n$, of the noise matrix  are independent Gaussian random variables, with a standard deviation that is  either equal to $0.1$ (green curve) or to $0.5$ (red curve).

One can observe in Figure \eqref{fig0:optimality} that the (averaged) ratio for  $\sigma = 0.1$ is (approximately) constant and  equal to $1$, while for $\sigma = 0.5$ it   decreases from $7$ to $5$. In other words, the maximum error of the Localize-and-Refine algorithm follows a  $\sqrt{\log(n)/n}$ rate, up to a multiplicative constant   $C \in [\frac{1}{2},8]$,  for sample sizes  $n\geq 100$. This corroborates the  conclusion of our theoretical findings (upper bound of Corollary \ref{coro:iid:perf}  and lower bound of  Theorem \ref{thm:lowerBound}) that the Localize-and-Refine algorithm achieves the  optimal $\sqrt{\log(n)/n}$ rate up to a multiplicative constant $C$ that is  bounded away from zero and bounded from above. An interesting question (for future research) would be to understand the dependencies of $C$ in the problem parameters.  Figure \eqref{fig0:optimality} indeed shows that $C$  behaves differently when $\si=0.1$ or $\si=0.5$, and that $C$ varies with $n$. 

\begin{figure}
\centering
  \includegraphics[width=0.7\linewidth]{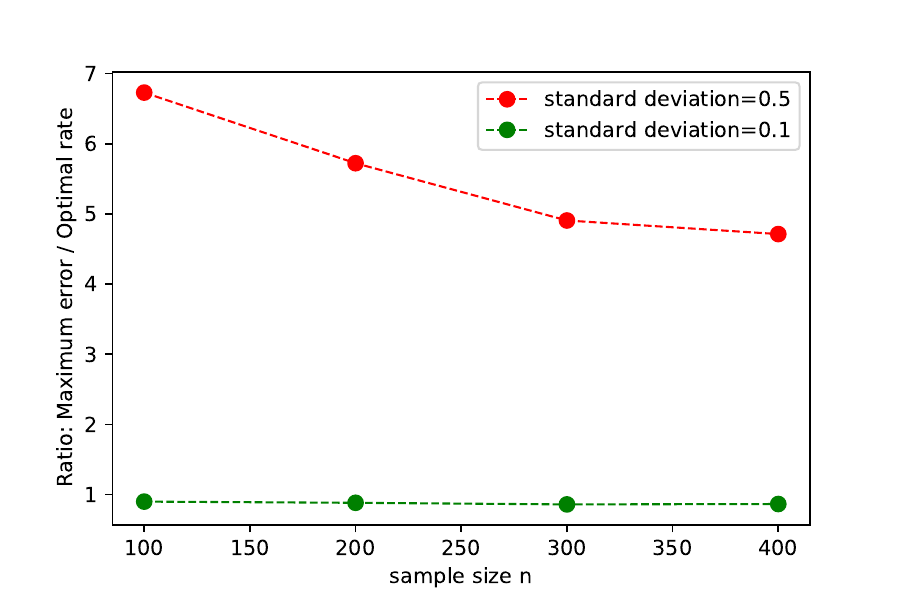}
\caption{Ratio of the maximum error  of the  Localize-and-Refine algorithm and the optimal rate $\sqrt{\log(n)/n}$, with noise level $\sigma = 0.1$ (green) and $\sigma = 0.5$ (red). }
\label{fig0:optimality}
\end{figure}

\subsection{Usefulness of data splitting? of refined estimation step 2?}
In this section, we investigate two questions relative to 
 the empirical performance of the Localize-and-Refine algorithm with the initial localization $\hat{\textbf{x}}^{(1)}_S$  given by the spectral output  $\tilde{\textbf{x}}^{(1)}_S$ (defined page \pageref{alg:LHO:geo}): (i) Is the data splitting  useful in practice? (ii) Does the refined estimation (step 2) empirically improve the initial localization (step 1)?

\begin{figure}%
\centering
\begin{subfigure}{.5\textwidth}
  \centering
  \includegraphics[width=1\linewidth]{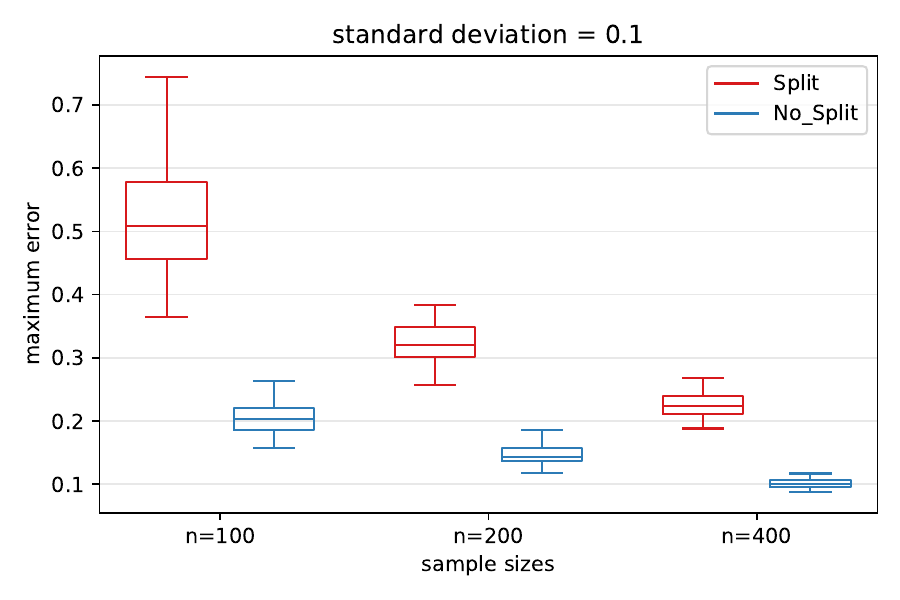}
  \caption{affine signal}
  \label{fig1:sub1}
\end{subfigure}%
\begin{subfigure}{.5\textwidth}
  \centering
  \includegraphics[width=1\linewidth]{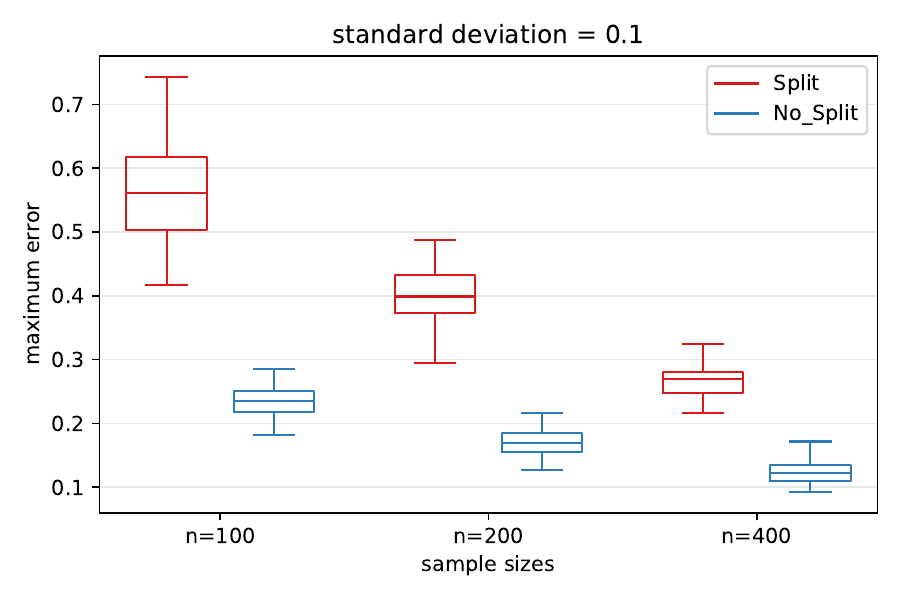}
  \caption{logit signal}
  \label{fig1:sub2}
\end{subfigure}

\begin{subfigure}{.5\textwidth}
  \centering
  \includegraphics[width=1\linewidth]{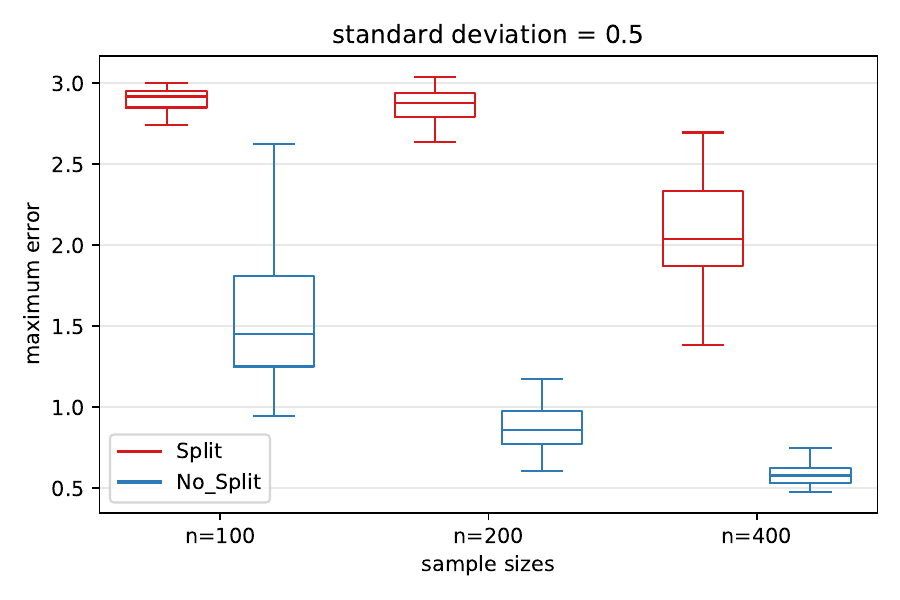}
  \caption{affine signal}
  \label{fig1:sub3}
\end{subfigure}%
\begin{subfigure}{.5\textwidth}
  \centering
  \includegraphics[width=1\linewidth]{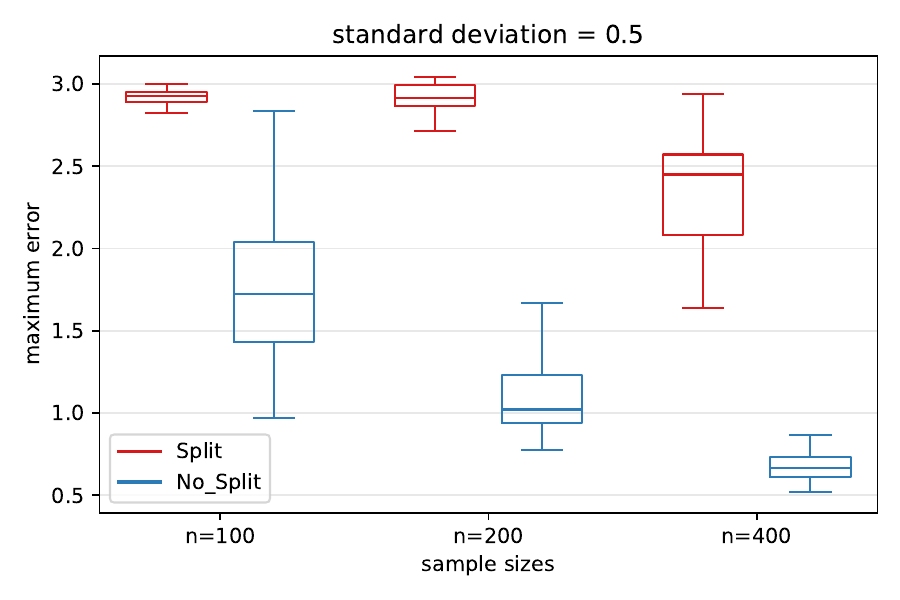}
  \caption{logit signal}
  \label{fig1:sub4}
\end{subfigure}

\caption{Localization error in $d_{\infty}$-distance for the  Localize-and-Refine algorithm with data-splitting (red) and without data-splitting (blue). Top line: noise with standard-deviation sd= 0.1. Bottom line: noise with sd=0.5. Left: affine geometric affinity function $f(x,y) = 1 - d(x,y)/(2 \pi)$. Right: Logit geometric affinity function $f(x,y) = \exp [-d(x,y)] / \left(1+ \exp[-d(x,y)] \right)$.}
\label{fig1:spltting}
\end{figure}

In each Figure \ref{fig1:spltting} and \ref{fig2:spectral}, we compare two algorithms, presenting boxplots of their localization errors in $d_{\infty}$-distance. Each boxplot represents the distribution of $50$ errors made on $50$ samplings of the data matrices $A$.  Each data matrix $A$ is generated as in the model (3), with the three following specifications. The latent points $x_1^*,\ldots,x_n^*$ are sampled independently and uniformly on $\mathcal{C}$.  For the affinity function, we choose  either the affine geometric function $f(x,y) = 1 - d(x,y)/(2 \pi)$, or the logit geometric function $f(x,y) = \exp [-d(x,y)] / \left(1+ \exp[-d(x,y)] \right)$. In the noise matrix, the entries $E_{ij}$, $1 \leq i < j\leq n$,  are independent Gaussian random variables, with a standard deviation that is either equal to $0.1$ (top line) or $0.5$ (bottom line). The same protocol is used in Figure \ref{fig3:spectral_L1_error}, except that we measure the localization error in $d_1$-distance, instead of $d_{\infty}$-distance. \smallskip

\textit{Question (i):} We use a data splitting scheme in the Localize-and-Refine algorithm in order to ensure independence between the data used in the two steps. This independence was convenient to prove theoretical guarantees (as Theorem 4.1). Yet, data splitting makes the initial localization run on a $(n/4)\times (n/4)$ data matrix, instead of the whole $n\times n$ matrix, which is expected to enlarge the variance of this initial localization by a factor 4. So, one can wonder whether the splitting is necessary and useful in practice. 
To answer this question, we illustrate in Figure \ref{fig1:spltting} the difference between the performances of the Localize-and-Refine algorithm and the homologous procedure  without data splitting (the former is plotted in red, the latter in blue). One can observe that the $d_{\infty}$-localization error is much smaller for the procedure without splitting. A plausible explanation for the good performances without data-splitting is that the statistical dependence between the steps 1 and 2 of the algorithm is negligible for large $n$, rendering the data splitting useless. Indeed, in the no-splitting version of the algorithm, step~1 uses $O(n^2)$ observations to release a first localization $\tilde{\textbf{x}}^{(1)}$ of the $n$ positions, then step~2 refines the estimation of a position $x_i$ using $n$ observations, which only represents a fraction $O(1/n)$ of the observations used in step 1. Hence, the dependence between $\tilde{\textbf{x}}^{(1)}$ in step 1 and the $n$ observations in step 2 could be sufficiently small to not require a data splitting. 
Accordingly, we recommend the version of  the Localize-and-Refine algorithm without data-spliting for practical use.
As a future direction of research, it would be interesting to investigate the theoretical performance of the algorithm without data splitting,  in order to  bridge the gap between the theory and the practice.
\smallskip

\textit{Question (ii):}  The strategy of the Localize-and-Refine algorithm is to get an initial localization with controlled $d_{1}$-error and then to refine the localization in order to ensure a control in the $d_{\infty}$-metric. 
A natural question is whether the refinement step 2  improves the initial localization obtained by  the Spectral algorithm in step 1.
We investigate numerically this question in Figure \ref{fig2:spectral}, by comparing the $d_{\infty}$-error of the Spectral Localization procedure (plotted in red) and of the Localize-and-Refine algorithm without data splitting (in blue). 
One can observe contrasting results, depending on the standard-deviation of the noise. 
When the standard-deviation is 0.5 (bottom line), the second step offers no significant improvement in the $d_{\infty}$-localization error.  Conversely, when the standard-deviation is 0.1 (top line), the $d_{\infty}$-error is significantly improved by the refinement step. This suggests that, to be useful, the refinement step requires a precise enough initial localization. 
We complement Figure \ref{fig2:spectral} with Figure \ref{fig3:spectral_L1_error},  which displays the errors in  $d_1$-distance (scaled by $1/n$ for a better comparison), instead of the $d_{\infty}$-distance, though this loss function is not our main concern in this paper. In Figure \ref{fig3:spectral_L1_error}, we observe  a  behavior in the $(d_{1}/n)$-metric very similar to the behavior in  the  $d_{\infty}$-metric, displayed in Figure \ref{fig2:spectral}.
In the light of the numerical performance of the Spectral Localization in Figure \ref{fig2:spectral}, an interesting open question is wether we can prove theoretical guarantees on the $d_{\infty}$-localization error of this procedure.

\begin{figure}%
\centering
\begin{subfigure}{.5\textwidth}
  \centering
  \includegraphics[width=1\linewidth]{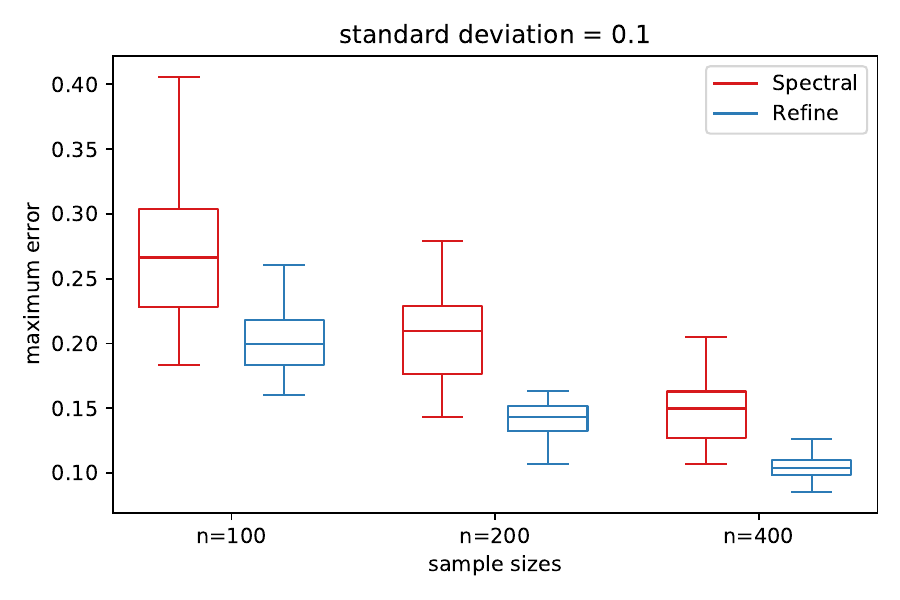}
  \caption{affine signal}
  \label{fig2:sub1}
\end{subfigure}%
\begin{subfigure}{.5\textwidth}
  \centering
  \includegraphics[width=1\linewidth]{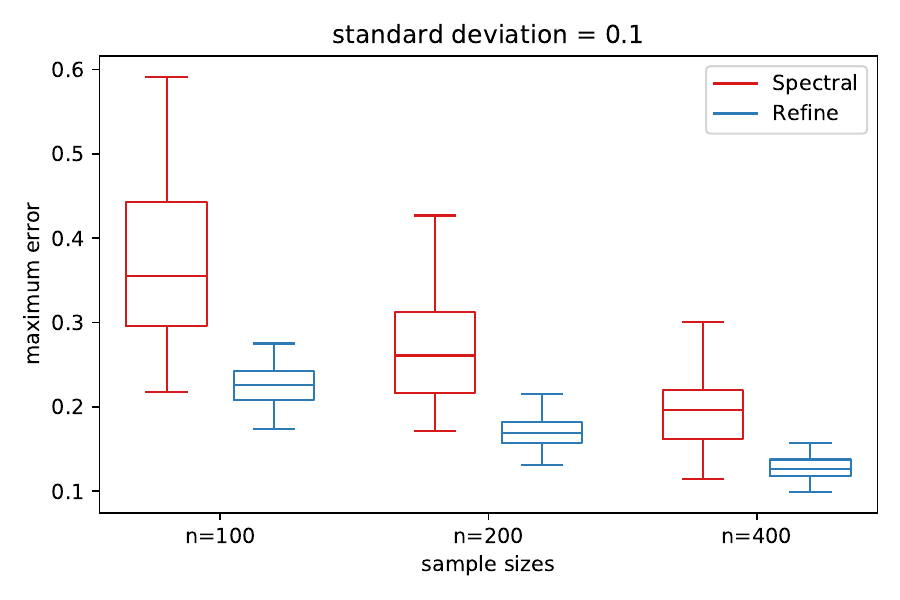}
  \caption{logit signal}
  \label{fig2:sub2}
\end{subfigure}

\begin{subfigure}{.5\textwidth}
  \centering
  \includegraphics[width=1\linewidth]{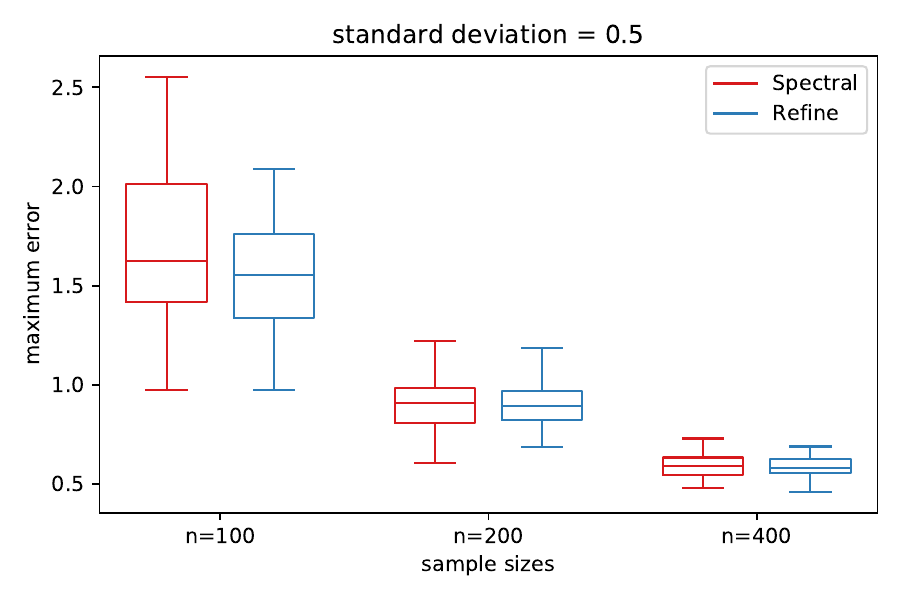}
  \caption{affine signal}
  \label{fig2:sub3}
\end{subfigure}%
\begin{subfigure}{.5\textwidth}
  \centering
  \includegraphics[width=1\linewidth]{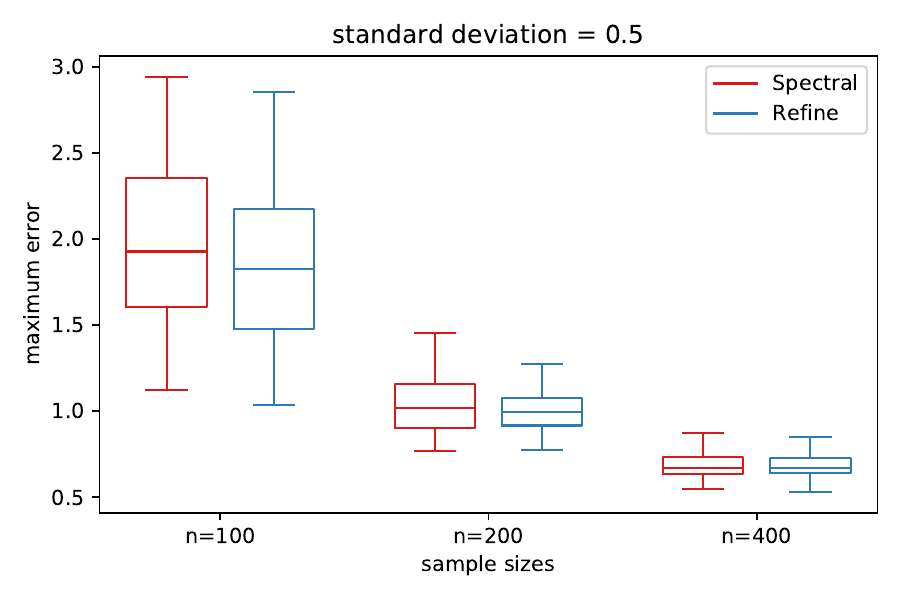}
  \caption{logit signal}
  \label{fig2:sub4}
\end{subfigure}

\caption{Localization error in $d_{\infty}$-distance for the Spectral algorithm (red) and the Localize-and-Refine algorithm (blue). Top line: noise with sd=0.1. Bottom line: noise with sd=0.5. Left: affine geometric affinity function. Right: Logit geometric affinity function.}
\label{fig2:spectral}
\end{figure}

\begin{figure}%
\centering
\begin{subfigure}{.5\textwidth}
  \centering
  \includegraphics[width=1\linewidth]{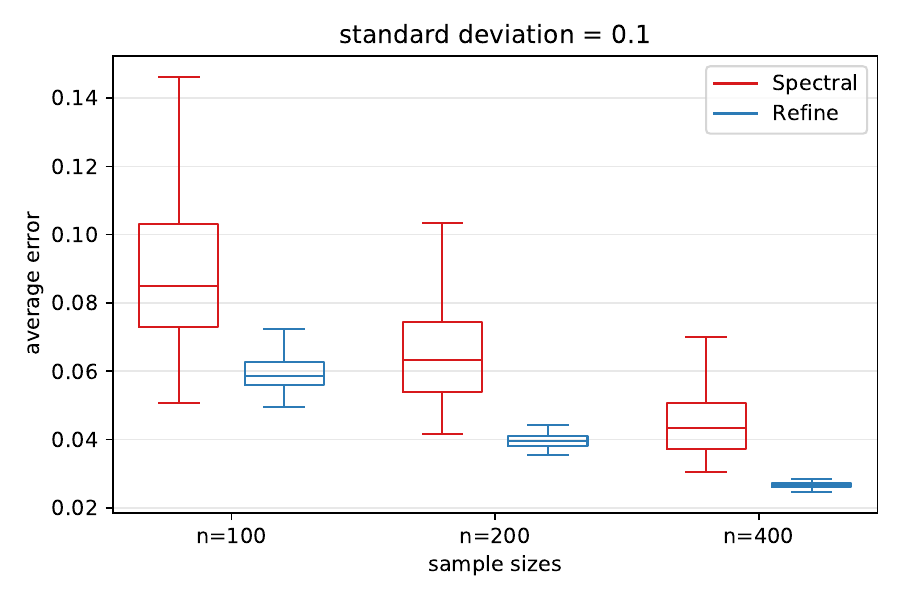}
  \caption{affine signal}
  \label{fig3:sub1}
\end{subfigure}%
\begin{subfigure}{.5\textwidth}
  \centering
  \includegraphics[width=1\linewidth]{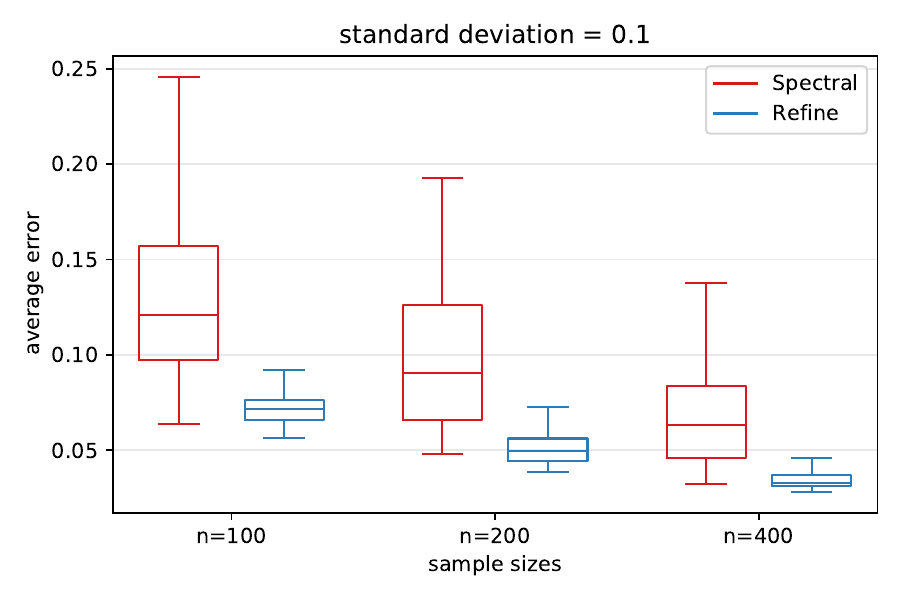}
  \caption{logit signal}
  \label{fig3:sub2}
\end{subfigure}

\begin{subfigure}{.5\textwidth}
  \centering
  \includegraphics[width=1\linewidth]{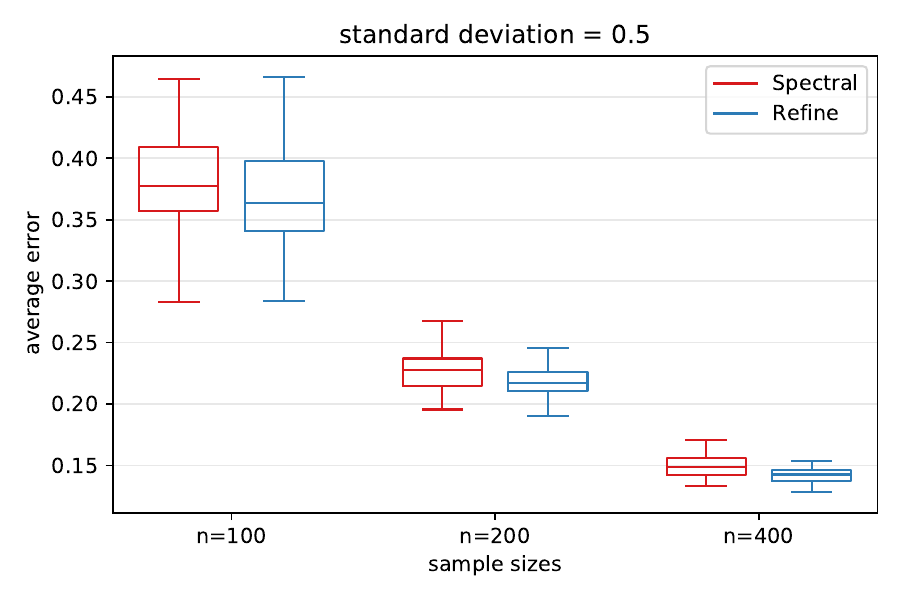}
  \caption{affine signal}
  \label{fig3:sub3}
\end{subfigure}%
\begin{subfigure}{.5\textwidth}
  \centering
  \includegraphics[width=1\linewidth]{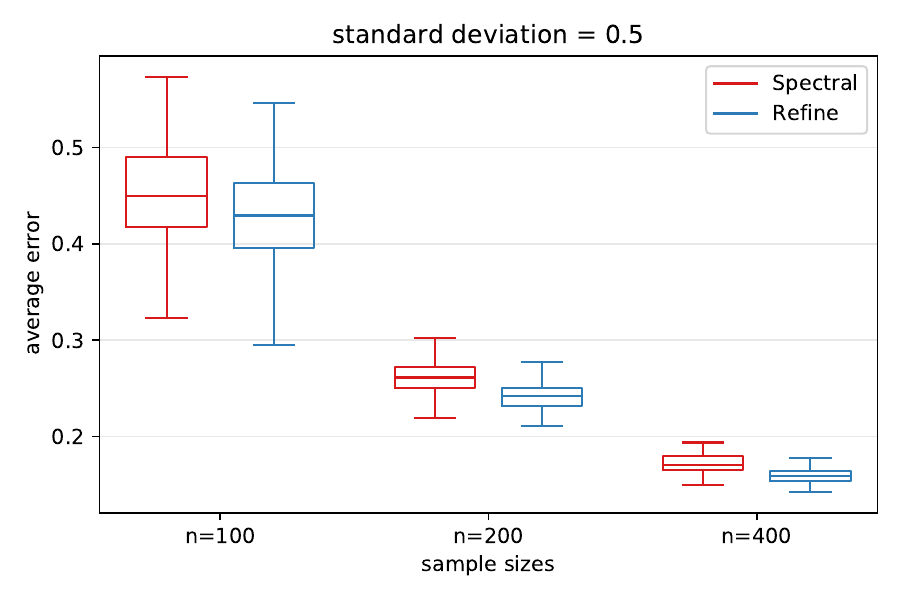}
  \caption{logit signal}
  \label{fig3:sub4}
\end{subfigure}

\caption{Localization error in $\pa{{1\over n}d_{1}}$-distance for the Spectral algorithm (red) and the Localize-and-Refine algorithm (blue). Top line: noise with sd=0.1. Bottom line: noise with sd=0.5. Left: affine geometric affinity function. Right: Logit geometric affinity function.}
\label{fig3:spectral_L1_error}
\end{figure}

%% file: seriation_appendix.tex

\appendix


\section{Some intuition on our analysis}\label{app:central}

To get some intuition on the rationale behind our analysis, 
we single out the next lemma, which is a cornerstone of the analysis at least in the simplified situation where $c_e=0$ and where the latent positions belong to $\bPi_n$. Then, we discuss some consequences in simplified versions of our work.

\begin{lem}\label{lem:conerstone:variante}
Let $\alpha,\eps\geq 0$ be two non-negative constants, and let $(a_{j})_{j=1,\ldots,p}$ and $(d_{j})_{j=1,\ldots,p}$ be two sequences fulfilling
$$a_{1}\geq a_{2}\geq \ldots \geq a_{p},\quad d_{1}\leq d_{2}\leq \ldots\leq d_{p},\quad \textrm{and} \quad 
a_{j}-a_{j+1}\geq \alpha (d_{j+1}-d_{j})-\eps \,,$$
for $j=1,\ldots,p-1$.
Then, for any permutation $\sigma:[p]\to[p]$ we have
\begin{equation}\label{eq:cornerstone:variante}
\sum_{j=1}^p a_{j}(d_{\sigma(j)}-d_{j})\geq {\alpha\over 2} \sum_{j=1}^p (d_{j}-d_{\sigma(j)})^2-\eps \sum_{j=1}^p j(d_{j}-d_{\sigma(j)}).
\end{equation}
\end{lem}

\noindent{\bf Proof of Lemma~\ref{lem:conerstone:variante}.}
Let us set the notation $D_{j}=d_{1}+\ldots+d_{j}$, and
$d^{\sigma}_{j}=d_{\sigma(j)}$, and $D^{\sigma}_{j}=d^{\sigma}_{1}+\ldots+d^{\sigma}_{j}$ for $j=1,\ldots,p$.
Since $d$ is non-decreasing, and since $\sigma$ is a permutation of $[p]$, we have
$$D^{\sigma}_{j}\geq D_{j},\ \ \textrm{for}\ j=1,\ldots,p,\quad \textrm{and}\quad D^{\sigma}_{p}=D_{p}\enspace .$$
Writing $d^{\sigma}_{j}=D^{\sigma}_{j}-D^{\sigma}_{j-1}$ and rearranging the sums, we get
\begin{align*}
\sum_{j=1}^p a_{j}(d^{\sigma}_{j}-d_{j})  &= \sum_{j=1}^p a_{j}\big((D^{\sigma}_{j}-D^{\sigma}_{j-1})-(D_{j}-D_{j-1})\big)\\
&= \underbrace{a_{p}\pa{D^{\sigma}_{p}-D_{p}}}_{=0}+\sum_{j=1}^{p-1} \underbrace{(a_{j}-a_{j+1})}_{\geq \alpha (d_{j+1}-d_{j})-\eps} \underbrace{(D^{\sigma}_{j}-D_{j})}_{\geq 0}\\
&\geq \sum_{j=1}^{p-1}  [\alpha(d_{j+1}-d_{j})-\eps(j+1-j)](D^{\sigma}_{j}-D_{j})\\
&=  \sum_{j=1}^p ( \alpha d_{j}-\eps j)(d_{j}-d^{\sigma}_{j})\\
&={\alpha\over 2} \sum_{j=1}^p (d_{j}-d^{\sigma}_{j})^2-\eps \sum_{j=1}^p j(d_{j}-d^{\sigma}_{j})\enspace , 
\end{align*}
where we used Abel transformation in the penultimate line. 
The proof of Lemma \ref{lem:conerstone:variante} is complete. \hfill$\square$
\medskip

Let us discuss some immediate consequences of the above lemma for our problem.
Let us consider the case where the entries $A_{ij}$ of the matrix $A$ decrease with $d(x^*_{i},x^*_{j})$ for some
$\xbf^*\in\bPi_{n}$. For a fixed $i$, let $\tau$ be a permutation of $[n]$ such that
$\big\{d(x^*_{i},x^*_{\tau(j)}): {j=1,\ldots,n}\big\}$ is ranked in increasing order. Let us set $a_{j}=A_{i\tau(j)}$ and $d_{j}=d(x^*_{i},x^*_{\tau(j)})$.
Since the entries $A_{ij}$ decrease with $d(x^*_{i},x^*_{j})$, the sequences $(a_{j})_{j=1,\ldots,n}$ and $(d_{j})_{j=1,\ldots,n}$ fulfill the conditions of Lemma~\ref{lem:conerstone:variante} with $\alpha=\eps=0$. 
Let us pick $k\in [n]$ and let us denote by $\sigma_{k}$ the permutation of $[n]$ such that
$d(x^*_{k},x^*_{\tau(j)})=d(x^*_{i},x^*_{\tau(\sigma_{k}(j))})=d_{\sigma_{k}(j)}$ {-- this is possible because $\xbf^*\in \bPi_n$}. We notice that $\sigma_{i}=Id$.
Then, Lemma~\ref{lem:conerstone:variante} ensures that
$$\sum_{j=1}^n A_{i\tau(j)} d(x^*_{k},x^*_{\tau(j)}) = \sum_{j=1}^n A_{i\tau(j)} d(x^*_{i},x^*_{\tau(\sigma_{k}(j))})  \geq \sum_{j=1}^n A_{i\tau(j)} d(x^*_{i},x^*_{\tau(j)})\enspace ,$$
so that
$$x^*_{i} \in \argmin_{z \in \Ccal_{n}} \sum_{j=1}^n A_{ij} d(z,x^*_{j})\enspace .$$
{This justifies that the criterion underlying the refined estimator~\eqref{def:estim:Linfini} is able to recover the true latent position $x^*_j$ at least in an idealized setting where the observations are noiseless, the entries of $A_{ij}$ are decreasing with $d(x^*_i,x^*_j)$, and the true latent positions $x^*_j$ are plugged in~\eqref{def:estim:Linfini} instead of the initial estimator $\widehat{x}^{(1)}$. 
}

\medskip 

When, in addition, we have a lower Lipschitz condition
\[    
A_{i\tau(j)} - A_{i\tau(j+1)}\geq c_{l} \pa{d(x^*_{i},x^*_{\tau(j+1)})-d(x^*_{i},x^*_{\tau(j)})}\enspace ,  
\]
then, applying Lemma~\ref{lem:conerstone:variante}, we can lower bound the difference 
$$\sum_{j=1}^n A_{ij}d(x^*_{k},x^*_{j}) -\sum_{j=1}^n A_{ij}d(x^*_{i},x^*_{j})
\geq {c_{l}\over 2} \sum_{j=1}^n \pa{d(x^*_{i},x^*_{j})-d(x^*_{k},x^*_{j})}^2\enspace  .$$
In particular, we observe that, for all $z\in\Ccal_{n}$,
$$\sum_{j=1}^n A_{ij} d(z,x^*_{j}) \geq \sum_{j=1}^n A_{ij}d(x^*_{i},x^*_{j})
+{c_{l}\over 2} \sum_{j=1}^n \pa{d(x^*_{i},x^*_{j})-d(z,x^*_{j})}^2,$$
so the sum $\sum_{j} A_{ij} d(z,x^*_{j})$  locally increases, when $z$ moves away from $x_{i}^*$. {In the general case, where $c_e>0$ and the observations are noisy, the criterion does not satisfy a simple local quadratic lower bound and we need to rely on finer arguments than Lemma~\ref{lem:conerstone:variante} -- see e.g. the proofs  Lemma~\ref{claim:LLowerB} and~\ref{lem:regret}.}
\medskip

Finally, we sketch here the proof of the second result of Proposition~\ref{prop:representative}, in the specific case where $c_e=0$. Consider any two representations $(\xbf,f)$ and $(\xbf,f')$ in $\Rcal[F,c_l,c_L,0]$ with $\xbf$, $\xbf'\in \bPi_n$. Since both $\xbf$ and $\xbf'$ belong to $\bPi_n$, this implies that, for any fixed $i$, the vectors $(d(x_i,x_j))_j$ and $(d(x'_i,x'_j))_j$ are equal, up to a permutation of the entries.  As $c_e=0$, the lower Lipschitz condition~\eqref{cond:lipschLower} ensures that $(F_{i,j})_{j=1,\ldots, n}$ is decreasing both with respect to $d(x_i,x_j)$ and $d(x'_i,x'_j)$. As a consequence, we have $d(x_i,x_j)= d(x'_i,x'_j)$ for any $i$, $j$ in $[n]$. We now show that this implies that $\xbf= Q\xbf'$ for some $Q\in \mathcal{O}$.  Denote $\sigma$ the permutation of $[n]$ such that $\sigma(1)=1$ and the arguments $\underline{x}_j$ satisfy $\underline{x}_{\sigma(i+1)}=  \underline{x}_{\sigma(i)}+2\pi/n$. As a consequence, we have $d(\xbf'_{\sigma(i)},\xbf'_{\sigma(i+1)})=  d(\xbf_{\sigma(i)},\xbf_{\sigma(i+1)})= 2\pi/n$. This implies that either  $\underline{x'}_{\sigma(i+1)}=  \underline{x'}_{\sigma(i)}+2\pi/n$ for all $i$, or $\underline{x'}_{\sigma(i+1)}=  \underline{x'}_{\sigma(i)}-2\pi/n$ for all $i$. In the former case, one easily sees that $\xbf= Q\xbf'$, where $Q$ is the rotation satisfying $x_1=Qx'_1$, whereas in the latter case, we have $\xbf= Q\xbf'$, where $Q$ is the reflection satisfying $x_1=Qx'_1$.

\section{Proofs of main results}\label{sect:proof:lem:perf:LHO:specificase}

Recall that $n = 4 n_0$.
Given an orthogonal transformation $Q\in \Ocal$, we define the $d_1$-loss relative to $Q$ as
\begin{equation}
    \mucal := \frac{d_1(\hat{\bf x}^{(1)}_S, Q{\bf x}^*_{S})}{n_0} \enspace.
\end{equation} 

Before proving Proposition~\ref{prop:LHO:perf}, we study, as a warm-up, the simpler situation where all the latent positions $x_i^*$ are elements of the regular grid $\Ccal_{n_0}$ and where the vector ${\bf x}^*_{S}$ (composed of $n_0$ coordinates of ${\bf x}^*$) belongs to $\bPi_{n_0}$. In this case, $d_{\infty}({\bf x}^{*}_{S},\bPi_{n_0})= 0 $.

\begin{lem}\label{LHO:perf} In addition of the assumptions listed in Proposition~\ref{prop:LHO:perf}, we assume that ${\bf x}^* \in \Ccal_{n_0}^{  n}$ and ${\bf x}^*_{S} \in \bPi_{n_0}$. Then for any $Q\in \Ocal$ and any $i\in \overline{S}$, the estimator \eqref{def:estim:Linfini} satisfies the following bound  $$d_{\infty}(\hat{x}_i^{(2)},  Q x^*_i) \leq C_{l,L,e}  \left(\mucal + \sqrt{\frac{\log(n)}{n}}\right)\enspace, $$
with probability at least $1-1/n^3$.
\end{lem}
Taking a union bound over the $n-n_0$ indices $i\in \overline{S}$, we  $$d_{\infty}(\hat{{\bf x}}_{\overline{S}}, Q{\bf x}^*_{\overline{S}}) \leq C_{l,L,e} \left( \mucal + \sqrt{\frac{\log(n)}{n}}\right)\enspace ,$$
with probability higher than $1-1/n^2$. This is exactly the conclusion of Proposition~\ref{prop:LHO:perf} in the special case where ${\bf x}^* \in \Ccal_{n_0}^{  n}$ and ${\bf x}^*_{S}\in \bPi_{n_0}$.
The proof of Proposition~\ref{prop:LHO:perf} for general ${\bf x}^*$ follows the same scheme as that of Lemma \ref{LHO:perf}, but also requires some slight refinements. We first prove Lemma~\ref{LHO:perf} before turning to the general case.

\subsection{Proof of Lemma \ref{LHO:perf}}

First, we claim that it suffices to restrict our attention to transformations $Q\in$$\mathcal{O}$ that let $\Ccal_{n_0}$ invariant. Indeed, for general $Q$, there exists an orthogonal transformation $Q'$, letting $\Ccal_{n_0}$ invariant, and 
such that $\max_{z \in \Ccal_{n_0}}$ $d(Qz,Q'z)\lesssim 1/n_0$. Replacing $Q'$ by $Q$ in the statement of Lemma \ref{LHO:perf} only entails an additional term of order $1/n_0$ which is negligible  compared to the term $\sqrt{\log(n)/n}$.

Let $i\in\overline{S}$. In the two next lemmas, we bound 
\begin{equation}\label{eq:demo:Li}
L_i := \langle F_{i ,S}, \, D(Q^{-1}\hat{x}_i^{(2)}, {\bf x}^*_{S}) - D(x_i^*, {\bf x}^*_{S}) \rangle
\end{equation} from above and below. We recall that $F_{i,S}$ is the vector ($f(x_i^*,x_j^*$) for $j\in S$.

\begin{lem}\label{claimLUpBound}
With probability at least $1-1/n^3$, we have 
\[
 L_i \leq  C_{L,e}  d(\hat{x}_i^{(2)}, Qx^*_i) \left(n\mucal + \sqrt{n \log(n)}\right)\enspace ,\]
 for some constant $C_{L,e}>0$.   
\end{lem}

\begin{lem}\label{claim:LLowerB}
  We have $$L_i  \geq C n d(\hat{x}_i^{(2)}, Qx^*_i) \Big{(}c_l d(\hat{x}_i^{(2)}, Qx^*_i) - \ep_n\Big{)} -\frac{c_e^3}{\pi c_l^{2}} \sqrt{\frac{\log^3(n)}{n}}- 2\frac{c_e}{c_l}\sqrt{\frac{\log(n)}{n}}\enspace ,$$ for some numerical constant $C>0$ and all $n$ larger than quantity $C_{l,e}$. 
\end{lem}

 These two lemmas imply that, for $n$ large enough and $$d(\hat{x}_i^{(2)}, Qx^*_i) \geq 2C'_{l,e}\sqrt{\log(n)/n}\enspace,$$ with $C'_{l,e}$ large enough, we have 
 $$C'_{l,e} nd^2(\hat{x}_i^{(2)}, Qx^*_i) \leq L_i \leq C_{L,e}   d(\hat{x}_i^{(2)}, Qx^*_i) \left(n\mucal + \sqrt{n \log(n)}\right)\enspace . $$
We conclude that the error bound $d(\hat{x}_i^{(2)}, Q x^*_i)\leq C_{l,L,e} [\mucal + \sqrt{\log(n)/n}] $ holds with probability at least $1-1/n^3$.

\subsubsection{Proof of Lemma \ref{claimLUpBound}} 

Since $|S|=n_0$, we can assume that $S=[n_0]$ for the ease of exposition. Let $i\in\overline{S}$.  First, we decompose $L_i$ as follows 
\begin{equation*} 
    L_i = \sum_{j=1}^{n_0} f(x^*_i, x^*_j) \Big{(} d(Q^{-1}\hat{x}_i^{(2)}, x^*_j) - d(x^*_i, x^*_j)\Big{)}.
\end{equation*}
The regular grid $\Ccal_{n_0}$ is invariant by $Q$, and ${\bf x}^*_S$ belongs to $\bPi_{n_0}$. Besides, $\hat{\bf x}_S^{(1)}$ belongs to $\bPi_{n_0}$. As a consequence, we have   $\{x_j^*; \, j \in [n_0]\} = \Ccal_{n_0} =$ $\{Q^{-1}\hat{x}^{(1)}_j; \, j \in [n_0]\}$.
Hence, we can reorder the sum  in $L_i$ as follows
\begin{equation*} 
    L_i = \sum_{j=1}^{n_0} f(x^*_i, Q^{-1}\hat{x}^{(1)}_j) \Big{(} d(Q^{-1}\hat{x}_i^{(2)}, Q^{-1}\hat{x}^{(1)}_j) - d(x^*_i, Q^{-1}\hat{x}^{(1)}_j)\Big{)}\enspace .
\end{equation*}
To alleviate the notation, we write $\hat{z}_i := Q^{-1}\hat{x}_i^{(2)}$ and $z^{(1)}_j := Q^{-1}\hat{x}^{(1)}_j$
so that 
\begin{equation*} 
    L_i = \sum_{j=1}^{n_0} f(x^*_i, z_j^{(1)}) \Big{(} d(\hat{z}_i, z_j^{(1)}) - d(x^*_i, z_j^{(1)})\Big{)}\enspace .
\end{equation*}

\begin{lem}\label{ineq:demo:intermLem1}We have
\begin{align*}
\sum_{j=1}^{n_0} \Big{(}f(x^*_i, z_j^{(1)})-f(x^*_i, x^*_j)\Big{)} &\Big{(} d(\hat{z}_i, z_j^{(1)}) - d(x^*_i, z_j^{(1)})\Big{)}\\
&\leq  C_{L,e} d(\hat{x}_i^{(2)}, Qx^*_i)  \left[ n\mucal + \sqrt{n \log(n)}\right]\ .  
\end{align*}
\end{lem}
Gathering this lemma with the definition of $L_i$ leads us to
\begin{align}\label{ineq:demo:Li}
    L_i\leq  C_{L,e} d(\hat{x}_i^{(2)}, Qx^*_i)  &\left[n \mucal + \sqrt{n \log(n)})\right]\\ \nonumber
    &+ \sum_{j=1}^{n_0} f(x^*_i, x^*_j) \left( d(\hat{z}_i, z_j^{(1)}) - d(x^*_i, z_j^{(1)})\right)\enspace . \end{align} The orthogonal transformation $Q$ preserves the distances, hence the last term of \eqref{ineq:demo:Li} is equal to 
    \begin{equation}\label{proof:lem:lhoPerf:smallEq}
\sum_{j=1}^{n_0} f(x^*_i, x^*_j) \Big{(}d(\hat{z}_i, z_j^{(1)}) - d(x^*_i, z_j^{(1)})\Big{)} = \langle F_{i , S}, \, D(\hat{x}_i^{(2)}, \hat{\bf x}^{(1)}_S) - D(Qx_i^*, \hat{\bf x}^{(1)}_S) \rangle\enspace . 
\end{equation} 
To handle this term, we come back to the definition~\eqref{def:estim:Linfini} of $\hat{x}_i^{(2)}$. Since $Qx_i^*\in \Ccal_{n_0}$, we have 
$$\langle A_{i , S}, D(\hat{x}_i^{(2)}, \hat{\bf x}^{(1)}_S) \rangle \leq \langle A_{i ,S},  D(Qx_i^*, \hat{\bf x}^{(1)}_S) \rangle \enspace .$$ 
This yields
\begin{equation*}
 \langle F_{i , S},  D(\hat{x}_i^{(2)}, \hat{\bf x}^{(1)}_S) - D(Qx_i^*, \hat{\bf x}^{(1)}_S) \rangle \leq \langle E_{i , S}, D(Qx_i^*, \hat{\bf x}^{(1)}_S) - D(\hat{x}_i^{(2)}, \hat{\bf x}^{(1)}_S) \rangle \enspace .   
\end{equation*}
The right hand-side $\langle E_{i , S}, D(Qx_i^*, \hat{\bf x}^{(1)}_S) - D(\hat{x}_i^{(2)}, \hat{\bf x}^{(1)}_S) \rangle$ depends on $\hat{x}_i^{(2)}$ which belongs to $\Ccal_{n_0}$. This is why we simultaneously control the expression $ \langle E_{i , S},$ $D(Qx_i^*, \hat{\bf x}^{(1)}_S) - D(z, \hat{\bf x}^{(1)}_S) \rangle$ for all $z\in \Ccal_{n_0}$. This expression is distributed as a mean zero sub-Gaussian random variable with norm at most $C\|D(Qx_i^*, \hat{\bf x}^{(1)}_S) - D(z, \hat{\bf x}^{(1)}_S)\|_2$. Applying a union bound over all $z \in \Ccal_{n_0}$ leads us to 
\[
 \langle E_{i , S}, D(Qx_i^*, \hat{\bf x}^{(1)}_S) - D(\hat{x}_i^{(2)}, \hat{\bf x}^{(1)}_S) \rangle  \leq C\sqrt{ \log(n_0)}\|D(Qx_i^*, \hat{\bf x}^{(1)}_S) - D(\hat{x}_i^{(2)}, \hat{\bf x}^{(1)}_S)\|_2
\]
with probability higher than $1-1/n^3$.  Invoking the triangular inequality for the distance $d$,
we deduce that 
$\|D(Qx_i^*, \hat{\bf x}^{(1)}_S) - D(\hat{x}_i^{(2)}, \hat{\bf x}^{(1)}_S)\|_2 \leq  d(\hat{x}_i^{(2)}, Qx^*_i) \sqrt{n_0}$.
 It follows that, with probability at least $1-1/n^3$,  $$\langle F_{i ,S}, \, D(\hat{x}_i^{(2)}, \hat{\bf x}^{(1)}_S) - D(Qx_i^*, \hat{\bf x}^{(1)}_S) \rangle \leq C d(\hat{x}_i^{(2)}, Qx^*_i)\sqrt{n_0\log(n_0)}\enspace .$$
Gathering this bound with 
 \eqref{ineq:demo:Li} and \eqref{proof:lem:lhoPerf:smallEq} concludes the proof. \hfill $\square$

\bigskip

\noindent 
\begin{proof}[Proof of Lemma \ref{ineq:demo:intermLem1}] The first bi-Lipschitz condition \eqref{cond:lipsch} ensures that $$|f(x^*_i, z_j^{(1)})- f(x^*_i, x^*_j) |\leq c_L d(z_j^{(1)},x^*_j) + \ep_n \enspace.$$ By triangular inequality, we also have  $|d(\hat{z}_i, z_j^{(1)}) - d(x^*_i, z_j^{(1)}) |\leq d(\hat{z}_i, x^*_i)$ so that
\begin{align*}
    \sum_{j=1}^{n_0} \Big{(}f(x^*_i, z_j^{(1)})-f(x^*_i, x^*_j)\Big{)} &\Big{(} d(\hat{z}_i, z_j^{(1)}) - d(x^*_i, z_j^{(1)})\Big{)}\\
    &\leq  d(\hat{z}_i, x^*_i)  \sum_{j=1}^{n_0} (c_L d(z_j^{(1)},x^*_j) + \ep_n)\enspace .
\end{align*}
We have $d(z_j^{(1)},x^*_j) = d(\hat{x}^{(1)}_j,Qx^*_j)$ since $Q$ is an orthogonal transformation. Hence,  we obtain
\begin{align*}
   \sum_{j=1}^{n_0} \Big{(}f(x^*_i, z_j^{(1)})-f(x^*_i, x^*_j)\Big{)}& \Big{(} d(\hat{z}_i, z_j^{(1)}) - d(x^*_i, z_j^{(1)})\Big{)}\\
   &\leq  d(\hat{x}_i^{(2)}, Qx^*_i)  \left( c_L n_0\mucal + c_e \sqrt{n \log(n)} \right)\enspace .
\end{align*}
\end{proof}

    
\subsubsection{Proof of Lemma \ref{claim:LLowerB}.}

An interval $I=[a,b]$ denotes the set of points lying between $a$ and $b$ in the one-dimensional torus $\mathbb{R}/(2\pi)$, when following the trigonometric direction from $a$ to $b$. The length of $I$ is denoted by $|I|$. For any point $x$ in the sphere $\Ccal$, its argument in $[0,2\pi)$ is denoted by $\underline{x}$.

Since $|S|=n_0$, we can assume that $S=[n_0]$ for the ease of exposition. Let $i\in \overline{S}$  and denote $\hat{z}_i:= Q^{-1}\hat{x}_i^{(2)}.$ Since $d(x_i^*, \hat{z}_i) \leq \pi$, we can assume without loss of generality that the arguments $\underline{x}_i^*=0$ and $\underline{\hat{z}}_i \in (0,\pi]$, so that we have the equality $d(x_i^*, \hat{z}_i)= |\underline{x}_i^*- \underline{\hat{z}}_i|$. If $\hat{z}_i = x^*_i$, Lemma \ref{claim:LLowerB} is trivial. We therefore assume in the following that $\underline{\hat{z}}_i \in (0,\pi]$. Below, we introduce a partition of $[n_0]$ according to the relative positions of $x_j^*$, $x_i^*$ and $\hat{z}_i$. This partition is depicted in Figure~\ref{fig:chapSeriation:dessin}. 
\begin{align*}
I_1 &= \{ j \in [n_0]: \underline{x}^*_j \in [\underline{x}^*_i, \underline{\hat{z}}_i)\}\enspace ; \quad \quad 
I_2 = \{ j \in [n_0]: \underline{x}^*_j \in [\underline{\hat{z}}_i, \underline{x}_i^*+ \pi)\}\enspace ; \\ I_3 &= \{ j \in [n_0]: \underline{x}^*_j \in [\underline{x}_i^*+ \pi, \underline{\hat{z}}_i + \pi)\}\ ; \quad\quad
I_4 = \{ j \in [n_0]: \underline{x}^*_j \in [\underline{\hat{z}}_i + \pi, \underline{x}_i^*)\}\enspace .    
\end{align*}
\begin{figure}\centering \includegraphics[width=5cm]{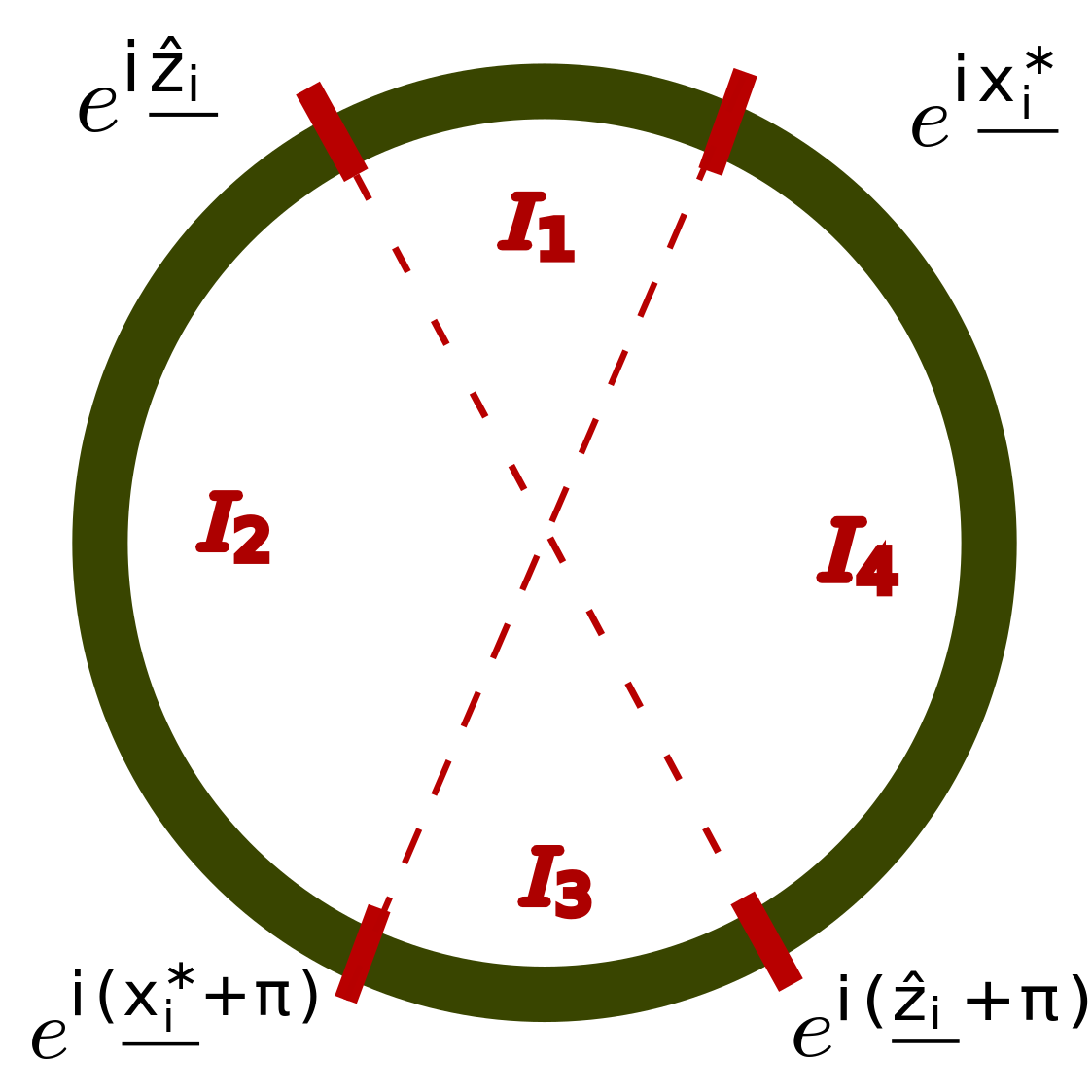}
\caption{Partition of $[n_0]$ in $I_1$--$I_4$}
\label{fig:chapSeriation:dessin}
\end{figure}
Although $I_s$ stands for a subset of indices, with a slight abuse of notation,  we still write $|I_s|$  for the length of the corresponding interval in $\mathbb{R}/(2\pi)$. For instance, $|I_1|:= |\underline{x}^*_i - \underline{\hat{z}}_i|$.

We decompose $L$ according to this partition of indices 
$ L_i = L_i^{(1)}+L_i^{(2)}+L_i^{(3)}+L_i^{(4)}$, where $L_i^{(s)}$ is the restriction of $L_i$ to the set $I_s$. In particular, if $\underline{\hat{z}}_i = \pi$, then the intervals $I_2$ and $I_4$ are empty, and $L_i^{(2)} = L_i^{(4)} = 0$. 

Next, we heavily rely   on the fact that the elements of ${\bf x}^*_{S}$ are evenly spaced on the sphere, that is $\{x_i^*:\, i\in[n_0]\} = \Ccal_{n_0}$ which holds true since we have assumed  ${\bf x}^*_{S} \in \bPi_{n_0}$. Using the symmetry of the set $\Ccal_{n_0}$, we establish below that the sums $L_i^{(2)}$ and $L_i^{(4)}$ nearly compensate so that $L_i^{(2)}+L_i^{(4)}$ admits a positive lower bound.

\begin{lem}\label{subCl:L2L4lowerB}We have 
$$ L_i^{(2)}+L_i^{(4)} \geq \frac{n_0|I_4|}{2\pi}\, c_ld^2(\hat{x}_i^{(2)}, Qx^*_i)  -  n_0d(\hat{x}_i^{(2)}, Qx^*_i) \ep_n  \enspace .$$
\end{lem}

As for $L_i^{(1)}$ (resp. $L_i^{(3)}$), we rely on the symmetry of $I_1$ (resp. $I_3$) around the point of $\Ccal$ whose argument is $(\underline{x}^*_i+\underline{\hat{z}}_i)/2$ (resp. $((\underline{x}^*_i+\underline{\hat{z}}_i)/2)+\pi$). 
\begin{lem}\label{subCl:L1L3lowerB}For some numerical constant $C>0$, we have
\begin{align*}L_i^{(1)}+L_i^{(3)} \geq C n \big{(}\frac{|I_1|}{4}-c_l^{-1}\ep_n\big{)}&\frac{d(\hat{x}_i^{(2)}, Qx^*_i)}{2}( c_l\frac{d(\hat{x}_i^{(2)}, Qx^*_i)}{2}-\ep_n)\\
&-(\pi c_l^{2})^{-1}c_e^3  \sqrt{\log^3(n)/n}- 2\frac{c_e}{c_l}\sqrt{\frac{\log(n)}{n}}\ .
\end{align*}
\end{lem}

By definition, $|I_1|+|I_4| = \pi,$ which yields the desired bound 
$$L_i  \geq C n d(\hat{x}_i^{(2)}, Qx^*_i) \Big{(}c_l d(\hat{x}_i^{(2)}, Qx^*_i) - \ep_n\Big{)} -\frac{c_e^3}{\pi c_l^{2}} \sqrt{\frac{\log^3(n)}{n}}- 2\frac{c_e}{c_l}\sqrt{\frac{\log(n)}{n}}\enspace .$$ 
\hfill $\square$

\begin{proof}[Proof of Lemma \ref{subCl:L2L4lowerB}]
 In Figure~\ref{fig:chapSeriation:dessin}, we can see that the difference $d(\hat{z}_i, x^*_j) -d(x^*_i , x^*_j)$ is equal to $-d(\hat{z}_i, x^*_i)$ for all $j \in I_2$, whereas it is equal to $d(\hat{z}_i, x^*_i)$ for $j\in I_4.$ Thus, we obtain 
\begin{eqnarray*}
L_i^{(2)} &= &\sum_{j\in I_2} f(x^*_i, x^*_j) \Big{(} d(\hat{z}_i, x^*_j) - d(x^*_i, x^*_j)\Big{)}= - d(\hat{z}_i, x^*_i) \sum_{j\in I_2} f(x^*_i, x^*_j) \enspace ; \\
L_i^{(4)}& = &\sum_{j\in I_4} f(x^*_i, x^*_j) \Big{(} d(\hat{z}_i, x^*_j) - d(x^*_i, x^*_j)\Big{)}= d(\hat{z}_i, x^*_i) \sum_{j\in I_4} f(x^*_i, x^*_j)\enspace .
\end{eqnarray*}
Let $\phi$ denote the reflection with respect to the line going through the two points of $\Ccal$ of arguments $$a = \frac{\underline{x}_i^* + \underline{\hat{z}}_i}{2} \quad \textup{  and  } \quad  b = \frac{(\underline{x}_i^*+\pi) + (\underline{\hat{z}}_i+\pi)}{2}.$$ As can be checked in Figure~\ref{fig:chapSeriation:dessin}, for any $l\in I_2$,  we have $\phi(x_j^*)= x_l^*$ for some $j$ in $I_4$. Hence, 
$$L_i^{(2)}+L_i^{(4)} = d(\hat{z}_i, x^*_i)  \sum_{j\in I_4} \big{(} f(x^*_i, x^*_j) - f(x^*_i, \phi(x^*_{j})) \big{)}\enspace .$$
To lower bound the difference in the sum , we invoke the bi-Lipschitz condition \eqref{cond:lipschLower}, which gives $$f(x^*_i, x^*_j) - f(x^*_i, \phi(x^*_{j})) \geq c_l (d(x^*_i, \phi(x^*_{j})) - d(x^*_i, x^*_j)) - \ep_n\enspace ,$$ 
since $x_j^*$ is closer to $x_i^*$ than $\phi(x_j^*)$ -- see again Figure~\ref{fig:chapSeriation:dessin}.
Also, we can check from Figure~\ref{fig:chapSeriation:dessin} that $d(x^*_i, \phi(x^*_{j})) - d(x^*_i, x^*_j) = d(\hat{z}_i, x^*_i)$ for all $j\in I_4$. Since $\Ccal_{n_0}$ is evenly spaced, the number of indices $j$ in $I_4$ is larger than $n_0|I_4|/(2\pi)$. This leads us to  
$$ L_i^{(2)}+L_i^{(4)} \geq  \frac{n_0|I_4|}{2\pi}\, d(\hat{z}_i, x^*_i)  c_l d(\hat{z}_i, x^*_i) -  n_0d(\hat{z}_i, x^*_i) \ep_n\enspace .$$
Since $d(\hat{z}_i, x^*_i) = d(\hat{x}_i^{(2)}, Qx^*_i)$, this concludes the proof. 
 
\end{proof}

\begin{proof}[Proof of Lemma \ref{subCl:L1L3lowerB}]
 From Figure~\ref{fig:chapSeriation:dessin}, we see that, for all $j \in I_{1}$, 
\[
d(\hat{z}_i, x^*_j) - d(x^*_i, x^*_j) = |\underline{\hat{z}}_i - \underline{x}^*_j| - |\underline{x}^*_i - \underline{x}^*_j| = \underline{\hat{z}}_i + \underline{x}^*_i - 2 \underline{x}^*_j \enspace .
\]
For $\alpha\in (0,1)$, write $I_{1}^{(\alpha)}$ the sub-interval of $I_1$ defined as $$I_{1}^{(\alpha)}=\left\{j \in [n_0]: \underline{x}^{*}_j \in \left[\underline{x}_i^*, (1-\alpha)\underline{x}_i^*+ \alpha \underline{\hat{z}}_i\right)\right\}\enspace .$$ In particular, for all $j\in I^{(1/2)}_1$, the above expression leads us to
$$d(\hat{z}_i, \phi(x^*_j)) - d(x^*_i, \phi(x^*_j)) = - \big{(} \underline{\hat{z}}_i + \underline{x}^*_i - 2 \underline{x^*_j}\big{)}\enspace ,$$
where $\phi$ is the symmetry introduced in the proof of Lemma~\ref{subCl:L2L4lowerB}. Hence, the terms with $j\in I^{(1/2)}_1$ partially compensate with the terms with $j$ outside $ I^{(1/2)}_1$.
\begin{align*}
    L_i^{(1)}& = \sum_{j\in I_1} f(x^*_i, x^*_j)\big[ d(\hat{z}_i, x^*_j) - d(x^*_i, x^*_j)\big]\\ &=  \sum_{j\in I_{1}^{(1/2)}} \big[f(x^*_i, x^*_j)-f(x^*_i, \phi(x^*_j))\big] \big[ \underline{\hat{z}}_i + \underline{x}^*_i - 2 \underline{x}^*_j\big]\enspace .
\end{align*}
For any $j\in I_1^{(1/2)}$, we have $d(x^*_i, \phi(x^*_j))\geq d(x^*_i, x^*_j)$. As a consequence, it follows from the  bi-Lipschitz condition \eqref{cond:lipschLower} that 
$$f(x^*_i, x^*_j)-f(x^*_i, \phi(x^*_j)) \geq c_l \left[d(x^*_i, \phi(x^*_j)) - d(x^*_i, x^*_j)\right] - \ep_n\enspace .$$
Since $d(x^*_i, \phi(x^*_j)) - d(x^*_i, x^*_j) = |\underline{\phi(x^*_j)}-\underline{x}^*_j|$ for all $j \in I_1^{(1/2)}$, we get
\begin{equation}\label{proof:subClaim:LowerB:L1}
L_i^{(1)} \geq  \sum_{j\in I_{1}^{(1/2)}} (c_l |\underline{\phi(x^*_j)}-\underline{x}^*_j|-\ep_n) (\underline{\hat{z}}_i + \underline{x}^*_i - 2 \underline{x}^*_j)\enspace .
\end{equation}
To control \eqref{proof:subClaim:LowerB:L1}, we split the interval $I_{1}^{(1/2)}$ according to the sign of the term $(c_l |\underline{\phi(x^*_j)}-\underline{x}^*_j|-\ep_n)$. That is, we write $I_{1}^{(1/2)} = I_{1}^{(1/2)-} \cup I_{1}^{(1/2)+}$ where $I_{1}^{(1/2)-}$ is the set of indices $j$ such that $c_l |\underline{\phi(x^*_j)}-\underline{x}^*_j| <\ep_n$.

\begin{cl}\label{subsubLBNeg}We have
\begin{align*}
    \sum_{j\in I_{1}^{(1/2)-}} (c_l |\underline{\phi(x^*_j)}-\underline{x}^*_j|-\ep_n) &(\underline{\hat{z}}_i + \underline{x}^*_i - 2 \underline{x}^*_j)\\
    &\geq - (c_l^{2}8\pi)^{-1} c_e^3 \sqrt{\log^3(n)/n}- \frac{c_e}{c_l}\sqrt{\frac{\log(n)}{n}}\enspace .
\end{align*}
\end{cl}

\begin{cl}\label{subsubLBPos}For some numerical constant $C>0$, we have
\begin{align*}
    \sum_{j\in I_{1}^{(1/2)+}} (c_l |\underline{\phi(x^*_j)}-\underline{x}^*_j|-\ep_n) &(\underline{\hat{z}}_i + \underline{x}^*_i - 2 \underline{x}^*_j) \\
    &\geq C n \big{(}\frac{|I_1|}{4}-c_l^{-1}\ep_n\big{)}\frac{|\underline{\hat{z}}_i - \underline{x}^*_i|}{2}( c_l\frac{|\underline{\hat{z}}_i - \underline{x}^*_i|}{2}-\ep_n)\enspace .
\end{align*}
\end{cl}

Gathering these two claims leads us to
\begin{align*}L_i^{(1)}\geq C n \big{(}\frac{|I_1|}{4}-c_l^{-1}\ep_n\big{)}&\frac{|\underline{\hat{z}}_i - \underline{x}^*_i|}{2}( c_l\frac{|\underline{\hat{z}}_i - \underline{x}^*_i|}{2}-\ep_n)\\
&-(c_l^{2}8\pi)^{-1}  c_e^3\sqrt{\log^3(n)/n}- \frac{c_e}{c_l}\sqrt{\frac{\log(n)}{n}}\ ,
\end{align*}which is the  desired bound since $|\underline{\hat{z}}_i - \underline{x}^*_i| = d(\hat{z}_i, x^*_i)= d(\hat{x}_i^{(2)}, Qx^*_i)$. By symmetry, the term $L_i^{(3)}$ is handled as $L_i^{(1)}$ and admits the same lower bound. 
 \end{proof}

\begin{proof}[Proof of Claim \ref{subsubLBNeg}]
For simplicity, the notation $\underline{x}$ is dropped out in the proof of Claim \ref{subsubLBNeg}, and $\underline{x}$ is simply denoted by $x$. By definition of $\phi$, we know that $(\hat{z}_i + x^*_i)/2 = (\phi(x^*_j)+ x^*_j)/2$ for all $j\in I_{1}^{(1/2)},$ which gives the equality $ \hat{z}_i + x^*_i - 2 x^*_j = \phi(x^*_j)- x^*_j$. Since $0 \leq \phi(x^*_j)-x^*_j < c_l^{-1} \ep_n$ for all $j\in I_{1}^{(1/2)-}$, we have
$$0 \leq \hat{z}_i + x^*_i - 2 x^*_j \leq c_l^{-1} \ep_n\enspace .$$
Since $c_l |\phi(x^*_j)-x^*_j|-\ep_n< 0$ for $j$ in $I_1^{(1/2)-}$, we obtain
$$(c_l |\phi(x^*_j)-x^*_j|-\ep_n) (\hat{z}_i + x^*_i - 2 x^*_j)\geq (c_l |\phi(x^*_j)-x^*_j|-\ep_n) c_l^{-1} \ep_n \geq -c_l^{-1} \ep_n^2\enspace. $$
Since the number of indices in $I_{1}^{(1/2)-}$ is at most $1+ n_0/(2\pi)|I_{1}^{(1/2)-}|$ (where $|I_{1}^{(1/2)-}|$ is the arc length), and the length of this arc is at most  $c_l^{-1}\, \ep_n$, we conclude that 
\begin{align*}
  \sum_{j\in I_{1}^{(1/2)-}} (c_l |\phi(x^*_j)-x^*_j|-\ep_n) (\hat{z}_i + x^*_i - 2 x^*_j) &\geq - \frac{n_0}{2\pi} c_l^{-2}  \ep_n^3\\
  &= -  \frac{c_e^3 }{8\pi c_l^{2}} \sqrt{\frac{\log^3(n)}{n}}- c_l^{-1}\epsilon_n\ .  
\end{align*}

 \end{proof}

\begin{proof}[Proof of Claim \ref{subsubLBPos}] 
Again, for convenience the notation $\underline{x}$ is dropped out here. Since all the terms in the sum are nonnegative, we can simply consider indices $j$ in $ I_{1}^{(1/2)+}\cap I_{1}^{(1/4)}$. Using $ \phi(x^*_j)- x^*_j= \hat{z}_i + x^*_i - 2 x^*_j$ for all $j \in I_{1}^{(1/2)}$ and $x_i^*=0$, we obtain that, for  $j \in I_{1}^{(1/4)}$, $\phi(x^*_j)- x^*_j\geq \hat{z}_i/2$.  This gives
$$(c_l |\phi(x^*_j)-x^*_j|-\ep_n) (\hat{z}_i + x^*_i - 2 x^*_j) \geq (c_l \frac{\hat{z}_i}{2}-\ep_n) \frac{\hat{z}_i}{2}\enspace ,$$ and, for some numerical constant $C > 0$, 
\begin{equation} \label{ineq:proof:subsubclaim:cpx}\sum_{j\in I_{1}^{(1/2)+}} (c_l |\phi(x^*_{j})-x^*_j|-\ep_n) (\hat{z}_i + x^*_i - 2 x^*_j) \geq  C n \big{|}I_{1}^{(1/2)+} \cap I_{1}^{(1/4)}\big{|}\frac{\hat{z}_i}{2}(c_l \frac{\hat{z}_i}{2}-\ep_n)\enspace .\end{equation}
Since either $I_{1}^{(1/2)+} \subset I_{1}^{(1/4)}$ or $ I_{1}^{(1/4)} \subset I_{1}^{(1/2)+}$ and $|I_{1}^{(1/4)} |= |I_{1} |/4$ and $|I_{1}^{(1/2)+}|= |I_{1}^{(1/2)}|-|I_{1}^{(1/2)-}| \geq |I_{1}^{(1/2)}|-c_l^{-1}\ep_n = \frac{|I_1|}{2}-c_l^{-1}\ep_n$, we deduce that 
\[
 |I_{1}^{(1/2)+} \cap I_{1}^{(1/4)}|\geq \frac{|I_1|}{4}-c_l^{-1}\ep_n\enspace .
\]
Thus, we have
\begin{equation*}  
\sum_{j\in I_{1}^{(1/2)+}} (c_l |\phi(x^*_j)-x^*_j|-\ep_n) (\hat{z}_i + x^*_i - 2 x^*_j) \geq C n \big{(}\frac{|I_1|}{4}-c_l^{-1}\ep_n\big{)}\frac{\hat{z}_i}{2}( c_l\frac{\hat{z}_i}{2}-\ep_n)\enspace .     
\end{equation*} Since $\hat{z}_i=| \hat{z}_i-x_i^*|$ (recall that $x_i^*=0$), this concludes the proof. 
\end{proof}

    
\subsection{Proof of Proposition~\ref{prop:LHO:perf}}\label{sect:proof:prop:LHO}

Let ${\bf x}^{**}_S$ be a best approximation of ${\bf x}^{*}_{S}$ in $\bPi_{n_0},$ that is, such that $d_{\infty}({\bf x}^{*}_{S},{\bf x}^{**}_S) = d_{\infty}({\bf x}^{*}_{S},\bPi_{n_0})$. As in the proof of Lemma~\ref{LHO:perf}, we restrict our attention to orthogonal transformations $Q\in \mathcal{O}$ that let $\mathcal{C}_{n_0}$ invariant. 
Fix $i$ in $\overline{S}$. To prove Proposition~\ref{prop:LHO:perf}, it suffices to establish variants of Lemmas \ref{claimLUpBound} and
\ref{claim:LLowerB} with 
\begin{equation}\label{eq:demo:Li:new}
\widetilde{L}_i := \langle F_{i , S}, \, D(Q^{-1}\hat{x}_i^{(2)}, {\bf x}^{**}_S) - D(x_i^*, {\bf x}^{**}_S) \rangle,\end{equation}
instead of $L_i$. In the definition of $\widetilde{L}_i$, ${\bf x}^*_S$ has been replaced by  ${\bf x}^{**}_{S}$.

\begin{lem}\label{claimLUpBoundVariante}
With probability at least $1-
1/n^3$, we have 
\[
\widetilde{L}_i\leq  C_{L,e} \left[1+   d(\hat{x}_i^{(2)}, Qx^*_i)\big{(}nd_{\infty}({\bf x}^{*}_{S},\bPi_{n_0}) + n\mucal+  \sqrt{n \log(n)}\,\big{)}\right]\enspace .
\]
\end{lem}

\begin{lem}\label{claim:LLowerBVariante}
 For $n$ large enough, one has
 \begin{eqnarray*}
  \widetilde{L}_i&\geq& C'_{l,e} n d(\hat{x}_i^{(2)}, Qx^*_i) \Big{(} d(\hat{x}_i^{(2)}, Qx^*_i) - \frac{\sqrt{\log(n)}}{n}\Big{)} -C''_{l,e} \sqrt{\frac{\log^3(n)}{n}} \\ && - C_{L,e} \left\{1+ n\left[d_{\infty}({\bf x}^{*}_{S},\bPi_{n_0}) +\sqrt{\frac{\log(n)}{n}}\right] d(\hat{x}_i^{(2)}, Qx^*_i)\right\}\enspace .
  \end{eqnarray*}
\end{lem}
These two lemmas enforce that, with probability higher than $1-1/n^3$, 
\begin{equation*}
 d(\hat{x}_i^{(2)}, Qx^*_i)\leq C_{l,L,e} \left[d_{\infty}({\bf x}^{*}_{S},\bPi_{n_0}) + \mucal+\sqrt{\log(n)/n}\right]\enspace . 
\end{equation*}
Indeed, assume that $d(\hat{x}_i^{(2)}, Qx^*_i)\geq C'_{l,L,e} \big{(}d_{\infty}({\bf x}^{*}_{S},\bPi_{n_0}) +\sqrt{\log(n)/n}\big{)} $ where  $C'_{l,L,e}$ is large enough. Then, Lemma~\ref{claim:LLowerBVariante} implies that  $\widetilde{L}_i \gtrsim_{c_l,c_L,c_e}  n d^2(\hat{x}_i^{(2)}, Qx^*_i)$. Together with Lemma~\ref{claimLUpBoundVariante}, we deduce that 
\[
 d(\hat{x}_i^{(2)}, Qx^*_i)\leq C_{l,L,e} \left[d_{\infty}({\bf x}^{*}_{S},\bPi_{n_0}) + \mucal  +\sqrt{\log(n)/n}\right]\ .
\]

In any case, we conclude that 
\[
 d(\hat{x}_i^{(2)}, Qx^*_i)\leq (C_{l,L,e}\vee  C'_{l,L,e})\left[ d_{\infty}({\bf x}^{*}_{S},\bPi_{n_0}) + \mucal  + \sqrt{\log(n)/n}\right]\ , 
\]
with probability higher than $1-1/n^3$. Taking the minimum over all $Q \in \Ocal$ that let $\mathcal{C}_{n_0}$ invariant and a union bound over all $i\in \overline{S}$, leads to Proposition~\ref{prop:LHO:perf}.

\subsubsection{Proof of Lemma \ref{claimLUpBoundVariante}}

To ease the exposition, we assume that $S=[n_0]$. Fix $i\in \overline{S}$. We start from 
\begin{equation} \label{def:li:proof:prp}
    \widetilde{L}_i = \sum_{j=1}^{n_0} f(x^*_i, x^*_j) \Big{(} d(Q^{-1}\hat{x}_i^{(2)}, x^{**}_j) - d(x^*_i, x^{**}_j)\Big{)}\enspace .
\end{equation}
In order to come back to the setting of Lemma~\ref{LHO:perf},  we replace $f(x^*_i, x^*_j)$ by $f(x^*_i, x^{**}_j)$, using the bi-Lipschitz condition~\eqref{cond:lipsch} so that 
$$f(x^*_i, x^*_j)- f(x^*_i, x^{**}_j)\leq  (c_L\vee c_e) \left[ d_{\infty}({\bf x}^{*}_{S},\bPi_{n_0}) + \sqrt{\frac{\log(n)}{n}}\right]\enspace .$$ 
By  triangular inequality, we have $d(Q^{-1}\hat{x}_i^{(2)}, x^{**}_j) - d(x^*_i, x^{**}_j)\leq d(Q^{-1}\hat{x}_i^{(2)},x^{*}
_i)$ which implies
  \begin{align*}
    \sum_{j=1}^{n_0} \Big{(}f(x^*_i, x^*_j)-f(x^*_i, x^{**}_j)\Big{)} & \Big{(} d(Q^{-1}\hat{x}_i^{(2)}, x^{**}_j) - d(x^*_i, x^{**}_j)\Big{)} \\
    &\leq (c_L\vee c_e) d(\hat{x}_i^{(2)},Qx^{*}_i) r_n \enspace ,
\end{align*}
where  we define $r_n =  n d_{\infty}({\bf x}^{*}_{S},\bPi_{n_0}) + \sqrt{n \log(n)}$. This leads us to  
\begin{equation*}
    \widetilde{L}_i\leq  (c_L\vee c_e)d(\hat{x}^{(2)}_i,Qx^{*}_i) r_n+ \sum_{j=1}^{n_0} f(x^*_i, x^{**}_j) \Big{(} d(Q^{-1}\hat{x}_i^{(2)}, x^{**}_j) - d(x^*_i, x^{**}_j)\Big{)}   \enspace .
\end{equation*}
Since $x^{**}_j$ now runs over $\bPi_{n_0}$, we can replace as in the proof of Lemma~\ref{claimLUpBound} the sum over  $x_j^{**}$ by a sum over $Q^{-1}\hat{x}^{(1)}_j$ using a suitable permutation:
\begin{align*}
      \widetilde{L}_i\leq  (c_L\vee c_e)&d(\hat{x}_i^{(2)},Qx^{*}_i) r_n \\
      &+ \sum_{j=1}^{n_0} f(x^*_i, Q^{-1}\hat{x}^{(1)}_j) \Big{(} d(Q^{-1}\hat{x}_i^{(2)}, Q^{-1}\hat{x}^{(1)}_j) - d(x^*_i, Q^{-1}\hat{x}^{(1)}_j)\Big{)}   \enspace.   
\end{align*}

The remainder of the proof follows the same lines as for Lemma \ref{claimLUpBound}, except for small differences. Still, we provide some details for the sake of completeness. as in that proof we write $\hat{z}_i= Q^{-1}\hat{x}_i^{(2)}$ and $z^{(1)}_i= Q^{-1}\hat{x}_i^{(1)}$. We first apply Lemma~\ref{ineq:demo:intermLem1} to obtain
\begin{align*}
    \widetilde{L}_i\leq  C_{L,e}d(\hat{x}^{(2)}_i,Qx^{*}_i) &\left[ n d_{\infty}({\bf x}^{*}_{S},\bPi_{n_0})+ n\mucal + \sqrt{n \log(n)} \right]\\
    &+ \sum_{j=1}^{n_0} f(x^*_i, x_j^*) \Big{(} d(\hat{z}_i, z_j^{(1)}) - d(x^*_i,z^{(1)}_j)\Big{)}   \enspace .    
\end{align*}
The last expression simplifies in 
\[
    \sum_{j=1}^{n_0} f(x^*_i, x_j^*) \Big{(} d(\hat{z}_i, z_j^{(1)}) - d(x^*_i,z^{(1)}_j)\Big{)}=\langle F_{i , S}, \, D(\hat{x}_i^{(2)}, \hat{\bf x}^{(1)}_S) - D(Qx_i^*, \hat{\bf x}^{(1)}_S) \rangle\enspace . 
\]
In Lemma \ref{claimLUpBound}, we had $Qx^*_i \in \Ccal_{n_0}$ so that we could use the definition of $\hat{x}_i^{(2)}$ to 
deduce that $\langle A_{i , S}, D(\hat{x}_i^{(2)}, \hat{\bf x}^{(1)}_S) \rangle \leq \langle A_{i ,S},  D(Qx_i^*, \hat{\bf x}^{(1)}_S) \rangle$.
 Unfortunately, $Qx^*_i$ does not necessarily belong to $\Ccal_{n_0}$ anymore. To handle this minor issue,  we replace $x^*_i$ by the closest element $y^*_i$ in $\Ccal_{n_0}$. It satisfies $d(x^*_i,y^*_i)\leq 2\pi/n_0$ and $Qy^*_i \in \Ccal_{n_0}$. This leads us to
\begin{align*}\langle A_{i ,S}, D(\hat{x}_i^{(2)}, \hat{\bf x}^{(1)}_S) \rangle &\leq \langle A_{i ,S},  D(Qy_i^*, \hat{\bf x}^{(1)}_S) \rangle \\ 
&\leq \langle A_{i ,S},  D(Qx_i^*, \hat{\bf x}^{(1)}_S) \rangle + \langle A_{i ,S}, D(Qy_i^*, \hat{\bf x}^{(1)}_S)-D(Qx_i^*, \hat{\bf x}^{(1)}_S)\rangle   \enspace .
\end{align*}
Since $|d(Qy_i^*,x_i)-d(Qx_i^*,x_i)| \leq 2\pi/n_0$ and $|f(x,y)|\leq 1$,  the above additional error term satisfies
\begin{align*}\langle A_{i ,S}, D(Qy_i^*, \hat{\bf x}^{(1)}_S)-D(Qx_i^*, \hat{\bf x}^{(1)}_S)\rangle &\leq \langle F_{i ,S}, D(Qy_i^*, \hat{\bf x}^{(1)}_S)-D(Qx_i^*, \hat{\bf x}^{(1)}_S)\rangle\\ &  + \langle E_{i ,S}, D(Qy_i^*, \hat{\bf x}^{(1)}_S)-D(Qx_i^*, \hat{\bf x}^{(1)}_S)\rangle \\
&\leq 2\pi +\\
&C' \sqrt{ \log(n)}\|D(Qy_i^*, \hat{\bf x}^{(1)}_S) - D(Qx_i^*, \hat{\bf x}^{(1)}_S)\|_2\\
&\lesssim 1\enspace ,
\end{align*}with probability higher than $1-1/(2n^3)$. Putting everything together, we have shown that 
\begin{align*}
    \widetilde{L}_i\leq  C_{L,e}d(\hat{x}^{(2)}_i,Qx^{*}_i)& \left[ n d_{\infty}({\bf x}^{*}_{S},\bPi_{n_0})+ n\mucal + \sqrt{n \log(n)} \right]\\
    &+ C + \langle E_{i , S}, \, D(\hat{x}_i^{(2)}, \hat{\bf x}^{(1)}_S) - D(Qx_i^*, \hat{\bf x}^{(1)}_S) \rangle\enspace ,
\end{align*}
with probability higher than $1-1/(2n^3)$. To conclude, it suffices to the rhs in the above expression. We do it exactly as in the end of the proof of Lemma~\ref{claimLUpBound} except that we now consider a probability $1-1/(2n^3)$.

\subsubsection{Proof of Lemma \ref{claim:LLowerBVariante}.}

Fix $i\in \overline{S}$ and  define $y^*_i \in \Ccal_{n_0}$ as a closest point to $x_i^*$ in $\Ccal_{n_0}$.
We introduce the quantity
\begin{equation*} 
    L'_i = \sum_{j=1}^{n_0} f(y^{*}_i, x^{**}_j)\big{(}d(Q^{-1}\hat{x}_i^{(2)} , x^{**}_j) - d(y^*_i, x^{**}_j)\big{)}\enspace ,
\end{equation*}
which has the same properties as the $L_i$ used in Lemma \ref{LHO:perf}, since each point involved in the expression of $L_i'$ is an element of $\Ccal_{n_0}$, and the sum runs over a vector ${\bf x}^{**}_S$ in $\bPi_{n_0}$.  This allows us to invoke Lemma \ref{claim:LLowerB} --from the proof of Lemma \ref{LHO:perf} -- to get
  $$L'_i \geq C n d(\hat{x}_i^{(2)} , Qy^*_i) \left[c_l d(\hat{x}_i^{(2)} , Qy^*_i) - \ep_n\right] -\frac{c_e^3}{\pi c_l^{2}} \sqrt{\frac{\log^3(n)}{n}}\enspace ,$$ for $n$ large enough. 

By definition of $y_i^*$, we know that 
$d(y^*_i, x_i^*) \leq 2\pi/n_0$. Hence, by triangular inequality, $d(\hat{x}_i^{(2)} , Qy^*_i) \geq  d(\hat{x}_i^{(2)}, Qx^*_i) - 2 \pi/n_0$ and  we derive that 
$$L_i' \geq C'_{l,e} n d(\hat{x}_i^{(2)}, Qx^*_i) \left[d(\hat{x}_i^{(2)}, Qx^*_i) - \sqrt{\frac{\log(n)}{n}}\right] -C''_{l,e} \sqrt{\frac{\log^3(n)}{n}}\enspace  .$$ 
Next, we rely on the following lemma to replace $L_i'$ by $\widetilde{L}_i$.
 \begin{lem}\label{cl:proof:prop:hlp}We have 
 \[
 |\widetilde{L}_i - L_i'|\leq  C_{L,e} \left[1+ \big{(}nd_{\infty}({\bf x}^{*}_{S},\bPi_{n_0}) +\sqrt{n\log(n)}\big{)} d(\hat{x}_i^{(2)}, Qx^*_i)\right] \ . 
 \]
\end{lem}
Gathering these two bounds completes the proof of Lemma \ref{claim:LLowerBVariante}. 
\hfill $\square$
 
 \bigskip
 
 \begin{proof}[Proof of Lemma \ref{cl:proof:prop:hlp}]
 From the definition of $y_i^*$ and the triangular inequality, we have $$ \Big{|}d(x^*_i, x^{**}_j)- d(y^*_i, x^{**}_j) \Big{|} \leq 2 \pi /n_0\enspace .$$  This allows us to deduce that the quantity 
\begin{equation} 
    L''_i = \sum_{j=1}^{n_0} f(x^*_i, x^*_j) \big{(}d(Q^{-1}\hat{x}_i^{(2)} , x^{**}_j) - d(y^*_i, x^{**}_j)\big{)} 
\end{equation} satisfies $|\widetilde{L}_i - L''_i| \leq 2\pi$.
Besides, we deduce from the bi-Lipschitz condition \eqref{cond:lipsch} and the triangular inequality that
\begin{align*}
    |f(x^*_i, x^*_j)-f(y^{*}_i, x^{**}_j)| &\leq c_L \left( d(x^{*}_i, y^{*}_i) +d(x^{*}_j, x^{**}_j)\right) +  2\ep_n\\
    &\leq  2c_L d_{\infty}({\bf x}^{*}_{S},\bPi_{n_0}) + 2\ep_n\enspace .
\end{align*}
Then, we deduce that 
$$|L'_i - L_i''|\leq 2n d(Q^{-1}\hat{x}_i^{(2)} , y^*_i) \big{(}c_Ld_{\infty}({\bf x}^{*}_{S},\bPi_{n_0}) + \ep_n\big{)} \enspace .$$
All in all, we have
\begin{align*}|\widetilde{L}_i - L_i'|&\leq |\widetilde{L}_i - L_i''| + |L_i'' - L_i'|\\
&\leq C_{L,e}\left[ 1 + d(Q^{-1}\hat{x}_i^{(2)} , y^*_i) \big{(}nd_{\infty}({\bf x}^{*}_{S},\bPi_{n_0})+\sqrt{n\log(n)}\big{)}\right]\enspace ,
\end{align*}
and the result follows.
\end{proof}

\subsection{Proof of Proposition \ref{l1bound}}\label{sec:proof:prop:l1bound:noconstr}
\subsubsection{Main Arguments}

Assume that $S=[n_0]$ for the ease of presentation. In this proof, we both interpret $\hat{\bf x}^{(1)}_S$ and ${\bf x}^{*}_{S}$ as vectors in $\mathcal{C}_n$ and matrices of size $2\times n_0$. We recall that $\|.\|_q$ refers to the entry-wise $l_q$ norm for matrices. 
We shall establish that the estimator $\hat{\bf x}^{(1)}_S$ is such that the matrix $\hat{\bf x}^{(1)T}_S\hat{\bf x}^{(1)}_S$ is close to ${\bf x}^{*T}_{S}{\bf x}^*_{S}$. In other words, the  distances between $\hat{ x}_1^{(1)},\ldots, \hat{ x}_{n_0}^{(1)}$ are close to the respective distances between the $x^*_1,\ldots,x^*_{n_0}$. Then, relying on a recent matrix perturbation result from~\cite{arias2020perturbation}, we deduce that, up to an orthogonal transformation,  $\hat{\bf x}^{(1)}_S$ and ${\bf x}^*_{S}$ are close. Let us  
first state this perturbation result.  Given any $p\times 2$ matrix  $M$ with real coefficients, we denote its transpose by $M^T$, and the Moore-Penrose pseudo-inverse by $M^{\dagger}$, and the usual operator norm by $\|M\|_{op}$. In this proof, the transformations $Q \in \mathcal{O}$ are interpreted as orthogonal matrices of  size $2 \times 2$. 
\begin{prop}[Theorem 1 in \cite{arias2020perturbation}]\label{thm:citationEry}
For any positive integer $p$ and any $p\times 2$ matrices $M$ and $N$, with $N$ having full rank, let $\nu = \|MM^T - NN^T\|_2$. Then, we have
$$\underset{Q \in \mathcal{O}}{\textup{min}} \, \|M - N Q\|_2 \lesssim \nu \|N^{\dagger}\|_{op}\enspace ,$$ as soon as  $2 \nu \|N^{\dagger}\|_{op}^2 \leq 1$.
\end{prop}

Let ${\bf x}^{**}_S\in \bPi_{n_0}$ denote a best approximation of ${\bf x}^{*}_{S}$ in $\bPi_{n_0},$  so that $d_{\infty}({\bf x}^{*}_{S},{\bf x}^{**}_S) = d_{\infty}({\bf x}^{*}_{S},\bPi_{n_0})$. In order to invoke the above proposition for $M= (\hat{\bf x}^{(1)}_S)^T$ and $N= ({\bf x}^{**}_S)^{T}$ in $\mathbb{R}^{n_0 \times 2}$, we need to check that the condition $2 \nu \|N^{\dagger}\|_{op}^2 \leq 1$ is fulfilled. First, to bound the term $\|N^{\dagger}\|_{op}^2$, we work out
\begin{align*}N^{\dagger} = (N^T N)^{-1} N^T =  \left({\bf x}^{**}_S {\bf x}_S^{**T}\right)^{-1}  {\bf x}^{**}_S &= \begin{pmatrix}
\sum_{i=1}^{n_0} ({\bf x}^{**}_{1i})^2 & 0 \\
0 & \sum_{i=1}^{n_0}({\bf x}^{**}_{2i})^2 
\end{pmatrix}^{-1} {\bf x}^{**}_S \\
&= \frac{2}{n_0} {\bf x}^{**}_S\enspace ,
\end{align*}
since $\sum_{i=1}^{n_0}({\bf x}^{**}_{1i})^2 = \sum_{i=1}^{n_0}({\bf x}^{**}_{2i})^2 =n_0/2$ -- see e.g.~\eqref{cosSum} for a proof. 
As a consequence, 
$$ \|N^{\dagger}  \|_{op}^2 \leq\|N^{\dagger}  \|_{2}^2 = \frac{4}{n_0}\ . $$
The following lemma bounds $\nu= \|\hat{\bf x}^{(1)T}_S\hat{\bf x}^{(1)}_S - {\bf x}^{**T} {\bf x}^{**}_S\|_2$.
\begin{lem}\label{cl:bound:nu:ery}
With probability at least $1-1/n^2$, we have
\[
\nu  \leq C_{l,L,e} \left(n  d_{\infty}({\bf x}^{*}_{S},\bPi_{n_0}) + \sqrt{n \log(n)}\right)\enspace .  
\]
\end{lem} 
Hence, with probability higher than $1-1/n^2$, we obtain 
$$2 \nu \|N^{\dagger}\|_{op}^2 \leq C'_{l,L,e}\left( d_{\infty}({\bf x}^{*}_{S},\bPi_{n_0}) + \sqrt{\frac{\log(n)}{n}}\right)\enspace .$$ 
If $d_{\infty}({\bf x}^{*}_{S},\bPi_{n_0}) + \sqrt{\log(n)/n}\geq 1/C'_{l,L,e}$, the conclusion of Proposition \ref{l1bound} obviously holds since $\min_{Q\in \mathcal{O}}d_{1}(\hat{\bf x}^{(1)}_S,Q{\bf x}^*_{S})\leq n_0 \leq n$. Hence, it suffices to consider the case where $d_{\infty}({\bf x}^{*}_{S},\bPi_{n_0}) + \sqrt{\log(n)/n}\leq 1/C'_{l,L,e}$ so that the  condition of Proposition~\ref{thm:citationEry} is fulfilled. This implies  $$\underset{Q \in \mathcal{O}}{\textup{min}} \, \|\hat{\bf x}^{(1)T}_S - {\bf x}^{**T} Q\|_2 \lesssim C \left( \sqrt{n}d_{\infty}({\bf x}^{*}_{S},\bPi_{n_0}) + \sqrt{\log(n)}\right)\enspace ,$$
and so $$\underset{Q \in \mathcal{O}}{\textup{min}} \, \|\hat{\bf x}^{(1)}_S - Q{\bf x}^{**}_S \|_1 \lesssim \sqrt{2} C\left( nd_{\infty}({\bf x}^{*}_{S},\bPi_{n_0}) + \sqrt{n\log(n)}\right)\enspace .$$

Since the distances in $\mathbb{R}^2$ and $\Ccal$ are equivalent, we have $\|Q{\bf x}^{**}_S - Q{\bf x}^{*}_{S}\|_1 \lesssim d_1(Q{\bf x}^{**}_S ,Q {\bf x}^{*}_{S}) = d_1({\bf x}^{**}_S , {\bf x}^{*}_{S}) $. Then, using the definition of $d_{\infty}({\bf x}^{*}_{S},\bPi_{n_0})$, we get $\|Q{\bf x}^{**}_S - Q{\bf x}^{*}_{S}\|_1 \lesssim nd_{\infty}({\bf x}^{*}_{S},\bPi_{n_0})$. Together with the triangular inequality, this leads us to
$$\underset{Q \in \mathcal{O}}{\textup{min}} \, \|\hat{\bf x}^{(1)}_S - Q{\bf x}^{*}_{S} \|_1 \leq C_{l,L,e}\left[ nd_{\infty}({\bf x}^{*}_{S},\bPi_{n_0}) + \sqrt{n\log(n)}\right]\enspace .$$
Using again the equivalence between the distances, that is $d(x,y) \lesssim \|x-y\|_1$ for all $x, y\in \Ccal$, we conclude that 
$$\underset{Q \in \mathcal{O}}{\textup{min}}d_{1}(\hat{\bf x}^{(1)}_S,Q{\bf x}_S^*) \leq C \underset{Q \in \mathcal{O}}{\textup{min}} \, \|\hat{\bf x}^{(1)}_S - Q{\bf x}^{*} \|_1 \leq C_{l,L,e}\left[ nd_{\infty}({\bf x}^{*}_{S},\bPi_{n_0}) + \sqrt{n\log(n)}\right]\enspace ,$$
and the proof of Proposition \ref{l1bound} is complete.

\subsubsection{Proof of Lemma \ref{cl:bound:nu:ery}}

Both $\hat{\bf x}^{(1)}_S$ and ${\bf x}^{**}_S$ are elements of $\bPi_{n_0}\subset \mathbb{R}^{2\times n_0}$, hence they both satisfy  $\hat{\bf x}^{(1)}_S 1 = 0$ and ${\bf x}^{**}_S1 =0$ where $1$ denotes the vector of ones. Indeed, since $\sum_{k=0}^{n_0-1}e^{\iota 2\pi k/n_0}=0$, we have $\sum_{k=0}^{n_0-1} \cos(2\pi k/n_0)= 0$ and $\sum_{k=0}^{n_0-1} \sin(2\pi k/n_0)= 0$. We can then invoke the next lemma  to bound $\nu= \|\hat{\bf x}^{(1)T}_S\hat{\bf x}^{(1)}_S - {\bf x}^{**T} {\bf x}^{**}_S\|_2.$

\begin{lem}\label{cl:distAuPosi} For any $Z = (z_1,\ldots,z_{n_0})$ and $Z' = (z'_1,\ldots,z'_{n_0})$ in $\mathbb{R}^{2 \times n_0}$ with $Z1=Z'1=0$, let $D = (D_{ij})$ and $D'=(D_{ij}')$ be their (squared) distance matrices, that is $D_{ij} =\|z_i - z_j\|_2^2$ and $D_{ij}' = \|z'_i - z'_j\|_2^2$ for all $i,j\in[n_0].$ Then we have 
$$\|Z^T Z - Z'^{T}Z'\|_2 \leq \| D- D' \|_2\enspace .$$
\end{lem}

For $Z={\bf x}^{**}_S$ and $Z'=\hat{\bf x}^{(1)}_S$, and accordingly $D_{ij} =\|x^{**}_i - x^{**}_j\|_2^2$ and $D_{ij}' = \|\hat{x}_i^{(1)} - \hat{x}_j^{(1)}\|_2^2$, it follows from Lemma \ref{cl:distAuPosi} that 
$\nu  \leq \| D-D'  \|_2$. Since all square distances $D_{ij}$ and $D'_{ij}$ are at most equal to $4$, we get $$\nu  \leq 4 \| \sqrt{D}-\sqrt{D'}  \|_2\enspace , $$ where $\sqrt{D}$ and $\sqrt{D'}$ denote the matrices of coefficients $\sqrt{D_{ij}} =\|x^{**}_i - x^{**}_j\|_2$ and $\sqrt{D_{ij}'} = \|\hat{x}_i^{(1)} - \hat{x}_j^{(1)}\|_2$.

For any $x,y\in \Ccal$, elementary geometry gives $\|x-y\|_2 = 2 \sin(d(x,y)/2)$. Since the sinus  function is $1$-Lipschitz,  we have 
\[
\big{|}\, \|x-y\|_2-\|x'-y'\|_2 \big{|}\leq |d(x,y) - d(x',y')|\enspace , 
\]
for any $x,y,x',y'\in \Ccal$. Hence, we deduce that  $\nu  \leq 4 \| D({\bf x}^{**}_S) - D(\hat{\bf x}^{(1)}_S)  \|_2$, 
where $D({\bf x}^{**}_S)$ and  $D(\hat{\bf x}^{(1)}_S)$ respectively denote the matrices of coefficients $d(x^{**}_i,x^{**}_j)$ and $d(\hat{x}_i^{(1)},\hat{x}_j^{(1)})$. As a consequence, we mainly have to control with high probability $\| D({\bf x}^{**}_S) - D(\hat{\bf x}^{(1)}_S) \|_2$.

\begin{lem}\label{cl:estimDist}With probability at least $1-1/n^2$, we have
$$\| D({\bf x}^{**}_S) - D(\hat{\bf x}^{(1)}_S) \|_2 \leq C'_{l,L,e}  \left[n  d_{\infty}({\bf x}^{*}_{S},\bPi_{n_0}) + \sqrt{n \log(n)}\right]\enspace .$$
\end{lem}

Hence $\nu \leq C_{l,L,e} [n  d_{\infty}({\bf x}^{*}_{S},\bPi_{n_0}) + \sqrt{n \log(n)}]$ and the proof of Lemma \ref{cl:bound:nu:ery} is complete. \hfill $\square$

\bigskip
\begin{proof}[Proof of Lemma \ref{cl:distAuPosi}] Let $H= I - J/n_0$, where $I$ is the identity and $J$ the matrix of ones. Since $Z1=0$, we have $ZH = Z$, so that $$Z^TZ = HZ^TZH = -\frac{1}{2}HDH,$$since $D$ is the matrix of distances associated with $Z$. Then we have 
$$\|Z^TZ - Z'^TZ'\|_2 = \frac{1}{2}\|H(D-D')H\|_2 \leq \frac{1}{2} \|D-D'\|_2\enspace ,$$ 
where the last inequality derives from the general relation $\|AB\|_2 \leq \|A\|_{op} \|B\|_2$ for any matrices $A,B,$ and the fact that $\|H\|_{op} =1$ -- because $H$ is an orthogonal projection. Lemma \ref{cl:distAuPosi} is proved.
\end{proof}

\begin{proof}[Proof of Lemma \ref{cl:estimDist}]
 First, we come back to the definition of the estimator $\hat{\bf x}^{(1)}_S$ defined in~\eqref{def:estim:l1bound}. 
We have   $\langle A_{S ,S}, D(\hat{\bf x}^{(1)}_S)\rangle \leq \langle A_{S ,S}, D({\bf x}^{**}_S)\rangle$, which implies that
\begin{equation*}
 \langle F_{S ,S},  D(\hat{\bf x}^{(1)}_S) - D({\bf x}^{**}_S) \rangle \leq \langle E_{S ,S}, D({\bf x}^{**}_S) - D(\hat{\bf x}^{(1)}_S) \rangle.   
\end{equation*} 

As in the last lines of the proof of Lemma \ref{claimLUpBound}, we bound the term  $\langle E_{S ,S}, D({\bf x}^{**}_S) - D(\hat{\bf x}^{(1)}_S) \rangle$ by a union bound over all possible vectors $\hat{\bf x}^{(1)}_S$. Hence, we get
\begin{align}\label{eq:union_bound}
\langle F_{S ,S},  D(\hat{\bf x}^{(1)}_S) - D({\bf x}^{**}_S)\rangle &\leq \langle E_{S ,S}, D({\bf x}^{**}_S) - D(\hat{\bf x}^{(1)}_S) \rangle \nonumber \\
&\lesssim \sqrt{n \log(n)} \| D({\bf x}^{**}_S) - D(\hat{\bf x}^{(1)}_S) \|_2\enspace ,
 \end{align}
 with probability at least $1-1/n^2$.
Conversely, we shall lower bound $\langle F_{S,S},  D(\hat{\bf x}^{(1)}_S) - D({\bf x}^{**}_S) \rangle$. 
\[
   \langle F_{S,S},  D(\hat{\bf x}^{(1)}_S) - D({\bf x}^{**}_S) \rangle = \sum_{i,j=1}^{n_0} f(x^*_i, x^*_j) \Big{(} d(\hat{x}_i^{(1)}, \hat{x}_j^{(1)}) - d(x^{**}_i, x^{**}_j)\Big{)} \ . 
\]
Using the bi-Lipschitz property of the function $f$, we deduce that 
\begin{align*}
 \big{|}f(x^*_i, x^*_j) - f(x^{**}_i, x^{**}_j) \big{|} &\leq c_L\left( d(x^{*}_i,x^{**}_i) + d(x^{*}_j,x^{**}_j) \right)+ 2c_e \sqrt{\log(n)/n}\\
 &\leq2 c_L d_{\infty}({\bf x}^{*}_{S},\bPi_{n_0}) + 2c_e \sqrt{\log(n)/n}\enspace ,
\end{align*}
by definition of ${\bf x}^{**}$. Then, we get 
\begin{eqnarray*}    
\lefteqn{
    \sum_{i,j=1}^{n_0}  \big{|}f(x^*_i, x^*_j) - f(x^{**}_i, x^{**}_j) \big{|} \big{|} d(\hat{x}_i^{(1)}, \hat{x}_j^{(1)}) - d(x^{**}_i, x^{**}_j)\big{|}}&&\\ &&\leq C_{L,e}   \left(d_{\infty}({\bf x}^{*}_{S},\bPi_{n_0})+\sqrt{\frac{\log(n)}{n}}\right)\, \| D({\bf x}^{**}_S) - D(\hat{\bf x}^{(1)}_S) \|_1\\
    && \leq C_{L,e}  n_0\left(d_{\infty}({\bf x}^{*}_{S},\bPi_{n_0})+\sqrt{\frac{\log(n)}{n}}\right) \| D({\bf x}^{**}_S) - D(\hat{\bf x}^{(1)}_S) \|_2\enspace \ , 
\end{eqnarray*}
where we applied Cauchy-Schwarz inequality on $\mathbb{R}^{n_0\times n_0}$. 
As a consequence, 
\begin{align}\label{eq:1_l1}
    \langle F_{S,S},  &D(\hat{\bf x}^{(1)}_S) - D({\bf x}^{**}_S) \rangle 
    \geq \sum_{i,j=1}^{n_0} f(x^{**}_i, x^{**}_j) \Big{(} d(\hat{x}_i^{(1)}, \hat{x}_j^{(1)}) - d(x^{**}_i, x^{**}_j)\Big{)} \nonumber \\
    &- C_{L,e}  n\left(d_{\infty}({\bf x}^{*}_{S},\bPi_{n_0})+\sqrt{\frac{\log(n)}{n}}\right)\, \| D({\bf x}^{**}_S) - D(\hat{\bf x}^{(1)}_S) \|_2\enspace . 
\end{align}
The following result bounds $\sum_{i,j=1}^{n_0} f(x^{**}_i, x^{**}_j) [d(\hat{x}_i^{(1)}, \hat{x}_j^{(1)}) - d(x^{**}_i, x^{**}_j)]$ in terms of the Frobenius norm $\|D({\bf x}^{**}_S)-D(\hat{\bf x}^{(1)}_S) \|_2^2$. This is a key step in our proof. Had the slack constant $c_e$ been equal to zero, the following result would have been a consequence of Lemma~\ref{lem:conerstone:variante}. Here the proof is slightly more involved and is provided below.

\begin{lem}\label{lem:regret}
 We have 
 \begin{align*}
  \sum_{i,j=1}^{n_0} f(x^{**}_i, x^{**}_j) \Big{(} d(\hat{x}_i^{(1)}, \hat{x}_j^{(1)}) - d(x^{**}_i, x^{**}_j)&\Big{)}\geq \frac{c_l}{2}\|D({\bf x}^{**}_S)-D(\hat{\bf x}^{(1)}_S) \|_2^2 \\
  &- c_e\sqrt{n\log(n)}\|D({\bf x}^{**}_S)-D(\hat{\bf x}^{(1)}_S) \|_2 . 
 \end{align*}

\end{lem}

We conclude from~\eqref{eq:union_bound} and the above lemma that 
\begin{align*} 
   C \sqrt{n \log(n)} \| D({\bf x}^{**}_S) &-  D(\hat{\bf x}^{(1)}_S) \|_2   \geq 
    \frac{c_l}{2}\| D({\bf x}^{**}_S) - D(\hat{\bf x}^{(1)}_S) \|^2_2 \\
   &- C_{L,e}   \left(n d_{\infty}({\bf x}^{*}_{S},\bPi_{n_0})+\sqrt{n\log(n)}\right) \| D({\bf x}^{**}_S) - D(\hat{\bf x}^{(1)}_S) \|_2 
\end{align*}
which in turn implies that 
$$
   \| D({\bf x}^{**}_S) - D(\hat{\bf x}^{(1)}_S) \|_2 \leq C'_{l,L,e}\left[  n  d_{\infty}({\bf x}^{*}_{S},\bPi_{n_0}) + \sqrt{n \log(n)}\right]\enspace .$$ Lemma \ref{cl:estimDist} is proved. 
 
\end{proof}

\begin{proof}[Proof of Lemma~\ref{lem:regret}]

To alleviate the notation, we introduce\\ $\gamma_i= \sum_{j=1}^{n_0} f(x^{**}_i, x^{**}_j) [ d(\hat{x}_i^{(1)}, \hat{x}_j^{(1)}) - d(x^{**}_i, x^{**}_j)]$  
so that we aim at establishing a lower bound for each $\gamma_i$ and in turn for $\gamma=\sum_{i=1}^{n_0} \gamma_i$. To simplify the arguments, we only consider the case where $n_0$ is odd, the case of $n_0$ even being almost similar.

Both ${\bf x}^{**}_S$ and $\hat{\bf x}^{(1)}_S$ belongs to $\bPi_{n_0}$ and we shall  heavily rely  on the symmetries of $\bPi_{n_0}$. Assume without loss of generality that $i=1$ and $x^{**}_j=e^{\iota 2\pi (j-1)/n_0}$ for all $j=1,\ldots, n_0$. Then, $d(x^{**}_i, x^{**}_j)= \tfrac{2\pi}{n_0}[|j-1|\wedge |n_0-j+1|]$. Since  $\hat{\bf x}^{(1)}_S$ also belongs to $\bPi_{n_0}$,  there exists a permutation $\sigma$ of $[n_0]$ such that $\sigma(1)=1$ and $d(\hat{x}_i^{(1)}, \hat{x}_j^{(1)})= \frac{2\pi}{n_0} [|\sigma(j)- 1| \wedge |n_0+1- \sigma(j)|]$.
Recall that we consider the case where $n_0$ is odd. Besides, we can focus on $n_0$ larger than $3$ since Lemma~\ref{lem:regret} is trivial for $n_0=1$. Thus, there exists a surjective map $\overline{\sigma}: [n_0-1] \mapsto [\lfloor n_0/2\rfloor ]$ such that $|\overline{\sigma}^{-1}(\{z\})|=2$ for any $z\in [\lfloor n_0/2\rfloor ]$ and $d(\hat{x}_i^{(1)}, \hat{x}_j^{(1)}) = \frac{2\pi }{n_0}\overline{\sigma}(j-1)$ for any $j=2,\ldots n_0$.
Finally, we write $\psi_j= f(1, e^{\iota 2\pi j/n_0})$ and $\psi'_j= f(1, e^{-\iota 2\pi j/n_0})$ for $j=1,\ldots \lfloor n_0/2\rfloor$. 
Equipped with this new notation, we arrive at
\[
 \gamma_i= \frac{2\pi}{n_0} \sum_{j=1}^{\lfloor n_0/2\rfloor } \psi_j[\overline{\sigma}(j) - j]+ \psi'_j[\overline{\sigma}(n_0-j)-j]\enspace .
\]
Finally, we denote
$a_j= \overline{\sigma}(j) - j$ and $a'_j=\overline{\sigma}(n_0-j)-j$ for $j= 1,\ldots, \lfloor n_0/2\rfloor$. Obviously, we have $\sum_{j=1}^{\lfloor n_0/2\rfloor } a_j+a'_j= 0$. More generally, one easily checks that, for any positive integer $s\leq \lfloor n_0/2\rfloor$,
the sum $\sum_{j=1}^s(a_j+a'_j)$ is nonnegative. Starting from
\[
 \gamma_i= \frac{2\pi}{n_0} \sum_{j=1}^{\lfloor n_0/2\rfloor } \psi_ja_j+ \psi'_ja'_j
\]
we partition the indices according to the signs of $a_j$ and $a'_j$. 
Define $A_+=\{j\in [\lfloor n_0/2\rfloor]:  a_j\geq 0\}$, $A_-=\{j\in [\lfloor n_0/2\rfloor]:  a_j<0\}$,  $A'_+=\{j\in [\lfloor n_0/2\rfloor]:  a'_j\geq 0\}$, and  $A'_-=\{j\in [\lfloor n_0/2\rfloor]:  a'_j<0\}$. Intuitively, we want to group indices $j$ such that $a_j>0$ with indices $k$ such that $a_k<0$. This can be done by recursion. First, consider the smallest index $k\in A_{-}\cup A'_{-}$. By symmetry, suppose that $a_k <0$. Since $\sum_{j=1}^k (a_j+a'_j)\geq 0$, this implies that $\sum_{j=1}^{k} \mathbf{1}_{j\in A_+}a_j +  \mathbf{1}_{j\in A'_+}a_j'\geq |a_k| + \mathbf{1}_{k\in A'_-}|a_k'|$. Hence, it is possible to build nonnegative numbers $b_{j,k,1}\leq a_{j}$ for $j\in A_+\cap [k]$ and $b'_{j,k,1}\leq a'_j$ for $j\in A'_+\cap [k]$ such that $\sum_{j=1}^{k} \mathbf{1}_{j\in A_+}b_{j,k,1} +  \mathbf{1}_{j\in A'_+}b'_{j,k,1}= |a_k|$.  Iterating the construction we obtain the following decomposition 
\begin{eqnarray*}
\frac{n_0}{2\pi} \gamma_i&= & \sum_{j\in A_{+}}\left(\sum_{k\in A_-} (\psi_j- \psi_k) b_{j,k,1}+\sum_{k\in A'_-} (\psi_j- \psi'_k) b_{j,k,2} \right) \\ && +  \sum_{j\in A'_{+}}\left(\sum_{k\in A_-} (\psi'_j- \psi_k) b'_{j,k,1}+\sum_{k\in A'_-} (\psi'_j- \psi'_k) b'_{j,k,2} \right)\enspace ,
\end{eqnarray*}
where all $b_{j,k,t}$'s are nonnegative, $b_{j,k,t}=0$ for $k< j$,  and
\[
 \left\{
 \begin{array}{cc}
\sum_{k\in A_{-}}b_{j,k,1}+ \sum_{k\in A'_{-}}b_{j,k,2}= a_{j}\text{ for }j\in A_+\ ;  \\
\sum_{k\in A_{-}}b'_{j,k,1}+ \sum_{k\in A'_{-}}b'_{j,k,2}= a'_{j} \text{ for }j\in A'_+\ ;  \\
\sum_{j\in A_+}b_{j,k,1}+ \sum_{j\in A'_+}b'_{j,k,1}= -a_{k}\text{ for } k\in A_{-}\ ; \\
\sum_{j\in A_+}b_{j,k,2}+ \sum_{j\in A'_+}b'_{j,k,2}= -a'_{k}\text{ for } k\in A'_{-} \ . 
 \end{array}
 \right.
\]

In the above decomposition all the  terms $b_{j,k,1}$, $b_{j,k,2}$, $b'_{j,k,1}$, and $b'_{j,k,2}$ are nonnegative. Besides, they are positive only when $k\geq j$, so that we can use the bi-Lipschitz condition~\eqref{cond:lipschLower}
\[
 (\psi_j- \psi'_k)= f(1,e^{\iota 2\pi (j-1)/n_0})- f(1,e^{-\iota 2\pi (k-1)/n_0})\geq c_l\frac{2\pi(k-j)}{n_0}- c_{e}\sqrt{\frac{\log(n)}{n}}\ .
\]
We obtain similarly the same lower bound for $\psi_j- \psi_k$, $\psi'_j- \psi_k$, and $\psi'_j- \psi'_k$. Coming back to the expression $\gamma_{j}$ and the definition of the $b_{i,j,t}$ with $t=1,2$ yields
\begin{eqnarray*}
 \frac{n_0}{2\pi}\gamma_i&\geq  & c_l\frac{2\pi}{n_0}  \sum_{j=1}^{\lfloor n_0/2\rfloor } -j \left[a_j+ a'_j\right] - c_e\sqrt{\frac{\log(n)}{n}}\sum_{j=1}^{\lfloor n_0/2\rfloor }|a_j|+|a'_j|\\
 & \geq & -c_l \frac{2\pi}{n_0} \sum_{j=1}^{\lfloor n_0/2\rfloor } \left[j(\overline{\sigma}(j)-j) + j(\overline{\sigma}(n_0-j) - j) \right]\\
& & - c_e\sqrt{\frac{\log(n)}{n}}\sum_{j=1}^{\lfloor n_0/2\rfloor }|\overline{\sigma}(j)-j| + |\overline{\sigma}(n_0-j) - j|\ . 
\end{eqnarray*}
Let us work out these two expressions in the rhs. By symmetry and definition of $\overline{\sigma}$ and $\sigma$ we get
\begin{align*}
\sum_{j=1}^{\lfloor n_0/2\rfloor } -&j(\overline{\sigma}(j)-j) - j(\overline{\sigma}(n_0-j) - j)\\
&= \frac{1}{2}\sum_{j=1}^{\lfloor n_0/2\rfloor } (\overline{\sigma}(j)-j)^2 + (\overline{\sigma}(n_0-j) - j)^2 \\
&= \frac{n_0^2}{8\pi^2}\sum_{j=1}^{n_0} \left[d(\hat{x}_i^{(1)}, \hat{x}_j^{(1)}) - d(x^{**}_i, x^{**}_j)\right]^2\enspace . 
\end{align*}
Similarly, we get 
\[
\sum_{j=1}^{\lfloor n_0/2\rfloor }|\sigma(j)-j| + |\overline{\sigma}(n_0-j) - j|= \frac{n_0}{2\pi}\sum_{j=1}^{n_0} |d(\hat{x}_i^{(1)}, \hat{x}_j^{(1)}) - d(x^{**}_i, x^{**}_j)|\enspace . 
\]
Putting everything together yields
\begin{align*}
  \gamma_i\geq &\frac{c_l}{2} \sum_{j=1}^{n_0} \left[d(\hat{x}_i^{(1)}, \hat{x}_j^{(1)}) - d(x^{**}_i, x^{**}_j)\right]^2 \\
  &-  c_e\sqrt{\frac{\log(n)}{n}}\sum_{j=1}^{n_0}  |d(\hat{x}_i^{(1)}, \hat{x}_j^{(1)}) - d(x^{**}_i, x^{**}_j)|\enspace ,
\end{align*}
which in turn allows us to conclude 
\begin{eqnarray*}
 \gamma &\geq& \frac{c_l}{2}\|D({\bf x}^{**}_S)-D(\hat{\bf x}^{(1)}_S) \|_2^2 - c_e\sqrt{\log(n)/n}\|D({\bf x}^{**}_S)-D(\hat{\bf x}^{(1)}_S) \|_1 \\
 &\geq & \frac{c_l}{2}\|D({\bf x}^{**}_S)-D(\hat{\bf x}^{(1)}_S) \|_2^2 - c_e\sqrt{n\log(n)}\|D({\bf x}^{**}_S)-D(\hat{\bf x}^{(1)}_S) \|_2 \ . 
\end{eqnarray*}
Lemma~\ref{lem:regret} is proved.  
\end{proof}

\subsection{Proof of Theorem \ref{thm:principal} and \ref{cor:thm:bis}} \label{subsec:proof:gluing}


In this section, we prove  Theorem~\ref{cor:thm:bis}. Theorem~\ref{thm:principal} then  follows directly from this result.

\subsubsection{Main arguments}

Recall that $n = 4 n_0$. For $n_0\leq 3$, the bound of Theorem~\ref{cor:thm:bis} is trivially true. Assume that $n_0 \geq 4$ in the following. In Step 1 of the main procedure, it follows from Propositions~\ref{prop:LHO:perf} and \ref{l1bound} that  the output $\hat{{\bf x}}_{\overline{S}}^{(2)}$ satisfies the following uniform bound 
\begin{equation}\label{roadmap:localerror}
    \min_{Q\in \mathcal{O}}d_{\infty}(\hat{{\bf x}}^{(2)}_{\overline{S}},Q{\bf x}^*_{\overline{S}}) \leq C_{l,L,e} \left[ d_{\infty}({\bf x}^{*}_{S},\bPi_{n_0})+  \sqrt{\frac{\log(n)}{n}}\right] \enspace,
\end{equation} with probability higher than $1-2/n^2$. Similarly, for the output $\hat{{\bf x}}^{(2')}_{\overline{S}'}$ in Step 2, we have
\begin{equation}\label{roadmap:localerror:bis}
    \min_{Q\in \mathcal{O}}d_{\infty}(\hat{{\bf x}}^{(2')}_{\overline{S}'},Q{\bf x}^*_{\overline{S}'}) \leq C_{l,L,e} \left[d_{\infty}({\bf x}^{*}_{S'},\bPi_{n_0})+  \sqrt{\frac{\log(n)}{n}}\right]\enspace ,
\end{equation}with probability higher than $1-2/n^2$. In \eqref{roadmap:localerror} and \eqref{roadmap:localerror:bis}, we shall prove that the bias terms $d_{\infty}({\bf x}^{*}_{S},\bPi_{n_0})$ and $d_{\infty}({\bf x}^{*}_{S'},\bPi_{n_0})$  are of the same order as $d_{\infty}({\bf x}^*,\bPi_{n})$  up to an additional error of the order of $\sqrt{\log(n)/n}$ -- see Lemma \ref{cl:equivBias:bis} below.

\begin{lem}\label{cl:equivBias:bis} Assume that $n_0\geq 4$ and fix ${\bf x}^* \in \Ccal^{ n}$. There exists an event of probability higher than $1-1/n^2$ such that 
\begin{align*}
     &d_{\infty}({\bf x}^{*}_{S},\bPi_{n_0}) \leq  C\left[d_{\infty}({\bf x}^*,\bPi_{n}) +  \sqrt{\frac{\log(n)}{n}}\right]\ ;\\
     & d_{\infty}({\bf x}^{*}_{S'},\bPi_{n_0}) \leq C\left[ d_{\infty}({\bf x}^*,\bPi_{n}) + \sqrt{\frac{\log(n)}{n}}\right]\enspace ;\\
    & d_{\infty}({\bf x}^{*}_{\overline{S}\cap \overline{S}'},\bPi_{2n_0})\leq C\left[ d_{\infty}({\bf x}^*,\bPi_{n}) + \sqrt{\frac{\log(n)}{n}}\right] \enspace .
\end{align*}
\end{lem}

Thus, by a union bound, the following inequalities hold together with probability at least $1-5/n^2$:
\begin{equation}\label{eq:garantie:x1}
\exists \,  Q_1 \in \Ocal: \qquad      d_{\infty}(\hat{{\bf x}}^{(2)}_{\overline{S}},Q_1{\bf x}^*_{\overline{S}}) \leq C_{l,L,e}\,  \left[ d_{\infty}(\xbf^*,\bPi_n) +  \sqrt{\frac{\log(n)}{n}}\right]\enspace ,
\end{equation} 
\begin{equation}\label{eq:garantie:X2}
   \exists \,  Q_2 \in \Ocal: \qquad    d_{\infty}(\hat{{\bf x}}^{(2')}_{\overline{S}'},Q_2{\bf x}^*_{\overline{S}'}) \leq C_{l,L,e}\left[  d_{\infty}(\xbf^*,\bPi_n) +  \sqrt{\frac{\log(n)}{n}}\right]\enspace .
\end{equation} 
Since the final estimator ${\bf \hat{x}}:= ({\bf \hat{x}}^{(2)}_{\overline{S}},\widehat{Q}{\bf \hat{x}}^{(2')}_{\overline {S}})$ satisfies $\hat{{\bf x}}_{\overline{S}} = \hat{{\bf x}}^{(2)}_{\overline{S}}$, we deduce from \eqref{eq:garantie:x1} 
that 
 \begin{equation*}
    d_{\infty}(\hat{{\bf x}}_{\overline{S}},Q_1{\bf x}^*_{\overline{S}}) \leq C_{l,L,e}\left[ d_{\infty}(\xbf^*,\bPi_n)+  \sqrt{\frac{\log(n)}{n}}\right]\enspace .
\end{equation*}
To prove Theorem~\ref{cor:thm:bis}, it suffices to show the counterpart of this bound on $S$:
\begin{equation}\label{target:proof:gluage}
    d_{\infty}(\hat{{\bf x}}_{S},Q_1{\bf x}^*_{S}) \leq C_{l,L,e}\left[ d_{\infty}(\xbf^*,\bPi_n) +  \sqrt{\frac{\log(n)}{n}}\right]\enspace .
\end{equation}
By the triangle inequality, we have 
\begin{align}\label{triang:ineq:gluag}
    d_{\infty}(\hat{{\bf x}}_{S},Q_1{\bf x}^*_{S}) &\leq d_{\infty}(\hat{{\bf x}}_{S},\widehat{Q}Q_2{\bf x}^*_{S}) + d_{\infty}(Q_1{\bf x}^*_{S},\widehat{Q}Q_2{\bf x}^*_{S}) \nonumber \\ 
    &= d_{\infty}(\hat{{\bf x}}^{(2')}_{S},Q_2{\bf x}^*_{S}) +\max_{y\in \Ccal}d(Q_1y,\widehat{Q}Q_2y) \enspace ,
\end{align}
since $\hat{{\bf x}}_{S} = \widehat{Q}\hat{{\bf x}}^{(2')}_{S}$ by definition of $\hat{{\bf x}}$. By~\eqref{eq:garantie:X2} and since $S\subset \overline{S}'$, we have 
\[
d_{\infty}(\hat{{\bf x}}_{S},\widehat{Q}Q_2{\bf x}^*_{S})\leq C_{l,L,e}\left[  d_{\infty}(\xbf^*,\bPi_n) +  \sqrt{\frac{\log(n)}{n}}\right]\enspace .
\]
In view of~\eqref{target:proof:gluage} and~\eqref{triang:ineq:gluag}, it remains to prove that 
\begin{equation}\label{eq:objective:gluage}
    \max_{y\in \Ccal}d(Q_1y,\widehat{Q}Q_2y)\leq C_{l,L,e}\left[  d_{\infty}(\xbf^*,\bPi_n) +  \sqrt{\frac{\log(n)}{n}}\right]\ . 
\end{equation}
Before consider this maximum, we control the quantity $d_{\infty}(Q_1{\bf x}^*_{\overline{S} \cap \overline{S}'}, \widehat{Q}Q_2{\bf x}^*_{\overline{S} \cap \overline{S}'})$ that will turn out to be instrumental. By the triangular inequality, 
\begin{eqnarray*}
d_{\infty}(Q_1{\bf x}^*_{\overline{S} \cap \overline{S}'}, \widehat{Q}Q_2{\bf x}^*_{\overline{S} \cap \overline{S}'})
&\leq&  d_{\infty}(Q_1{\bf x}^*_{\overline{S} \cap \overline{S}'},  \hat{{\bf x}}^{(2)}_{\overline{S} \cap \overline{S}'})+ d_{\infty}(\hat{{\bf x}}^{(2)}_{\overline{S} \cap \overline{S}'}, \widehat{Q}\hat{{\bf x}}^{(2')}_{\overline{S} \cap \overline{S}'})\\
&+&  d_{\infty}(\widehat{Q}\hat{{\bf x}}^{(2')}_{\overline{S} \cap \overline{S}'}, \widehat{Q}Q_2{\bf x}^*_{\overline{S} \cap \overline{S}'}) \enspace.
\end{eqnarray*}
By definition of $\widehat{Q}$, the second term of the right hand-side is bounded by $d_{\infty}(\hat{{\bf x}}^{(2)}_{\overline{S} \cap \overline{S}'}, Q_1Q^{-1}_2\hat{{\bf x}}^{(2')}_{\overline{S} \cap \overline{S}'})$, which, in turn, is bounded as follows
\begin{align*}
 d_{\infty}(\hat{{\bf x}}^{(2)}_{\overline{S} \cap \overline{S}'}, Q_1Q^{-1}_2\hat{{\bf x}}^{(2')}_{\overline{S} \cap \overline{S}'})
 &\leq d_{\infty}(\hat{{\bf x}}^{(2)}_{\overline{S} \cap \overline{S}'}, Q_1{\bf x}^*_{\overline{S} \cap \overline{S}'}) + d_{\infty}(Q_1{\bf x}^*_{\overline{S} \cap \overline{S}'}, Q_1Q^{-1}_2\hat{{\bf x}}^{(2')}_{\overline{S} \cap \overline{S}'})\enspace.
\end{align*}
Together with \eqref{eq:garantie:x1} and \eqref{eq:garantie:X2}, this leads us to 
\begin{eqnarray} 
    d_{\infty}(Q_1{\bf x}^*_{\overline{S} \cap \overline{S}'}, \widehat{Q}Q_2{\bf x}^*_{\overline{S} \cap \overline{S}'})\nonumber 
&\leq& 2    d_{\infty}(\hat{\bf x}^{(2)}_{\overline{S}},Q_1{\bf x}^*_{\overline{S}} ) +  2d_{\infty}(\hat{\bf x}^{(2')}_{\overline{S}'},Q_2{\bf x}^*_{\overline{S}'}) \\
&\leq & C_{lLe}\left[ d_{\infty}(\xbf^*,\bPi_n) +  \sqrt{\frac{\log(n)}{n}}\right] \enspace . \label{eq:upper_gluage}
\end{eqnarray}
Let us now come back to proving~\eqref{eq:objective:gluage}. Since the symmetric group on the plane is only made of rotations and reflections, we consider two cases. 

\noindent 
{\bf Case 1: $Q_1^{-1}\widehat{Q}Q_2$ is a rotation.} Then, $d(Q_1y,\widehat{Q}Q_2y)$ does not depend on $y$. In particular, $\max_{y\in \Ccal}d(Q_1y,\widehat{Q}Q_2y)= d_{\infty}(Q_1{\bf x}^*_{\overline{S} \cap \overline{S}'}, \widehat{Q}Q_2{\bf x}^*_{\overline{S} \cap \overline{S}'})$ and \eqref{eq:objective:gluage} is a consequence of~\eqref{eq:upper_gluage}.

 \noindent 
{\bf 
Case 2: $Q_1^{-1}\widehat{Q}Q_2$ is a reflection}. Then,  $\max_{y\in \Ccal}d(Q_1y,\widehat{Q}Q_2y)=\pi$. \\
If $d_{\infty}({\bf x}^*_{\overline{S} \cap \overline{S}'}, Q_1^{-1}\widehat{Q}Q_2{\bf x}^*_{\overline{S} \cap \overline{S}'})\geq \pi/4$, then 
\[
    \max_{y\in \Ccal}d(Q_1y,\widehat{Q}Q_2y)\leq 4     d_{\infty}(Q_1{\bf x}^*_{\overline{S} \cap \overline{S}'}, \widehat{Q}Q_2{\bf x}^*_{\overline{S} \cap \overline{S}'})\enspace ,
\]
and \eqref{eq:objective:gluage} is again a consequence of~\eqref{eq:upper_gluage}. If $d_{\infty}({\bf x}^*_{\overline{S} \cap \overline{S}'}, Q_1^{-1}\widehat{Q}Q_2{\bf x}^*_{\overline{S} \cap \overline{S}'})\leq \pi/4$, this implies that the points in ${\bf x}^*_{\overline{S} \cap \overline{S}'}$ belong to two arcs of length $\pi/4$ that are (individually) symmetric around the axis of the reflection $Q_1^{-1}\widehat{Q}Q_2$. It follows that $d_{\infty}({\bf x}^*_{\overline{S} \cap \overline{S}'},\bPi_{2n_0})\geq \pi/8$ as soon as $2n_0\geq 4$, that is $n\geq 8$. Indeed, if $d_{\infty}({\bf x}^*_{\overline{S} \cap \overline{S}'},\bPi_{2n_0})< \pi/8$ and $2n_0\geq 4$, this would imply that, any point on $\Ccal$ is at distance less than $3\pi/8$ from ${\bf x}^*_{\overline{S} \cap \overline{S}'}$ which is impossible because those points in ${\bf x}^*_{\overline{S} \cap \overline{S}'}$ belong to these two arcs of length $\pi/4$. 
Since Lemma~\ref{cl:equivBias:bis} ensures that $d_{\infty}({\bf x}^*_{\overline{S} \cap \overline{S}'},\bPi_{2n_0})$ is of the same order as $d_{\infty}({\bf x}^*,\bPi_{n})$, this implies that the latter is of the order of a constant and~\eqref{eq:objective:gluage} is obviously valid.

\subsubsection{Proof of Lemma \ref{cl:equivBias:bis}}

We claim that it suffices to restrict our attention to the case where ${\bf x}^* = (x_{1}^*,\ldots,x_{ n}^*)$ are $n$ distinct points. Indeed, for general points $x_{1}^*,\ldots,x_{n}^*$ in $\Ccal$, there exist points $y_{1},\ldots,y_{n}$ that are all distinct and satisfy $d(y_{j}, x_{j}^*) \leq 1/n$ for all $j \in[ n]$. Replacing $x_{1}^*,\ldots,x_{n}^*$ by $y_{1},\ldots,y_{n}$ in the statement of Lemma \ref{cl:equivBias:bis} only entails an additional term $1/n$ which is negligible compared to the term $\sqrt{\log(n)/n}$.

For any $k\in[n]$ and any vector $\textbf{x} \in \Ccal^{ k}$, we introduce a new quantity that is equivalent to $d_{\infty}(\textbf{x},\bPi_k)$, but more easy to handle. For any interval $I\subset \mathbb{R}/(2\pi)$, we write $N_I(\textbf{x})$ the number of coordinates of $\textbf{x}$ that lie in the interval $I$, i.e. the number of $i\in [k]$ such that $\underline{x}_i\in I$. We then define the quantity $V_I(\textbf{x})$ as
\begin{equation}\label{def:Vi}
    V_I(\textbf{x}) = N_I(\textbf{x}) -k\frac{|I|}{2\pi}\enspace .
\end{equation}
Remark that, for a uniform $k$-sample of $\Ccal$, the fraction $k|I|/(2\pi)$ would be the expected number of points in $I$. The next lemma shows that the supremum $\textup{sup}_{I} V_I(\textbf{x})$ is equivalent to  $k\, d_{\infty}({\bf x},\bPi_{k})$. We note $\mathcal{I}$ the set of all closed intervals $I\subset \mathbb{R}/(2\pi)$.

\begin{lem}\label{equivAlphaVi} For any integer $k\in [n]$ and any vector ${\bf x}=(x_1,\ldots,x_k)$ of $k$ distinct points of $\Ccal$, we have
\begin{equation*}
    \underset{I\in \mathcal{I}}{\textup{sup}}\, |V_I({\bf x})|- 4\leq \frac{k}{\pi}d_{\infty}({\bf x},\bPi_{k}) \leq 2 \, \underset{I\in \mathcal{I}}{\textup{sup}}\, |V_I({\bf x})|+ 4\enspace .
\end{equation*}
\end{lem}

Thus, to prove Lemma \ref{cl:equivBias:bis}, it is enough to show that,  for $T=S$, $S'$, and $\overline{S}\cap \overline{S}'$, one has
\begin{align}\label{goal:equiv:VI}
\mathbb{P}\left[\sup_{I\in\mathcal{I}}\left||V_I({\bf x}^*_{T})| - \frac{|T|}{n}|V_I({\bf x}^*)|\right|> C \sqrt{ n \log(n)} \right]\leq \frac{1}{n^3}\ . 
\end{align}
The next Lemma states a uniform concentration bound for $V_I$.  

\begin{lem}\label{VI:hyperGeo}
    Consider any integer $n> 4$ and any integer $k< n$.  Fix any ${\bf x} \in \Ccal^{ n}$. Sampling uniformly at random $k$ coordinates of ${\bf x}$ without replacement, we write ${\bf x}^{(k)} \in \Ccal^{ k}$ the resulting vector. Then, with probability higher than $1-1/n^3$, one has 
\begin{equation}
    \sup_{I \in \mathcal{I}}\left|\left|V_{I}({\bf x}^{(k)}) \right| -\frac{k}{n}  \left| V_{I}({\bf x})\right|\right|  \leq  6\sqrt{n \log(n)} \enspace .
\end{equation}
\end{lem} 
Since the marginal distributions of $S$, $S'$, and $\overline{S}\cap \overline{S}'$ are uniform, we can apply  Lemma~\ref{VI:hyperGeo} to ${\bf x}^*_S$, ${\bf x}^*_{S'}$, and ${\bf x}^*_{\overline{S}\cap \overline{S}'}$ and the conclusion of the Lemma holds with probability higher than $1-3/n^3$, which is higher than $1-1/n^2$. 
\hfill $\square$

\begin{proof}[Proof of Lemma \ref{VI:hyperGeo}]
 We start with a fixed interval $I \in \mathcal{I}$. Since $N_I(\textbf{x}^{(k)})$ is a hypergeometric random variable with parameters $(k,\frac{N_I(\textbf{x})}{n},n)$, we can invoke Hoeffding inequality \eqref{Hoeffding}  for hypergeometric distributions and get
\begin{equation}\label{Hoefdin:HyperGep:appli}
\P \left( \left|N_I(\textbf{x}^{(k)}) - k\frac{N_I(\textbf{x})}{ n}\right| \geq \sqrt{\frac{7 k \log(n)}{2}} \right) \leq \frac{2}{n^7} \leq \frac{1}{n^6} \enspace .
\end{equation}
We combine \eqref{Hoefdin:HyperGep:appli} with 
$$ N_I(\textbf{x}^{(k)}) - k\frac{|I|}{2\pi}  = \left( N_I(\textbf{x}^{(k)}) - k\frac{N_I(\textbf{x})}{n} \right) +  \frac{k}{n} \left( N_I(\textbf{x}) -  n\frac{|I|}{2\pi} \right) \enspace,$$
to conclude that 
\begin{equation}\label{resultForOneInt}
 \P\left[\left||V_I(\textbf{x}^{(k)}) |  -   \frac{k}{ n}  | V_I(\textbf{x})| \right| \leq  \sqrt{\frac{7 n \log(n)}{2}} \right]\leq \frac{1}{n^6}  \enspace .
\end{equation}

In order to extend \eqref{resultForOneInt} to all intervals $I \in \mathcal{I}$, we use an $\ep$-net approach with a  subcollection $\mathcal{I}_n(\textbf{x})$ of $\mathcal{I}$. Let $\mathcal{I}_n(\textbf{x})$ be the collection of all intervals $I_n=[a_n, b_n]$ where $a_n, b_n \in \{\underline{x}_1\ldots,\underline{x}_{n}\} \cup  \Ccal_{ n} $, i.e., $a_n, b_n$ are either coordinates of $\textbf{x}$ or elements of the $ n$-regular grid $\{ 2\pi i/ n ; \ \, i\in [ n]\}$. We then apply~\eqref{resultForOneInt} together with a union bound over all intervals $I \in \mathcal{I}_n(\textbf{\textup{x}})$. Since  $|\mathcal{I}_n(\textbf{\textup{x}})|\leq   (2 n)^2 \leq n^3$, we obtain
\begin{equation}\label{nbrpt:inter:lem:proof:bis:new}
\sup_{I \in \mathcal{I}_n(\textbf{\textup{x}})}\left|\left|V_{I}(\textbf{x}^{(k)}) \right| -\frac{k}{n}  \left| V_{I}(\textbf{x})\right|\right|  \leq \sqrt{ \frac{7n \log(n)}{2}} \enspace , \end{equation}
with probability higher than $1-1/n^3$. 

To obtain \eqref{nbrpt:inter:lem:proof:bis:new} for all  $I\in \mathcal{I}$, we observe that, for any $I\in \mathcal{I}$, there exists $I_n \in \mathcal{I}_n(\textbf{x})$ such that
\begin{equation}\label{proof:claim:approx:decomp:bis}
I = I^{(l)} \cup I_n \cup I^{(r)},    
\end{equation}
where $I^{(l)}$ and $I^{(r)}$ are two closed intervals of $I\setminus{I_n}$ whose  lengths are smaller than $2\pi/n$  and that satisfy $N_{I^{(l)}}(\textbf{x}) = N_{I^{(r)}}(\textbf{x}) = 0$. In particular, we have $N_{I^{(l)}}(\textbf{x}^{(r)}) = N_{I^{(l)}}(\textbf{x}^{(k)}) = 0$. We then deduce that 
\begin{align*}
V_I(\textbf{x}^{(k)}) &= V_{I^{(l)}}(\textbf{x}^{(k)}) + V_{I_n}(\textbf{x}^{(k)}) + V_{I^{(r)}}(\textbf{x}^{(k)})\\
& = -k\frac{|I^{(l)}|}{2\pi} + V_{I_n}(\textbf{x}^{(k)}) - k\frac{|I^{(r)}|}{2\pi}\enspace .
\end{align*} 
Since the same decomposition holds for $V_I(\textbf{x})$, we get
\[
    \left||V_I(\textbf{x}^{(k)})| - \frac{k}{n}  |V_{I}(\textbf{x})|\right|\leq 4 + \left||V_{I_n}(\textbf{x}^{(k)})| - \frac{k}{n}  |V_{I_n}(\textbf{x})|\right|
\]
Together with~\eqref{nbrpt:inter:lem:proof:bis:new}, we obtain 
\[
    \sup_{I \in \mathcal{I}}\left|\left|V_{I}(\textbf{x}^{(k)}) \right| -\frac{k}{n}  \left| V_{I}(\textbf{x})\right|\right|  \leq  4+ \sqrt{\frac{7n \log(n)}{2}} \enspace ,
\]
with probability higher than $1-1/n^3$. Lemma \ref{VI:hyperGeo} is proved
\end{proof}

\bigskip

\begin{proof}[Proof of Lemma \ref{equivAlphaVi}]
We first prove  the upper bound $$k d_{\infty}({\bf x},\bPi_{k}) \leq 2\pi [\underset{I\in \mathcal{I}}{\textup{sup}}\, |V_I({\bf x})| + 2]\enspace.$$ 
Recall that for a vector ${\bf x}= (x_1, x_2, \ldots,x_k) \in \Ccal^{ k}$, we say that ${\bf x}$ is ordered, if these points are consecutive when one walks on the sphere with the trigonometric direction. Without loss of generality and for ease of exposition, we assume that the identity permutation is a latent order, that is $x_{1},\ldots,x_{k}$ is ordered.

We define ${\bf x}^{**}_S = (x_1^{**},\ldots,x_{k}^{**})$ a vector of $\bPi_{k}$ as follows. The first point $x_1^{**}\in \Ccal_{k}$ is a closest point to $x_1^{*}$ with respect to $d$ and the other points $x_{j+1}^{**}$ are elements of $\Ccal_k$ with  arguments 
\[ 
\underline{x}_{j+1}^{**}= \underline{x}_1^{**}+ j\frac{2\pi }{k} \ \, (\mathrm{mod}\, 2\pi)\, ,\quad \quad \text{ for }j=1,\ldots,k-1\enspace .
\]
Fix any $i \in \{2,\ldots, k\}$ and consider the intervals $I_i=[\underline{x}_1, \underline{x}_i]$ and $I'_i=[\underline{x}_1, \underline{x}_i^{**}]$. We have 
 \begin{equation}\label{eq:nul:pro:lem:truPos:bis}
d(x_i,x_i^{**}) \leq \big{|}|I_i| - |I'_i|\big{|}\enspace .  
\end{equation}
Observe that $N_{I_i}(\textbf{x}) = i$ since $x_{1}, \ldots, x_{k}$ are ordered and all distinct. Hence, 
\begin{equation}\label{eq1:proof:lem:app:truPos:bis}
\left|\frac{2 \pi i}{k} -|I_i| \right| = \left|\frac{2 \pi N_{I_i}}{k} -|I_i| \right| = 2\pi  \frac{|V_{I_i}(\textbf{x})|}{k} \leq 2\pi\,  \underset{I\in \mathcal{I}}{\textup{sup}}\, \frac{|V_I(\textbf{x})|}{k}\enspace .   
\end{equation}
Besides, we know that the length of $I'_i$ is equal to $\big{|}[\underline{x}_1^{**}, \underline{x}_i^{**}]\big{|}$ up to an additional term $d(x_1,x_1^{**})$, that is
$$\left||I'_i|- \big{|}[\underline{x}_1^{**}, \underline{x}_i^{**}]\big{|} \right|  \leq  d(x_1,x_1^{**})\enspace .$$ By construction of the $x^{**}_j$'s,  we have $d(x_1,x_1^{**}) \leq 2\pi /k$  and $\big{|}[\underline{x}_1^{**}, \underline{x}_i^{**}]\big{|} = 2\pi (i-1) /k$. Hence, 
we obtain $  k||I'_i|-2\pi i  |\leq  4\pi$. 
We then deduce from \eqref{eq1:proof:lem:app:truPos:bis} and the triangular inequality  that    $ k\big{|}|I_i| - |I'_i|\big{|} \leq2\pi \,  \underset{I\in \mathcal{I}}{\textup{sup}}\, |V_I| +   4\pi$. Coming back to \eqref{eq:nul:pro:lem:truPos:bis}, taking the supremum over all $i\in \{2,\ldots, k\}$, and noting that $d(x_1,x_1^{**})\leq 2\pi/k$  leads us to 
$$d_{\infty}({\bf x}, {\bf x}^{**}) \leq \frac{2\pi}{k}  \left[\underset{I\in \mathcal{I}}{\textup{sup}}\, |V_i|+2\right]\enspace ,$$
where ${\bf x}^{**}\in \bPi_k$. Finally, we take the minimum over $\bPi_k$ to get the desired bound.

\bigskip

We now turn to the lower bound  $ k d_{\infty}({\bf x},\bPi_{k}) \geq \pi  [\underset{I\in \mathcal{I}}{\textup{sup}}\, |V_I({\bf x})|-4] $. 
Consider any such interval $I$ and $\xbf^{**}\in \bPi_k$. Since the entries of $\xbf^{**}$ are regularly spaced on $\Ccal$, it follows that 
$|N_I(\xbf^{**})-k|I|/(2\pi)|\leq 1$ so that  $|V_I(\xbf)|\leq |N_I(\xbf) - N_I(\xbf^{**})|+1$. 
Now, assume that $N_I(\xbf)> N_{I}(\xbf^{**})$. We claim that $\sup_{j:\, x_j\in I}|x_j-x_j^{**}|\geq \frac{\pi}{k}[|N_I(\xbf) - N_I(\xbf^{**})|-3]$. Otherwise, the set of $x_j^{**}$ with $j$ satisfying $x_j\in I$ is included in an interval of size 
\[
    |I|+ \frac{2\pi}{k}[N_I(\xbf) - N_I(\xbf^{**})-3]\leq \frac{2\pi}{k}[N_I(\xbf) -2]\enspace .
\]
This contradicts the fact that  this set of equi-spaced points has size $N_I(\xbf)$. If $N_I(\xbf)< N_{I}(\xbf^{**})$, we simply consider the complement\footnote{Although $\overline{I}$ is an open interval, the arguments are still valid.} interval $\overline{I}$ that satisfies $N_{\overline{I}}(\xbf)= k -N_I(\xbf)$ and  $N_{\overline{I}}(\xbf^{**})= k- N_{I}(\xbf^{**})$ to conclude that  
\[
    \sup_{j:\, x_j\in \overline{I}}|x_j-x_j^{**}|\geq \frac{\pi}{k}\left[|N_{\overline{I}}(\xbf) - N_{\overline{I}}(\xbf^{**})|-3\right]=  \frac{\pi}{k}\left[|N_I(\xbf) - N_I(\xbf^{**})|-3 \right]\enspace .
\]
Putting everything together, we have shown that 
\[
     d_{\infty}(\xbf,\xbf^{**})\geq  \frac{\pi}{k}\left[|N_I(\xbf) - N_I(\xbf^{**})|-3 \right]\geq \frac{\pi}{k}\left[|V_I(\xbf) |-4 \right]\ . 
\]
Taking the infimum over $\xbf^{**}$ and the supremum over $I$ leads to the desired result.

\end{proof}

\subsection{Proof of Corollary \ref{coro:iid:perf}}\label{append:coro}

Theorem~\ref{cor:thm:bis} ensures that, conditionally to ${\bf x}^*$, 
 \begin{equation*}
    \min_{Q\in \mathcal{O}}d_{\infty}(\hat{{\bf x}},Q{\bf x}^*) \leq C'_{lLe} \left(  d_{\infty}({\bf x}^*,\bPi_n) +  \sqrt{\frac{\log(n)}{n}}\right) \enspace ,
\end{equation*} with probability at least $1-5/n^2$. Thus,  it suffices to show that, with probability at least $1-2/n^2$, one has  
 \begin{equation*}
   d_{\infty}({\bf x}^*,\bPi_n)  \leq C  \sqrt{\frac{\log(n)}{n}} \enspace,
\end{equation*} for some $C>0$. 
We shall rely on  Dvoretzky–Kiefer–Wolfowitz (DKW) inequality. Indeed, the arguments $\underline{x}_1^*,\ldots, \underline{x}_n^*$ are independent and uniformly distributed on $[0,2\pi)$. Besides, any interval $I$ of the torus $\mathbb{R}/(2\pi)$ can be represented as a union of at most two intervals of $[0,2\pi)$. For any interval $I$, we denote $|I|$ its length and $N_I({\bf x}^*)$ the number of points $x^*_i$ whose argument lies in $I$. Then, we deduce from  DKW inequality that, for any $t >0$,
\begin{equation*}
   \P\left( \underset{I \subset \,  \mathbb{R}/(2\pi)}{\textup{sup}} \left|\frac{N_I({\bf x}^*)}{n} -\frac{|I|}{2\pi}\right| > 4t \right) \leq 2 e^{-2 nt^2} \enspace.
\end{equation*}
We then choose $t = \sqrt{\log(n)/n}$ to obtain 
\begin{equation*}
   \P\left( \underset{I \subset \,  \mathbb{R}/(2\pi)}{\textup{sup}} \left|\frac{N_I({\bf x}^*)}{ n} -\frac{|I|}{2\pi}\right| > 4 \sqrt{\frac{\log(n)}{n}} \right) \leq \frac{2}{n^2} \enspace,
\end{equation*}
 Besides, by Lemma \ref{equivAlphaVi}, we know that the quantity $V_I({\bf x}^*) = N_I({\bf x}^*) -\frac{ n|I|}{2\pi}$ introduced in \eqref{def:Vi} satisfies 
$$ d_{\infty}({\bf x}^*,\bPi_{n}) \leq C\left( \underset{I\subset \,  \mathcal{I}}{\textup{sup}}\, \frac{|V_I({\bf x}^*)|}{ n} + \frac{1}{ n}\right) \enspace ,$$
where $\mathcal{I}$ stands for the set of interval on the torus $\mathbb{R}/(2\pi)$. The last two displays lead to the desired result.
 \hfill $\square$

\section{Proof of the identifiability results and minimax lower bound }\label{proof:lowerbound}

\subsection{Proof of Proposition~\ref{prop_identif}}
For simplicity, we assume that $n/8$ is an integer in the rest of the example and we write $n= 8 n_1$. The construction of $f'$ mainly amounts to contracting the function $f$ in some regions and dilating it in other regions which allows to contracting and dilating the positions ${\bf x}$.

 Consider a partition of the latent space $\Ccal = \Ccal_1 \cup \Ccal_2 \cup \Ccal_3$ in three arcs $\Ccal_1=(\underline{x}_{n},\underline{x}_{ n_1}] = (0,\pi/4]$, $\Ccal_2=(\underline{x}_{n_1},\underline{x}_{4n_1}]=(\pi/4, \pi]$ and $\Ccal_3=(\underline{x}_{4n_1}, \underline{x}_{ n}]=(\pi, 2\pi]$. For $x$ and $y$ belonging $\Ccal_1$, define $f'_1(x,y)$ by  $f'_1(x,y)=  1 - d(x,y)/\pi$. For $k=1,\ldots, 2n_1$, define $x'_k= e^{ \iota k \pi/n}$ and let $x'_{n}=x_{ n}=1$. In other words, we contract the positions $x_k$ for $k=1,\ldots, 2n_1$. Although we have not yet completely defined ${\bf x}'$, we already can certify that  $\min_{Q\in \mathcal{O}}d_{\infty}({\bf x},Q{\bf x}') \geq \pi /8$. Besides, we have $f'_1(x'_i,x'_j)=f(x_i,x_j)$  for all $i,j \in [2n_1]\cup \{n\}$.

For $x$ and $y$ in $\Ccal_2$, we define  $f'_2(x,y)$ by $f'_2(x,y)=  1 - d(x,y)/(3\pi)$. For $k=1,\ldots, 2n_1$, we set  $x'_{k+ 2n_1}= e^{ \iota \pi/4} e^{ \iota k 3\pi/n}$. Again, observe that $f'_2(x'_i,x'_j)=f(x_i,x_j)$ for all integers $i,j$  in  $[2n_1+1,4n_1]$. Finally, for $x$ and $y$ in $\Ccal_3$, set $f'_3(x,y)= f(x,y)$, and let $x'_k= x_k$ for all integers $k=4n_1+1,\ldots, n-1$. Obviously, we have $f'_3(x'_i,x'_j)=f(x_i,x_j)$ for all integers $i,j \in [4n_1,n]$. 

It remains to deal with the situations where the pairs of points lie in different parts of the partition $\Ccal_1 \cup \Ccal_2 \cup \Ccal_3$. In the case where $x \in \Ccal_1$ and $y \in \Ccal_2$, define $f'_{1-2}(x,y)=  f'_1(x,e^{\iota\pi/4}) + f'_2(e^{\iota\pi/4},y)-1.$ For all integers $i \in [0,2n_1]$ and $j \in [2n_1,4n_1]$, we have already seen that $f'_1(x'_i,e^{\iota\pi/4}) = f(x_i,e^{\iota\pi/2})$ and $f'_1(e^{\iota\pi/4},x'_j) = f(e^{\iota\pi/2},x_j)$. Hence $f'_{1-2}(x'_i,x'_j) = f(x_i,e^{\iota\pi/2}) + f(e^{\iota\pi/2},x_j)-1 = f(x_i,x_j)$.

In the case where $x \in \Ccal_1$ and $y \in \Ccal_3$, define $f'_{1-3}(x,y)=  f'_1(x,e^{\iota0}) + f'_3(e^{\iota0},y)-1$ if the length of the arc $[x,e^{\iota0}]\cup (e^{\iota0},y]$ is less than $\pi$; otherwise, set $f'_{1-3}(x,y)=  f'_1(x,e^{\iota\pi/4}) + f'_2(e^{\iota\pi/4},e^{\iota\pi})+ f'_3(e^{\iota\pi},y)-2.$ Since $f$ admits similar decompositions, one can deduce from the above that $f'_{1-3}(x'_i,x'_j)  = f(x_i,x_j)$ for all $i\in [1, 2n_1]$ and $j \in [4n_1+1,n]$.

 The remaining cases can be handled in the same manner. Finally, we define the symmetric function $f'$ on $\Ccal\times \Ccal$ relying on $f'_1$, $f'_2,$ $f'_3$, $f'_{1-2}$, $f'_{1-3}$, and $f'_{2-3}$. Then, we can readily check that  $f \in \mathcal{BL}[(3\pi)^{-1},\pi^{-1},0]$ and that $f(x_i,x_j)= f'(x'_i,x'_j)$ for all $i,j\in[n]$. As a consequence, $({\bf x}',f')$ belongs to $\Rcal[F,(3\pi)^{-1},\pi^{-1},0]$. One easily check that ${\bf x}'\in \mathcal{S}_{ev}$ and the result follows. 
 \hfill $\square$

\subsection{Proof of~\eqref{eq:identif:Q} in Proposition~\ref{prop:representative}}\label{sec:prop_representative:deuxieme_preuve}


We show in the paragraph below that~\eqref{eq:identif:Q}  is a consequence of the proof of Theorem~\ref{cor:thm:bis} in the noiseless case ($E=0$), after application of the triangular inequality. Indeed, since the noise is equal to zero, the conclusion of Theorem~\ref{cor:thm:bis} is deterministic (and not with high probability anymore), so it can be used to prove  deterministic inequalities such as \eqref{eq:identif:Q}. By doing so, we establish \eqref{eq:identif:Q} via our localization algorithm (Theorem~\ref{cor:thm:bis}), though \eqref{eq:identif:Q} is an approximation result (independent of any algorithm) which could be proved directly.

Consider any two representations $(\xbf, f)$ and $(\xbf', f')$ in $R[F,c_l,c_L,c_e]$ and apply our Localize-and-Refine procedure to noiseless observations $A=F$. The conclusion of Theorem~\ref{cor:thm:bis} applies to both $\xbf$ and $\xbf'$, so that we have 
\begin{eqnarray*}
    \min_{Q\in\Ocal}d_{\infty}(\hat \xbf,Q\xbf) &\leq &C'_{lLe} \left(d_{\infty}(\xbf,\bPi_{n})+ \sqrt{\log(n)\over n}\right) \\
    \min_{Q\in\Ocal}d_{\infty}(\hat \xbf,Q\xbf') & \leq &C'_{lLe} \left(d_{\infty}(\xbf',\bPi_{n})+ \sqrt{\log(n)\over n}\right)\enspace .    
\end{eqnarray*}
Hence, it follows from the triangular inequality that 
\begin{eqnarray*}
    \min_{Q\in \Ocal}d_{\infty}(\xbf,Q\xbf')&=& \min_{Q_1,\ Q_2\in\Ocal}d_{\infty}(Q_1\xbf,Q_2\xbf')\\
    &\leq& C'_{lLe} \left(d_{\infty}(\xbf,\bPi_{n})+d_{\infty}(\xbf',\bPi_{n})+ 2\sqrt{\log(n)\over n}\right)\enspace . 
\end{eqnarray*}

\subsection{Proof of Theorem \ref{thm:lowerBound}}
\label{append:LB}

We  establish the lower bound $\sqrt{\log(n)/n}$ in the particular setting where the observations $A_{ij}$ are independent Bernoulli random variables of parameters $F_{ij}=f_0(x_i,x_j)$, for the specific function \begin{equation}\label{def:f:lowerbound}
    f_0(x_i,x_j) = (3/4)- d(x_i,x_j)/(4\pi),
\end{equation}with $\textup{{\bf x}}=(x_1,\ldots,x_{n}) \in \bPi_n.$ The corresponding probability distribution is denoted by $\P_{(\textup{{\bf x}},f_0)}$. 

This minimax lower bound is based on Fano's method as stated below. For two configuration ${\bf x}$ and ${\bf x}'$ in $\bPi_n$, we denote the Kullback-Leibler divergence of $\P_{(\textup{{\bf x}},f_0)} $ and $\P_{(\textup{{\bf x}}',f_0)}$ by $KL(\P_{(\textup{{\bf x}},f_0)} \| \P_{(\textup{{\bf x}}',f_0)})$.  Besides, we quantify the quasi-metric $\rho({\bf x},{\bf x}') = \min_{Q\in \mathcal{O}}d_{\infty}({\bf x},Q{\bf x}')$.  Given a radius $\delta>0$ and a subset $\Scal\subset \bPi_n$, the packing number $\mathcal{M}(\delta, \mathcal{S}', \rho)$ is defined as the largest number of points in $\mathcal{S}'$ that are at quasi-distance  $\rho$ at  least $\delta$ away from each other. Below, we state a specific version of Fano's lemma. 

\begin{lem}[from \cite{yu1997assouad}]\label{fano:prop}
Consider any subset $\mathcal{S}'\subset \bPi_n$. Define the Kullback-Leibler diameter of $\mathcal{S}'$ by 
\begin{equation*}
    d_{KL}(\mathcal{S}')= \underset{{\textup{{\bf x}}},{\textup{{\bf x}}}'\in \mathcal{S}'}{\textup{sup}} KL(\P_{(\textup{{\bf x}},f_0)} \| \P_{({\textup{{\bf x}}}',f_0)})\enspace .
\end{equation*}Then, for any estimator $\hat{\textup{{\bf x}}}$ and for any $\delta>0$, we have
\begin{equation*}
    \underset{\textup{{\bf x}} \in \mathcal{S}'}{\textup{sup}} \quad  \P_{(\textup{{\bf x}},f_0)}\cro{\rho(\hat{{\textup{{\bf x}}}},{\textup{{\bf x}}}) \geq \frac{\delta}{2}} \geq 1 - \frac{d_{KL}(\mathcal{S}') + \log(2)}{\log \mathcal{M}(\delta, \mathcal{S}', \rho)}\enspace  .
\end{equation*}
\end{lem}

In view of the above proposition, we mainly have to choose a suitable subset $\mathcal{S}'$, control its Kullback diameter, and get a sharp lower bound of its packing number. The main difficulty stems from the fact that the loss function $\rho({\bf x},{\bf y}) = \min_{Q\in \mathcal{O}}d_{\infty}({\bf x},Q{\bf y})$ is a minimum over a collection of orthogonal transformations. It is therefore challenging to  derive  a tight lower bound for this loss.

Let $k :=  C' \sqrt{n\log(n)}$, for a small enough constant $C' \in (0,1]$ that will be set later. Define $n/2$ vectors ${\bf x}^{(s)} \in \bPi_n$, $s=1,\ldots,n/2$, as follows. For each $s\in [n/2]$, we define $x_j^{(s)}$ by its argument $\underline{x}^{(s)}_j$
 \begin{align*}
    \underline{x}^{(s)}_j &= \frac{2\pi j}{ n} \ ,  \qquad \forall j \in [n]\setminus{\{ s, s+k\}}\ ,
    \underline{x}^{(s)}_s = \frac{2\pi (s+k)}{ n} ,\quad 
    \underline{x}^{(s)}_{s+k} = \frac{2\pi s}{ n}\enspace .
\end{align*}
Each vector of arguments $\underline{\bf x}^{(s)}$ is therefore equal to the vector $(2\pi j /n)_{j\in[n]}$ up to an exchange of the positions $2\pi s / n$ and $2\pi (s+k)/n$. This collection of $n/2$ vectors is denoted by $\mathcal{S}':=\{{\bf x}^{(1)},\ldots,{\bf x}^{(n/2)}\}$. Obviously $\mathcal{S}'\subset \bPi_n$, and one can readily checks that 
\begin{equation}\label{separation}
 \rho({\bf x}^{(t)},{\bf x}^{(s)}) \geq \frac{\pi k}{n} \enspace, \qquad \forall s,t \in \left[\frac{n}{2}\right], \, s\neq t \enspace,
\end{equation} which in turn ensures that the packing number $\mathcal{M}(\de_n,\mathcal{S}',\rho)$ of radius $\de_n := \pi k /n$ satisfies  
$\mathcal{M}(\de_n,\mathcal{S}',\rho) \geq n/2$.
To upper bound the KL diameter of $\mathcal{S}'$, we use the following claim whose proof is postponed to the end of the section.
\begin{cl}\label{claim:divergenceBernouilli}For any ${\bf x}, {\bf x}' \in \Ccal^{  n},$ we have 
$$KL(\P_{({\bf x},f_0)} \| \P_{({\bf x}',f_0)}) \leq 8 \sum_{i,j} (f_0(x_i,x_j)-f_0(x_i',x_j'))^2\enspace. $$
\end{cl}
Together with the definition \eqref{def:f:lowerbound} of $f_0$, we get 
$$KL(\P_{({\bf x}^{(t)},f_0)} \| \P_{({\bf x}^{(s)},f_0)})\leq C n \de_n^2 \leq C(C')^2 \log(n) \enspace ,$$ for some numerical constant $C$. Then, choosing the constant $C'$ in the definition of $k$ such that  $C' = (2\sqrt{C} )^{-1}$  leads to $d_{KL}(\mathcal{S}') \leq \log(n)/4.$

Applying Lemma \ref{fano:prop} to this set $\mathcal{S}'$, we arrive at
 \begin{equation*}
      \inf_{\hat{\textup{{\bf x}}}} \underset{{\bf x} \in \mathcal{S}'}{\textup{sup}} \quad  \P_{({\bf x},f_0)}\cro{\rho(\hat{\textup{{\bf x}}},\textup{{\bf x}}) \geq \frac{\de_n}{2}} \geq 1 - \frac{\log(n)/4 + \log(2)}{\log(n/2)} \geq \frac{1}{2} \enspace ,
\end{equation*}as soon as $n$ is large enough. Theorem \ref{thm:lowerBound} is proved. \hfill $\square$

\begin{proof}[Proof of Claim \ref{claim:divergenceBernouilli}]
 By definition of the Kullback-Leibler divergence, and $F_{ij}= f_{0}(x_i,x_j)$ and $F'_{ij}$ $=$ $f_{0}(x_i',x_j')$, we have $$KL(\P_{({\bf x},f_0)} \| \P_{({\bf x}',f_0)}) = \sum_{i < j} F_{ij} \log \frac{F_{ij}}{F'_{ij}} +  (1-F_{ij}) \log \frac{1-F_{ij}}{1-F'_{ij}}\enspace ,$$
and since $\log(t) \leq t-1$ for all $t >0$, it follows that 
$$KL(\P_{({\bf x},f_0)} \| \P_{({\bf x}',f_0)}) \leq \sum_{ij} \frac{(F_{ij}-F'_{ij})^2}{F'_{ij}(1-F'_{ij})} \leq 8 \sum_{i,j} (F_{ij}-F'_{ij})^2 \enspace ,$$ where the second inequality follows from the fact that $1/4\leq F'_{ij}\leq 3/4$.
 \end{proof}


\input{seriation_appendix_D}

\section{Probabilistic inequalities}

We recall Hoeffding inequality for hypergeometric distributions.

\begin{lem}
For $N \geq 1$, $p\in [0,1]$ and $n \geq N$, let $X$ be a hypergeometric random variable with parameters $(N,p,n)$. Then, for all $t >0$,
\begin{equation}\label{Hoeffding}
    \P\left(|X- Np| \geq \sqrt{\frac{Nt}{2}} \right) \leq 2 e^{-t} \enspace .
\end{equation}
\end{lem}

%% file: seriation_appendix_D.tex
\section{Proof for the spectral method}\label{sec:appendix:spectral}

\subsection{Proof of Theorem~\ref{thm:graphgeo:new} }

Recall that the Spectral Localization (LS) algorithm is applied to the data matrix $A_{SS}$, where $S$ is a subset  of indices of $[n]$, with a cardinal number $|S|=n_0= n/4$. We can assume that $S=\{1,\ldots,n_0\}$ for the ease of exposition. Vanilla Spectral Localization in LS algorithm returns $\hxsp_S :=(\hat{x}_1^{\textup{VSA}},\ldots,\hat{x}_{n_0}^{\textup{VSA}})$ with  $\hat{x}_i^{\textup{VSA}} :=(\hat{x}_{ij}^{\textup{VSA}})_{j\in [2]} = \sqrt{\frac{n_0}{2}}(\hat{u}_i, \hat{v}_i) \in \mathbb{R}^2$ , $i=1,\ldots, n_0$.  
We denote by $\hat{\lambda}_{0}^{(S)}\geq \ldots \geq \hat{\lambda}_{n_0-1}^{(S)}$ the eigenvalues of the adjacency matrix $A_{SS}$.

Note that the position estimates $\hat{x}_i^{\textup{VSA}}$ do not lie on the unit sphere $\mathcal{C}$. As a consequence, the quantity $d(\hat{x}_i^{\textup{VSA}},x^*_i)$ is not defined, and we will use the distance $\|\hat{\bf x}^{\textup{VSA}}_S- {\bf x}^*_{S}\|_1= \sum_{i=1}^{n_0} \sum_{j=1}^2 |\hat{x}_{ij}^{\textup{VSA}}-x^*_{ij}|$  where   ${\bf x}^*_{S}$ is interpreted as a $2\times n_0$ matrix. Besides, since ${\bf x}^*_{S}$ can only be recovered up to orthogonal transformations, we consider the loss $\underset{Q \in \mathcal{O}}{\textup{min}}\|\hxsp_S - Q{\bf x}^*_{S}\|_{1}$ where $Q$ is interpreted as $2\times 2$ orthogonal matrix.

Let $\lambda_{0}^{*(S)}\geq \ldots \geq \lambda_{n_0-1}^{*(S)}$ denote the eigenvalues of the signal matrix $F_{SS}:= [f(x_i^{*},x_j^{*}))]_{i,j\in S}$. We denote by $\Delta_1^{(S)} := \lambda_{0}^{*(S)} -\lambda_{1}^{*(S)}$ and $\Delta_2^{(S)}:= \lambda_{2}^{*(S)} - \lambda_{3}^{*(S)}$ the two relevant spectral gaps.

\begin{prop}\label{conj:graphgeo:l1:new} 
Let $n_0\geq 4$, and  $f$ be a geometric function as defined in \eqref{def:geometricSetting}, such that $f$ belongs to $\mathcal{BL}[c_l,c_L,c_e]$. Let $c_a>0$ be any positive constant. Assume that the  latent positions ${\bf x}^*_{S}$ fulfill the following inequality
\begin{equation}\label{GraphGeo:assump:latentPosition:S}
 d_{\infty}({\bf x}^*_{S},\bPi_{n_0}) \leq c_a  \sqrt{\frac{\log(n)}{n}}\enspace .
\end{equation}
Then, with probability higher than $1-1/n^2$, the spectral estimator $\hxsp_S$ satisfies 
\begin{equation*}
    \underset{Q \in \mathcal{O}}{\textup{min}}\|\hxsp_S -Q{\bf x}^*_{S}\|_{1} \leq C_{lLea} \, \frac{n \sqrt{n \log(n)}}{\left(\Delta_1^{(S)} \wedge \Delta_2^{(S)}\right) \vee 1} \enspace .
\end{equation*}
\end{prop}

Proposition \ref{conj:graphgeo:l1:new}  is based on the fact that the signal matrix $F_{SS}$ is well approximated by a circulant and circular-R matrix, which benefits from nice spectral properties. See Appendix \ref{appendix:spectral} for a proof.

Assumption \eqref{GraphGeo:assump:latentPosition} of the theorem states that $d_{\infty}(\xbf^*,\bPi_{n}) \leq c_a  \sqrt{\log(n)/n}$. Since Lemma \ref{cl:equivBias:bis} ensures that $d_{\infty}(\xbf^*_S,\bPi_{n_0})\leq C[d_{\infty}(\xbf^*,\bPi_{n})  + \sqrt{\log(n)/n}]$ with probability higher than $1-1/n^2$, we get the bound
\begin{equation}\label{ineq:assup:repart:S}
    d_{\infty}(\xbf^*_S,\bPi_{n_0})  \leq C_a \enspace \sqrt{\log(n)/n} \enspace,
\end{equation}which holds with probability higher than $1-1/n^2$.

The $\ell^1$-type localization bound in~\eqref{conj:graphgeo:l1:new} depends on the spectral gap $\Delta_1^{(S)}\wedge \Delta_2^{(S)}$ of the signal matrix $F_{SS}$. We combine the next lemma with the assumption $\Delta_1 \wedge \Delta_2\geq c_b n$ of the theorem to get the following lower bound
\begin{equation}\label{hyp:chech:1:spec}
    \Delta_1^{(S)} \wedge \Delta_2^{(S)}\geq c_b n/4 \enspace,
\end{equation}which holds with probability higher than $1-1/n^2$ , as soon as $n\geq C_{abL}$.

\begin{lem}\label{lem:implication:specgaps:S} If $\Delta_1 \wedge \Delta_2\geq c_b n$ for some constant $c_b>0$, then  with probability higher than $1-1/n^2$ we have $\Delta_1^{(S)} \wedge \Delta_2^{(S)}\geq c_b n/4$ for all $n\geq C_{ab}$ where $C_{abL}$ is a positive quantity depending only on $c_a$, $c_b$, and $c_L$.
\end{lem}

Hence, By (\ref{hyp:chech:1:spec}) and (\ref{ineq:assup:repart:S}) the conditions of Proposition \ref{conj:graphgeo:l1:new} are satisfied. In summary, there exists an event of probability higher than $1-3/n^2$ such that 
\begin{equation}\label{eq:borne:nlogn:spec:l1}
    \underset{Q \in \mathcal{O}}{\textup{min}}\|\hxsp_S -Q{\bf x}^*_{S}\|_{1} \leq C_{lLeab}  \sqrt{n \log(n)} \enspace ,
\end{equation} as soon as $n\geq C_{abL}$. For $n\leq C_{abL}$,  the bound \eqref{eq:borne:nlogn:spec:l1} trivially holds provided that we  adjust the constant $C_{lLeab}$ if necessary.

Since $\hxsp_S$ does not lie in $\bPi_{n_0} \subset \mathcal{C}^{ n_0}$, we cannot directly plug it into the local refinement step defined by~\eqref{def:estim:Linfini}.  Accordingly, the Uniform Approximation (UA) in LS algorithm,  projects $\hxsp_S$ onto $\bPi_{n_0} \subset \mathcal{C}^{{n_0}}$.  The UA outputs a vector $\tilde{\bf x}_S^{(1)}$ in $\bPi_{n_0}$ that is close to the input $\hxsp_S$ -- see Lemma~\ref{lem:approx:estimPos} below. 

Lemma~\ref{lem:approx:estimPos} actually gives a more general result that holds for any input in $\mathbb{R}^{ 2 \times n_0}$ given to UA. For clarity, we write below the UA procedure in full generality.

\medskip

\noindent { \def\arraystretch{1.3}
\begin{tabular}{|l|}
\hline
{\bf   Uniform Approximation (UA) in $\bPi_{n_0}$} \label{alg:LHO:geo:UA} \\
\hline
\begin{minipage}{0.95\textwidth} \centering
\begin{minipage}{0.9\textwidth}

\medskip

 \underline{Input:} $\xbf_{S}=(x_1,\ldots,x_{n_0}) \in \mathbb{R}^{2 \times n_0}$.

\begin{enumerate}
\item Set $z_i=x_i/\|x_i\|_2$, for $i=1,\ldots,n_0$.  

\item Pick any permutation $\si$ such that $z_{\sigma(1)},\ldots,z_{\sigma(n_{0})}$ is in trigonometric order.

\item Set $\displaystyle{\tilde x_{\sigma(i)}=e^{\iota {2\pi (\widehat k+i)\over n_{0}}},\quad \textrm{where}\quad 
\widehat k\in\argmin_{k\in [n_{0}]}\sum_{i=1}^{n_{0}} \left\|e^{\iota {2\pi (k+i)\over n_{0}}}-z_{\si(i)}\right\|_1}.$

\end{enumerate}

\noindent \underline{Output:} $\tilde\xbf_{S} \in \bPi_{n_0}$. ~\\
\end{minipage}%
\end{minipage} \\
\hline
\end{tabular} }

\begin{lem}\label{lem:approx:estimPos}Let ${\bf x}^*_{S}\in \Ccal^{n_0}$. For any input ${\bf x}_S\in \mathbb{R}^{ 2 \times n_0}$, UA returns a vector $\widetilde{{\bf x}}_S\in \bPi_{n_0}$ such that  
\begin{equation*}
    \underset{Q \in \mathcal{O}}{\textup{min}} \|\widetilde{{\bf x}}_S - Q{\bf x}^*_{S}\|_{1} \lesssim  \underset{Q \in \mathcal{O}}{\textup{min}}\|{\bf x}_S - Q{\bf x}^*_{S}\|_{1} + n d_{\infty}({\bf x}^*_S,\bPi_{n_0}) + 1\enspace .
\end{equation*}
\end{lem}

By  (\ref{ineq:assup:repart:S}- \ref{eq:borne:nlogn:spec:l1}) and Lemma \ref{lem:approx:estimPos}, the projection $\widetilde{{\bf x}}_S^{(1)}$  satisfies
\begin{equation}\label{xtilde:proof:spec:L1}
    \underset{Q \in \mathcal{O}}{\textup{min}}\|\widetilde{{\bf x}}_S^{(1)} - Q{\bf x}^*_{S}\|_{1} \lesssim  \underset{Q \in \mathcal{O}}{\textup{min}}\|\hxsp_S - Q{\bf x}^*_{S}\|_{1} + n d_{\infty}({\bf x}^*_S,\bPi_{n_0}) + 1\leq C_{lLeab} \,  \sqrt{n \log(n)} 
\end{equation} with probability higher than $1-3/n^2$.

Finally, we plug  $\widetilde{\bf x}_S^{(1)}$  in the criterion \eqref{def:estim:Linfini}  to localize the remaining points. In other words, we compute, for $i\in \overline{S}$,
\begin{equation}\label{def:estim:Linfini:bis}
    \hat{x}_i^{(2)} = \underset{z \in \Ccal_{n_0}}{\textup{argmin}} \,  \langle A_{i , S}, D(z, \tilde{{\bf x}}_S^{(1)})\rangle\enspace ,
\end{equation}and get the position estimates $\hat{{\bf x}}_{\overline{S}}^{(2)}=[\hat{x}_i^{(2)}]_{i\in S}$.
As a direct consequence of \eqref{xtilde:proof:spec:L1}   and Propositions~\ref{prop:LHO:perf}    (and the equivalence between the $\ell^1$ norm in $\mathbb{R}^2$ and the distance $d_1$ in the sphere $\Ccal$), we arrive at the following uniform bound 
\begin{equation*}
    \min_{Q\in \mathcal{O}}d_{\infty}(\hat{{\bf x}}_{\overline{S}}^{(2)},Q{\bf x}^*_{\overline{S}}) \leq C_{lLeab} \,  \sqrt{\frac{\log(n)}{n}} \enspace ,
\end{equation*}
which holds with probability higher than $1-4/n^2$.

As in section~\ref{subsection:LTSalgo}, we finally rely on a cross-validation scheme to estimate and realign all the positions. This straightforwardly allows us to uniformly localize, with probability higher than $1-9/n^2$, all positions ${\bf x}^*_i$  within an error of the order of $\sqrt{\log(n)/n}$. This concludes the proof of Theorem~\ref{thm:graphgeo:new}. \hfill $\square$

\subsection{Proof of Lemma \ref{lem:implication:specgaps:S} }

Recall that $S=[n_0]$ for the ease of exposition. In order too show that the spectrums of $F$ and  $F_{SS}:= [f(x_i^{*},x_j^{*})]_{i,j\in S}$ are linked together, we introduce an intermediate matrix $F^{(4)}:=[f(x_i^{*(4)},x_j^{*(4)})]_{i,j\in [n]}$ based on the vector ${\bf x}^{*(4)}_S \in \mathcal{C}^{n}$ with coordinates $x_i^{*(4)} = x_{\lceil i/4 \rceil}^{*}$ for $i\in [n_0]$ (where $\lceil  \cdot \rceil$ denotes the ceiling function). 
In other words, we replicate $4$-times each coordinate of the vector ${\bf x}^*_S$ to get the vector ${\bf x}^{*(4)}_S$ of size $n$ which is close to ${\bf x}^*$. 

Let us show first that the spectrums of $F_{SS}$ and $F^{(4)}$ are almost the same. By construction of $F^{(4)}$, each of the $n_0$ eigenvectors of $F_{SS}$ can be transformed into an eigenvector of $F^{(4)}$, by replicating $4$-times the coordinates of these vectors.  Besides, the rank of $F^{(4)}$ is the same as that of $F_{SS}$.  We deduce that all non-zero eigenvalues of $F^{(4)}$ are eigenvalues of $2F_{SS}$. Formally, denoting the eigenvalues of $F^{(4)}$ by $\lambda_{0}'\geq \ldots \geq \lambda_{n-1}'$, and recalling that the eigenvalues of $F_{SS}$ are denoted by $\lambda_{0}^{*(S)}\geq \ldots \geq \lambda_{n_0-1}^{*(S)}$, we have
\begin{equation}\label{prf:linkSpectral:new:eq1}
    2\lambda_{i}^{*(S)}= \lambda_{i}'\ , \quad \quad \text{ for }i=1,\ldots, n_0\enspace .
\end{equation}

We then show that the spectrums of $F^{(4)}$ and $F$ are close. By  \eqref{ineq:assup:repart:S} there is a probability higher than $1-1/n^2$  that $d_{\infty}({\bf x}^*_S,\bPi_{n_0}) \leq C_a \sqrt{\log(n)/n}$. Hence, one can readily check that  
\[
d_{\infty}({\bf x}^{*(4)},\bPi_n)\leq  d_{\infty}({\bf x}^{*(4)},\bPi_{n_0})+ 2\pi/n_0 \leq C'_a \sqrt{\log(n)/n}\enspace .
\] 
Furthermore,  Assumption \eqref{GraphGeo:assump:latentPosition} of the theorem ensures that $$d_{\infty}({\bf x}^{*},\bPi_n) \leq c_a \sqrt{\log(n)/n}\enspace.$$ Therefore, both ${\bf x}^{*(4)}$ and ${\bf x}^{*}$ are close to $\bPi_n$. Since (any) two elements of $\bPi_n$ are equal up to a permutation of their indices $[n]$, we deduce that there exists a permutation $\si$ of $[n]$ satisfying $d_{\infty}({\bf x}^{*(4)},{\bf x}_{\si}^*) \leq C_a \sqrt{\log(n)/n}$. Combining this with the bi-Lipschitz condition \eqref{cond:lipsch} we deduce  that  
\[
    \|F^{(4)} - F_{\si} \|_2 \leq C_a c_L \sqrt{n \log(n)} \enspace .
\]

We are now ready to control the difference between the spectrums of $F$ and $F^{(4)}$. Recalling that $\lambda_{0}^{*}\geq \ldots \geq \lambda_{n-1}^{*}$ denote the eigenvalues of $F$, and since  $F_{\si}$ has the same eigenvalues as $F$, it follows from   Weyl's inequality (see e.g. \cite[page 45]{tao2012topics}) that 
\begin{equation}\label{prf:linkSpectral:new:eq2}
|\lambda_i'-\lambda_i^* | \leq \| F^{(4)}- F_{\si}\|_{op} \leq \| F^{(4)}-F_{\si}\|_2 \leq C_{aL} \sqrt{n \log(n)} \ , 
\end{equation}
for all $i=0,\ldots,n-1$, and some constant $C_{aL}$ depending only on $c_a$ and $c_L$.

Gathering (\ref{prf:linkSpectral:new:eq1}-\ref{prf:linkSpectral:new:eq2}), we conclude that the following implication holds. If $\Delta_1:=\lambda_0^* -\lambda_1^*$ and $\Delta_2:=\lambda_2^* -\lambda_3^*$ satisfies $\Delta \wedge \Delta_2 \geq c_b n$, then $\Delta_1^{(S)}:=\lambda_0^{*(S)} -\lambda_1^{*(S)}$ and $\Delta_2^{(S)}:=\lambda_2^{*(S)} -\lambda_3^{*(S)}$ fulfills $\Delta_1^{(S)} \wedge \Delta_2^{(S)} \geq \frac{c_b}{2} n - C \sqrt{n \log(n)}$ which is larger than $c_b n/4$ as soon as $n \geq n_{ab}$ for $n_{ab}$ the smallest integer satisfying $c_b n \geq 4C_{aL}  \sqrt{n\log(n)}$. \hfill  $\square$


\subsection{Proof of Lemma \ref{lem:approx:estimPos} (Uniform approximation)}

Recall that for a vector ${\bf x}= (x_1, x_2, \ldots,x_{n_0}) \in \Ccal^{ n_0}$, we say that ${\bf x}$ is ordered, if these points are consecutive when one walks on the circle following the trigonometric direction.

We introduce some notation. For any vector ${\bf v}=(v_{1},\ldots,  v_{n_0}) \in \Ccal^{n_0}$, denote $\bPi_{n_0}({\bf v})$ all elements  ${\bf u}=(u_{1},\ldots,  u_{n_0})$ of $\bPi_{n_0}$  such that,  for all permutations $\si$ making $v_{\si(1)},\ldots,  v_{\si(n_0)}$ ordered, the sequence $u_{\si(1)}, \ldots , u_{\si(n_0)}$ is ordered. For any vector ${\bf v}\in \Ccal^{ n_0}$ with distinct values, the set $\bPi_{n_0}({\bf v})$ can be described by a single element ${\bf u} \in \bPi_{n_0}({\bf v})$ and all circular permutations of ${\bf u}$.

The bound of Lemma \ref{lem:approx:estimPos} trivially holds for $n_0\leq 3$. Henceforth, we assume that $n_0\geq 4$. The next lemma is a key element in the proof; it states that re-ordering two vectors is almost optimal for minimizing their $d_1$ distance. This result is fairly classical for real vectors. Here, as the vectors ${\bf v}$ and ${\bf u}$ take their values on $\mathcal{C}$ the proof is slightly more complicated.

\begin{lem}\label{goal:proof:lem:approx:esti:ComplJ} Consider any ${\bf v}\in \Ccal^{n_0}$. Provided that $n_0\geq 4$, we have 
$$\min_{{\bf u}\in \bPi_{n_0}({\bf v})}d_{1}({\bf v}, {\bf u}_{\si}) \lesssim \min_{{\bf u}\in \bPi_{n_0}}d_{1}({\bf v}, {\bf u})\enspace .$$ 
\end{lem}

Recall that $S=\{1,\ldots,n_0\}$ for the ease of exposition. We shall prove the following statement which implies Lemma~\ref{lem:approx:estimPos}. For any ${\bf x}^*_{S}\in \Ccal^{ n_0}$, any input ${\bf x}_S\in \mathbb{R}^{2\times n_0}$,  and any $Q\in \Ocal$, the vector $\widetilde{{\bf x}}_S$ of $\bPi_{n_0}$ fulfills
\begin{equation}\label{target:graphGeo:lemma}
    \|\widetilde{{\bf x}}_S - Q{\bf x}^*_{S}\|_1 \lesssim  \|{\bf x}_S - Q{\bf x}^*_{S}\|_1 + n d_{\infty}({\bf x}^*_S,\bPi_{n_0}) + 1.
\end{equation}

UA computes in step 1  the projection ${\bf z}_S= [x_{i}/\|x_{i}\|_2]_{i\in S}$ of the input ${\bf x}_S$ onto $\Ccal^{ n_0}$.  Given the projection ${\bf z}_S$, UA  picks in step 3 a vector $\tilde{\bf x}_S\in \bPi_{n_0}({\bf z}_{S})$ that has the smallest $\ell^1$-error: 
\begin{equation*}
    \tilde{\bf x}_S\in \underset{{\bf u}\in \bPi_{n_0}({\bf z}_{S})}{\textup{argmin}}\|{\bf z}_{S}-{\bf u}\|_1.
\end{equation*}
 It follows from these definitions and the equivalence between the distance $d$ on $\Ccal$ and $\ell^1$-norm in $\mathbb{R}^2$ that 
\[
\|{\bf z}_S- \tilde{{\bf x}}_S\|_1 = \underset{{\bf v}\in \bPi_{n_0}({\bf z}_S)}{\textup{min}}\|{\bf z}_S-{\bf v}\|_1\lesssim \underset{{\bf v}\in \bPi_{n_0}({\bf z}_S)}{\textup{min}}d_1({\bf z}_S,{\bf v})\enspace .  
\]
Gathering this bound with Lemma~\ref{goal:proof:lem:approx:esti:ComplJ}, we derive that 
\[
 \|{\bf z}_S- \tilde{{\bf x}}_S\|_1 \lesssim \underset{{\bf u}\in \bPi_{n_0}}{\textup{min}}d_1({\bf z}_S,{\bf u})\lesssim  \underset{{\bf u}\in \bPi_{n_0}}{\textup{min}}\|{\bf z}_S-{\bf u}\|_1 \enspace .  
\]
As a consequence, it suffices to exhibit some $u\in \boldsymbol{\Pi}_{n_0}$ such that its $\ell^1$ distance to the projection  ${\bf z}_S$ is small. This is precisely the purpose of the next lemma. 
\begin{lem} \label{goal:proof:lem:approx:esti:new} Consider any matrix $Q\in \mathcal{O}$.  There exists ${\bf y} \in \bPi_{n_0}$ such that 
$$\|{\bf z}_S-{\bf y}\|_1 \lesssim  \|{\bf x}_S - Q{\bf x}^*_{S}\|_1 + n d_{\infty}({\bf x}^*_S,\bPi_{n_0}) + 1\enspace .$$ 
\end{lem}
We conclude that 
$$\|{\bf z}_S- \tilde{{\bf x}}_S \|_1 \lesssim \|{\bf z}_S-{\bf y}\|_1 \lesssim  \|{\bf x}_S - Q{\bf x}^*_{S}\|_1 + n d_{\infty}({\bf x}^*_S,\bPi_{n_0})+1\enspace .$$ 
By triangular inequality, we have  
\begin{equation*}
\| \tilde{{\bf x}}_S-{\bf x}_S \|_1 \leq   \|\tilde{{\bf x}}_S - {\bf z}_S\|_1 +  \|{\bf z}_S-{\bf x}_S\|_1\enspace .    
\end{equation*}The definition of a projection --and the equivalence between the $\ell^1$-norm and the euclidean norm in $\mathbb{R}^2$-- ensure that $$\|{\bf z}_S-{\bf x}_S\|_1 \lesssim \|Q{\bf x}^*_{S} - {\bf x}_S\|_1\enspace ,$$ since ${\bf z}_S$ is the projection of ${\bf x}_S$ on $\Ccal^{ n_0}$ and $Q{\bf x}^*_{S}$ is an element of $\Ccal^{ n_0}$. 
The last three displays allow us  to conclude that $$ \| \tilde{{\bf x}}_S-{\bf x}_S \|_1 \lesssim  \|{\bf x}_S - Q{\bf x}^*_{S}\|_1 + n d_{\infty}({\bf x}^*_S,\bPi_{n_0})+1\enspace ,$$ which gives \eqref{target:graphGeo:lemma} using the triangle inequality again.\hfill $\square$

\subsubsection{Proofs of Lemma \ref{goal:proof:lem:approx:esti:new}}

Let ${\bf x}^{**}_S\in \bPi_{n_0}$ be a closest approximation of ${\bf x}^{*}_{S}$ in $\bPi_{n_0}$, that is, such that $d_{\infty}({\bf x}^{*}_{S},{\bf x}^{**}_S) = d_{\infty}(\xbf^*_S,\bPi_{n_0})$.
The triangular inequality gives
$$\|Q{\bf x}^{**}_S-{\bf z}_S\|_1 \leq  \|Q{\bf x}^{**}_S -{\bf x}_S \|_1 + \|{\bf x}_S -{\bf z}_S\|_1 \lesssim \|Q{\bf x}^{**}_S -{\bf x}_S \|_1\enspace ,$$
where the last inequality comes from the definition of a projection and the equivalence between the $\ell^1$-norm and the euclidean norm in $\mathbb{R}^2$.
By the triangular inequality again, we get 
$$ \|Q{\bf x}^{**}_S -{\bf x}_S \|_1 \leq \|Q{\bf x}^{**}_S -Q{\bf x}^{*}_{S} \|_1 +  \|Q{\bf x}^{*}_{S} -{\bf x}_S \|_1\enspace .$$ An orthogonal transformation preserves the distances. 
$$ \|Q{\bf x}^{**}_S -Q{\bf x}^{*}_{S} \|_1 = \|{\bf x}^{**}_S -{\bf x}^{*}_{S} \|_1 \lesssim d_1({\bf x}^{**}_S,{\bf x}^{*}_{S}) \leq n d_{\infty}({\bf x}^*_S,\bPi_{n_0})\enspace ,$$where we use again the equivalence between the distance $d$ in $\Ccal$ and the $\ell^1$-norm in $\mathbb{R}^2$. Putting everything together, we conclude that 
$$\|Q{\bf x}^{**}_S-{\bf z}_S\|_1 \lesssim  \|{\bf x}_S - Q{\bf x}^*_{S}\|_1 + n d_{\infty}({\bf x}^*_S,\bPi_{n_0})\enspace .$$ 
Although ${\bf x}^{**}_S$ belongs to $\bPi_{n_0}$, this is not necessarily the case for $Q{\bf x}^{**}_S$. Nevertheless, it is easy to check that there exists some $Q'\in \Ocal$ such that $Q'{\bf x}^{**}_S \in \bPi_{n_0}$ and $\|Q'{\bf x}^{**}_S- Q{\bf x}^{**}_S\|_1 \lesssim 1$. Setting 
${\bf y} := Q'{\bf x}^{**}_S\in \bPi_{n_0}$, then we see that  
\begin{equation}\label{eq:claim:order}
\|{\bf y}-{\bf z}_S\|_1 \lesssim  \|{\bf x}_S - Q{\bf x}^*_{S}\|_1 + n d_{\infty}({\bf x}^*_S,\bPi_{n_0}) + 1\enspace ,
\end{equation}
which concludes the proof. 
\hfill $\square$

\subsubsection{Proof of Lemma \ref{goal:proof:lem:approx:esti:ComplJ}}
Fix any vector ${\bf v}\in \mathcal{C}^{n_0}$ and any ${\bf u}\in \bPi_{n_0}$. We shall prove that  $$\min_{{\bf w}\in \bPi_{n_0}({\bf v})}d_1({\bf v},{\bf w})\lesssim  d_1({\bf v},{\bf u})\enspace.$$ Let $\tau$ be a permutation ordering the coordinates of ${\bf v}$ on the unit sphere, meaning that $v_{\tau(1)},\ldots v_{\tau(n_0)}$ is ordered. For simplicity and without loss of generality, assume that $\tau$ is the identity.  Since ${\bf u}\in \bPi_{n_0}$, it there suffices to prove the existence of a permutation $\sigma$ of $[n_0]$ such that $u_{\sigma}$ is ordered and 
\[
d_1({\bf v},{\bf u}_{\sigma})\lesssim  d_1({\bf v},{\bf u})   
\]

Define the set of 'bad' indices $\mathcal{B}=\{i: d(u_i,v_i)\geq \pi/16\}$. If the cardinal of $\mathcal{B}$ is larger than $n_0/2$, then $d_1({\bf u}, {\bf v})\geq n_0\pi/32$ and any vector $ {\bf u}_{\sigma}$ satisfies
$d_1( {\bf v},{\bf u}_{\sigma})\leq n_0\pi\leq  32d_1({\bf v},{\bf u})$. Hence, we assume henceforth that $|\mathcal{B}|\leq n_0/2$. 
First, we focus on the set of 'good' indices $\mathcal{G}= [n_0]\setminus \mathcal{B}$.  We establish the following claim at the end of the proof. 
\begin{cl}\label{claim:permutation}
 There exists  a permutation $\sigma$ of $\mathcal{G}$ such that the sequence $(u_{\sigma(j)})$ with $j\in \mathcal{G}$ is ordered and
 \[
  \sum_{i\in \mathcal{G}}d(u_{\sigma(i)}, v_i)\leq \sum_{i\in \mathcal{G}}d(u_i, v_i)
  \]
\end{cl}
Hence, it is possible to order the restriction of $u$ to $\mathcal{G}$ without increasing the sum of the distances. It remains to transform $\sigma$ into a permutation of  $[n_0]$. We iteratively add elements of $\mathcal{B}$ into $\sigma$. Consider any $i \in \mathcal{B}$. Let $k$ and $l$ be the two consecutive (modulo $n_0$) elements of $\mathcal{G}$ such that $\underline{u}_i$ belongs to the interval $[\underline{u}_{\sigma(k)}, \underline{u}_{\sigma(l)})$ of the torus $\mathbb{R}/(2\pi)$. Let $r$ and $s$ be the two consecutive elements of $\mathcal{G}$ such that $i\in (r,s)$ (where we work modulo $n_0$). Then, we define the permutation $\sigma'$ of $(\mathcal{G}\cup \{i\})$ as follows. 

If $(r,s)= (k,l)$, then we take $\sigma'(j)= \sigma(j)$ if $j\in \mathcal{G}$ and $\sigma'(i)=i$. One readily checks that the sequence $(u_{\sigma'(j)})$ with $j\in \mathcal{G}\cup\{i\}$ is ordered and that $\sum_{j\in \mathcal{G}\cup \{i\}}d(u_{\sigma'(j)},v_j)\leq \sum_{j\in \mathcal{G}\cup \{i\}}d(u_{j},v_j)$. 

Otherwise, we set $\sigma'(i)=\sigma(s)$ and $\sigma'(k)=i$. For $j\in \mathcal{G}$, let $\text{succ}_{\mathcal{G}}(j)$ be the successor of $j\in \mathcal{G}$, that is the smallest index $j'\in \mathcal{G}$ which is larger than $j$ (modulo $n_0$). 
For any $j\in \mathcal{G}$ in the segment $[s,k)$, we set $\sigma'(j)= \sigma(\text{succ}_{\mathcal{G}}(j))$. Besides, we set  $\sigma'(j)=\sigma(j)$ for all $j\in \mathcal{G}$ in the segment $[l, r]$. In other words, we have shifted all elements in the segment $[s,k]$ to successfully include $i$ in the permutation $\sigma'$. It follows from this definition that the sequence $u_{\sigma'(j)}$ with $j\in \mathcal{G}\cup \{i\}$ is ordered. 
By the triangular inequality, we have 
\begin{eqnarray*}
\sum_{j\in \mathcal{G}\cup \{i\}}d(u_{\sigma'(j)},v_j)&= & \sum_{j\in \mathcal{G}\cap [l,r]}d(u_{\sigma(j)},v_j)+  d(u_{\sigma(s)},v_i)+d(u_{i},v_k)\\
&+&\sum_{j\in \mathcal{G}\cap  [s,k)} d(u_{\sigma(\text{succ}_{\mathcal{G}}(j))} , v_j) \\
&\leq & 2\pi + \sum_{j\in \mathcal{G}}d(u_{\sigma(j)} , v_j)+ \sum_{j\in \mathcal{G}\cap  [s,k)} d(u_{\sigma(j)},u_{\sigma(\text{succ}_{\mathcal{G}}(j))} )\\
&\leq & 4\pi + \sum_{j\in \mathcal{G}}d(u_{\sigma(j)} , v_j)\leq 4\pi + \sum_{j\in \mathcal{G}\cup \{i\}}d(u_j , v_j)\enspace , 
\end{eqnarray*}
where we used in the third line that $\sum_{j\in \mathcal{G}\cap  [s,k)} d(u_{\sigma(j)},u_{\sigma(\text{succ}_{\mathcal{G}}(j))})\leq 2\pi$. Indeed, the sequence $(u_{\sigma(j)})_{j\in \mathcal{G}\cap  [s,k)}$ is ordered on the sphere and this sum is therefore equal to the length of the arc $[\underline{u}_{\sigma(s)}, \underline{u}_{\sigma(k)}]$. 

By a straightforward induction, we manage to build a permutation $\overline{\sigma}$ on $[n_0]$ such that $(u_{\overline{\sigma}(j)})_{j\in[n_0]}$ is ordered and 
\begin{eqnarray*}
\sum_{j\in [n_0]}d(u_{\overline{\sigma}(j)},v_j)& \leq &   4\pi\big|\big\{i\in[n_0]\,: \, d(u_i,v_i)\geq \frac{\pi}{16}\big\}\big| + \sum_{j\in [n_0]}d(u_j , v_j)\\
&\leq & 65 \sum_{j\in [n_0]}d(u_j , v_j)\ , 
\end{eqnarray*}
where we used Markov's inequality in the last line. We have shown the desired result. \hfill $\square$

\bigskip 
\begin{proof}[Proof of claim \ref{claim:permutation}]
Without loss of generality, we assume in the proof that $\mathcal{B}=\emptyset$ so that we build a permutation $\sigma$ of $[n_0]$.  Since  $\mathcal{B}=\emptyset$, ${\bf u}\in \bPi_{n_0}$ satisfies $d_{\infty}({\bf u},{\bf v})\leq \pi/16$. We shall iteratively build a permutation $\sigma$ such that ${\bf u}_{\sigma}$ is ordered. Let us first partition the one-dimensional torus $\mathbb{R}/(2\pi)$ into three parts $\mathbb{R}/(2\pi) = \Dcal_1 \cup \Dcal_2 \cup \Dcal_3$ where $\Dcal_s= \left[(s-1)\frac{2\pi}{3}, s\frac{2\pi}{3}\right)$ for $s=1,2,3$.

For $s=1,2,3$, define $I_s=\{i: \underline{v}_i \in \Dcal_s\}$. Since $d_{\infty}({\bf u},{\bf v})\leq \pi/16$, it follows that $\{\underline{u}_i: i\in I_s\}\subset \left[(s-1)\frac{2\pi}{3}-\frac{\pi}{16}, s\frac{2\pi}{3}+\frac{\pi}{16}\right)=\Dcal'_s$. Note that the diameter of $\Dcal'_s$ is smaller  $2\pi/3+\pi/8 < \pi$. 
We have the decomposition
\[
 d_1({\bf v},{\bf u})= \sum_{s=1}^3 \sum_{i\in I_s}d(v_i,u_i)\enspace . 
\]
For $s=1,2,3$, let $\sigma_s$ denote the permutation of $I_s$ such that the sequence $u_{\sigma_s(i)}$ is ordered  when $i$ is in $I_s$.  
Since the diameter of $\Dcal'_s$ is at most $\pi$, the sequence $\underline{u}_{\sigma_s(i)}$ in $\Dcal'_s$ is isometric to an increasing sequence of points in $[0,\pi]\subset \mathbb{R}$ endowed with the absolute value distance.  It goes the same for the ordered sequence $\underline{v}_i$ in $\Dcal'_s$. Next, we use the following classical property. 
\begin{cl}\label{claim:monotnew}
Let $l\geq 1$ be an integer and ${\bf a}, {\bf b}$ be two monotonic vectors of $\mathbb{R}^l$, that is, $a_1 \leq  a_2 \leq  \ldots \leq a_l$ and  $b_1 \leq b_2 \leq \ldots \leq b_l.$ Then, for any permutation $\tau$ of the indices $\{1,\ldots, l\}$, we have $$\sum_{j=1}^l |a_i-b_i|  \leq \sum_{j=1}^l |a_i-b_{\tau(i)}|\quad  \text{ and }\quad \max(|a_i-b_{i}|)\leq \max(|a_i-b_{\tau(i)}|) \enspace .$$ \end{cl}
It follows  that, for $s=1,2,3$, we have 
\[
 \sum_{i\in I_s}d(v_i,u_{\sigma_s(i)})\leq \sum_{i\in I_s}d(v_i,u_{i})\enspace .
\]
Let $\sigma$ be the permutation such that $\sigma(i)=\sigma_s(i)$ if $i\in I_s$. Obviously, we have $d_1({\bf v},{\bf u}_{\sigma})\leq d_1({\bf v},{\bf u})$. Besides, $u_{\sigma}$ is ordered except possibly at the indices  $J_s=\{i : \underline{u}_{\sigma(i)}\in [(s-1)\frac{2\pi}{3}-\frac{\pi}{16}; (s-1)\frac{2\pi}{3}+\frac{\pi}{16}]\}$ with $s=1,2,3$. Since $\max_id(u_{\sigma(i)},v_i)\leq \frac{\pi}{16}$ by the second part of the above claim, all $\underline{u}_{\sigma(i)}$ and $\underline{v}_i$ with $i\in J_s$ belong to an interval of length smaller than $\pi$. Besides, 
\[
 d_1({\bf v},{\bf u}_{\sigma})= \sum_{j\notin (\cup_s J_s)}d(v_i,u_{\sigma_s(i)})+ \sum_{s=1}^3\sum_{i\in J_s}d(v_i,u_{\sigma(i)})\ .
\]
Hence, we can build as previously partitions $\sigma'_s$ of $J_s$ that make $u_{\sigma'_s(\sigma(i))}$ ordered on $J_s$ and so that 
\[
 \sum_{i\in J_s}d(v_i,u_{\sigma'_s(\sigma(i))})\leq \sum_{i\in J_s}d(v_i,u_{\sigma(i)})\ . 
\]
Defining $\overline{\sigma}(i)=\sigma_s'(\sigma(i))$ if $i\in J_s$ for $s=1,2,3$ and $\overline{\sigma}(i)= \sigma(i)$ otherwise, we conclude that ${\bf u}_{\overline{\sigma}}$ is ordered and that $d_1({\bf v},{\bf u}_{\overline{\sigma}})\leq d_1({\bf v},{\bf u})$.
\end{proof}


\subsection{Proof of Proposition \ref{conj:graphgeo:l1:new} }\label{appendix:spectral}

Under the extra assumption that $f$ is geometric, i.e. $f$  satisfies \eqref{def:geometricSetting}, we will show that the estimation error of the spectral algorithm is bounded by $\frac{n \sqrt{n \log(n)}}{\left(\Delta_1^{(S)} \wedge \Delta_2^{(S)}\right) \vee 1}$ in $\ell^1$-type norm. The proof consists in approximating the signal $F_{SS}$ by a circulant and circular-R matrix (Definition \ref{defi:circulan}) whose spectrum is known (Lemma \ref{lem:spectrumformula}) and provides information on the latent positions ${\bf x}^{*}_{S}$. The difference between the spectrums of $F_{SS}$ and $A_{SS}$ will be bounded using the Davis-Kahan perturbation bound.

\subsubsection{Preliminaries: general facts on R-matrices}\label{section:def:Rcicular}  Let us start by introducing the notion of circulant matrix (see \cite{gray2006toeplitz,recanati2018reconstructing}). 
\begin{defi}\label{defi:circulan} 
For any integer $n\geq 1$, a symmetric matrix $M\in \mathbb{R}^{n\times n}$ is circulant if there exists a vector ${\bf a }$ of size $n$ such that $M_{ij} = a_{|i-j|}$ and 
\begin{align*}
  \forall k=1,\ldots,n-1, \qquad   a_k = a_{n-k}. 
\end{align*} Moreover, $M$ is a circulant and circular R-matrix if the above holds and the sequence $(a_{j})_{0 \leq j\leq \lfloor n/2 \rfloor}$ is non-increasing.
\end{defi}

The spectrum of circulant matrices is known --see \cite{gray2006toeplitz} and the references therein, which allows to easily deduce the spectrum of symmetric circulant matrices, see Proposition C.4 from \cite{recanati2018reconstructing}. For clarity, we recall this result  below --with a  small correction on the first coordinate of the eigenvector ${\bf v}^{(m)}$.

\begin{lem}[spectrum of symmetric circulant matrices]\label{lem:spectrumformula}
Let $M\in \mathbb{R}^{n \times n}$ be any symmetric circulant matrix associated to a vector ${\bf a}$ (as above). 
\begin{itemize}
 \item  For $n=2p+1,$ the eigenvalues of $M$ are equal to
$$\alpha_m = a_0 + 2 \sum_{j=1}^p a_j \cos\left(j \frac{2\pi m}{n}\right)\enspace , \qquad m=0,\ldots, p\enspace , $$
where each $\alpha_m$, for $m=1,\ldots,p,$ has multiplicity $2$ and is associated with the two following eigenvectors \begin{align}\label{def:eigenvector}
    {\bf u}^{(m)}&= (1,\cos(2\pi m/n),\ldots, \cos((n-1)2\pi m/n))\\
    {\bf v}^{(m)}&= (0,\sin(2\pi m/n),\ldots, \sin((n-1)2\pi m/n)) \enspace .\nonumber
\end{align}
For $m=0$, $\alpha_0$ has multiplicity 1 and is associated to  ${\bf u}^{(0)}=(1,\ldots,1).$
\item  For $n=2p,$
$$\alpha_m = a_0 + 2 \sum_{j=1}^{p-1} a_j \cos(j \frac{2\pi m}{n}) + a_p \cos(\pi m)\enspace , \qquad m=0,\ldots, p\enspace ,$$where each $\alpha_m$, for $m=1,\ldots,p-1,$ is associated with the two eigenvectors in \eqref{def:eigenvector}. The eigenvalue $\alpha_p$ is associated with  ${\bf u}^{(p)}= (1,-1,\ldots,1,-1)$. 
For $m=0$, $\alpha_0$ has multiplicity 1 and is associated to  ${\bf u}^{(0)}=(1,\ldots,1).$
 \end{itemize}

\end{lem}

If the vector ${\bf a}$ has nonnegative entries, then $\alpha_0$ is obviously the largest eigenvalue. The next lemma ensures that, for circular R-matrices, $\alpha_1$ is the second largest eigenvalue.  Its proof can be found in \cite[Proposition C.5]{recanati2018reconstructing}. 

\begin{lem}[second largest eigenvalue]\label{lem:strictEigenvaluesOptim}For any symmetric and circulant circular R-matrix, with nonnegative entries and eigenvalues $\{\alpha_m\}$ for $m=0,\ldots,\lfloor n/2 \rfloor$ (as defined in Lemma \ref{lem:spectrumformula}), we have $ \alpha_1 \geq \alpha_j$ for all $j=2,\ldots,\lfloor n/2 \rfloor$.
\end{lem}

\noindent 
{\bf Remark}:  if $a_j=g(j\frac{2\pi}{n})$, then the discrete  Fourier transform $\Fcal_{k,n}(g)$ as defined in \eqref{spectre:fourier:paper} satisfies $\Fcal_{k,n}(g) = \alpha_k$, for all for $k=0,\ldots,p$. In addition,  $\Fcal_{n-k,n}(g) = \Fcal_{k,n}(g)$ for all $k=1,\ldots,p$.

\subsubsection{Main Proof of Proposition \ref{conj:graphgeo:l1:new}} \label{sec:prop:VSA}

Recall that $S\subset \{1,2,\ldots,n\}$ satisfies $|S|=n_0 = n/4$. For the ease of exposition we assume that $S =\{1,2,\ldots,n_0\}$ and we only consider the case where  $n_0$ is odd (the case of even $n_0$ being similar). Thus, we write $n_0=2p+1$ in the following. 
If $\Delta_1^{(S)} \wedge \Delta_2^{(S)} \leq C_{lLea} \sqrt{n \log(n)}$, then the bound in Proposition \ref{conj:graphgeo:l1:new} trivially holds
\begin{equation*}
    \underset{Q \in \mathcal{O}}{\textup{min}}\|\hxsp_S-Q{\bf x}^*_{S}\|_{1}\leq 2 n_0 \leq  C_{lLea} \frac{n \sqrt{n \log(n)}}{\left(\Delta_1^{(S)} \wedge \Delta_2^{(S)}\right) \vee 1} \enspace .
\end{equation*} 
We assume therefore that  $\Delta_1^{(S)} \wedge \Delta_2^{(S)} \geq C_{lLea} \sqrt{n \log(n)}$ for a quantity $C_{lLea}$ that will be set later. By definition of $\Delta_1^{(S)}$ and $\Delta_2^{(S)}$, this means that \begin{equation}\label{bound:spec}
    |\lambda_0^{*(S)}- \lambda_1^{*(S)}| \wedge |\lambda_2^{*(S)} - \lambda_3^{*(S)}|\geq  C_{lLea} \sqrt{n \log(n)}.
\end{equation}

Let ${\bf u}^{(1)}=(u^{(1)}_1,\ldots,u^{(1)}_{n_0})$ and ${\bf v}^{(1)}=(v^{(1)}_1,\ldots,v^{(1)}_{n_0})$ denote the eigenvectors of a circular and circulant $R$-matrix as described in Lemma \ref{lem:spectrumformula}. For any matrix $M=(m_{ij})$, we write $\|M\|_{\infty }$ its entry-wise $l_{\infty}$ norm, that is $\|M\|_{\infty } = \textup{max}_{ij} \,  m_{ij}$. Recall that $F_{SS}:=[f(x_i^*,x_j^*)]_{i,j\in S}$ .

\begin{lem}\label{cl:approxRmatrix}
There exist a permutation $\si$ and a circulant circular R-matrix $R$ with nonnegative entries  such that the following inequality holds $\|F_{SS}-R_{\si}\|_{\infty} \leq C_{lLea} \sqrt{ \log(n)/n}$. Besides, the vector ${\bf x}^{**}_S\in \bPi_{n_0}$ defined by $x^{**}_i := (u^{(1)}_{\si(i)},  v_{\si(i)}^{(1)})$ for $i=1,\ldots, n_0$ satisfies 
\begin{equation}\label{eq:approx}
    \min_{Q\in \mathcal{O}}d_{\infty}({\bf x}^{**}_S, Q{\bf x}^*_{S})\leq C_a \sqrt{\frac{\log(n)}{n}}  \enspace .
\end{equation}
\end{lem}

Denote $\lambda_0 \geq \ldots \geq \lambda_{n_0-1}$ the eigenvalues of $R$. Lemmas  \ref{lem:spectrumformula} and \ref{lem:strictEigenvaluesOptim}  ensure that 
\begin{equation}\label{order}
    \lambda_0 = \alpha_0 \geq \lambda_1 = \alpha_1 \geq \lambda_2 = \alpha_1 \geq \lambda_j \enspace ,
\end{equation} where $\lambda_j \in \{\alpha_2,\ldots,\alpha_{\lfloor n_0/2 \rfloor}\}$ for all $j=3,\ldots,n_0-1$.

Lemma \ref{cl:approxRmatrix} ensures that there exists a constant $C''_{lLea}$  depending only on $c_l, c_L, c_e, c_a$ such that $\|F_{SS}-R_{\si}\|_2 \leq C''_{lLea} \sqrt{n \log(n)}$. Since  $R_{\si}$ has the same eigenvalues as $R$, it follows from   Weyl's inequality (see e.g. \cite[page 45]{tao2012topics}) that 
\begin{equation}\label{ineq:diff:2spec}
|\lambda_i^{*(S)} - \lambda_i| \leq \| F_{SS}-R_{\si}\|_{op} \leq \| F_{SS}-R_{\si}\|_2 \leq C''_{lLea} \sqrt{n \log(n)} \ , 
\end{equation}
for all $i=0,\ldots,n_0-1$.  

If the constant $C_{lLea}$ in \eqref{bound:spec} is chosen  as $4C''_{lLea}$ where $C''_{lLea}$ is introduced in \eqref{ineq:diff:2spec}, it follows that 
\begin{eqnarray}\nonumber
   \lambda_0- \lambda_1 &\geq& (\lambda_0^{*(S)}- \lambda_1^{*(S)})- (\lambda_0^{*(S)} - \lambda_0)- ( \lambda_1 -\lambda_1^{*(S)} )\\
   &\geq& (C_{lLea}-2C_{lLea}'') \sqrt{n \log(n)}\\ &\geq& \frac{C_{lLea}}{2} \sqrt{n \log(n)}\enspace , \nonumber
\end{eqnarray}
and similarly,
\begin{equation*}
    \lambda_2- \lambda_3 \geq (\lambda_2^{*(S)}- \lambda_3^{*(S)})- (\lambda_2^{*(S)} - \lambda_2)- ( \lambda_3 -\lambda_3^{*(S)} ) \geq \frac{C_{lLea}}{2}\sqrt{n \log(n)}\enspace .
\end{equation*}

Since the eigenvectors $(\sqrt{2/n_0}){\bf u}^{(1)}$ and $(\sqrt{2/n_0}){\bf v}^{(1)}$ of $R$ are orthonormal (see Lemma \ref{cl:orthonormal:vector} below), the vectors $(\sqrt{2/n_0}){\bf u}^{(1)}_{\si}$ and $(\sqrt{2/n_0}){\bf v}^{(1)}_{\si}$ are orthonormal eigenvectors of $R_{\si}$, with the same eigenvalue $\lambda_1=\lambda_2 =\alpha_1$.

\begin{lem}\label{cl:orthonormal:vector}The vectors $(\sqrt{2/n_0}){\bf u}^{(1)}$ and $(\sqrt{2/n_0}){\bf v}^{(1)}$ are orthonormal.
\end{lem}

Next, we state a variant of  Davis-Kahan perturbation bound \cite[see Theorem 2]{yu2015useful}.

\begin{lem}[Davis-Kahan]\label{DavisKahan}Let $M$, $\hat{M}\in \mathbb{R}^{n_0\times n_0}$ be two symmetric matrices, with eigenvalues $\lambda_0\geq \ldots \geq \lambda_{n_0-1}$ and  $\hat{\lambda}_{0} \geq \ldots \geq \hat{\lambda}_{n_0-1}$ respectively. Fix $0\leq r\leq s \leq n_0-1$ and assume that $(\lambda_{r-1}-\lambda_r) \wedge (\lambda_s-\lambda_{s+1})> 0,$ where $\lambda_{-1}= \infty$ and $\lambda_{n}=-\infty.$ Let $d=s-r+1,$ and let $\textup{{\bf V}}=({\bf v}_r,{\bf v}_{r+1},\ldots,{\bf v}_s)\in \mathbb{R}^{n_0 \times d}$ and $\hat{\textup{{\bf V}}}=(\hat{\bf v}_r,\hat{\bf v}_{r+1},\ldots,\hat{\bf v}_s)\in \mathbb{R}^{n_0 \times d}$ have orthonormal columns satisfying $M{\bf v}_j=\lambda_j{\bf v}_j$ and $\hat{M}\hat{\bf v}_j= \hat{\lambda}_j \hat{\bf v}_j$ for $j=r,r+1,\ldots,s$. Then, there exists an orthogonal matrix $Q\in \mathbb{R}^{d\times d}$ such that $$\|\hat{\textup{{\bf V}}}Q-\textup{{\bf V}}\|_2 \leq \sqrt{8d} \frac{\|\hat{M}-M\|_{op}}{(\lambda_{r-1}-\lambda_r) \wedge (\lambda_{s}-\lambda_{s+1})}\enspace .$$
\end{lem}

 The assumptions of Lemma \ref{DavisKahan} are therefore fulfilled for the orthonormal eigenvectors $(\sqrt{2/n_0}){\bf u}^{(1)}_{\si}$ and $(\sqrt{2/n_0}){\bf v}^{(1)}_{\si}$ and the positive spectral gaps $(\lambda_0 -\lambda_1) \wedge (\lambda_2-\lambda_3) > 0.$ Hence, for $\hxsp_S$ and ${\bf x}^{**}_S = ({\bf u}^{(1)}_{\si},{\bf v}^{(1)}_{\si})^T$ in $\mathbb{R}^{2 \times n}$, Lemma \ref{DavisKahan} entails 
\begin{equation*}
  \sqrt{ \frac{2}{n_0}} \|Q\hxsp_S - {\bf x}^{**}_S \|_2 \lesssim  \frac{\|A_{S,S}-R_{\si}\|_{op}}{(\lambda_0 - \lambda_1) \wedge (\lambda_2 - \lambda_{3})}
\end{equation*} for some $Q\in \Ocal.$

It remains to control $\|A_{S,S}-R_{\si}\|_{op}$ and the spectral gap. Since $A_{ii}=0$ for all $i$ (by convention), we introduce the matrix $F_{SS}^{(0)}$ such that $F_{ii}^{(0)}=0$ for all $i$, and $F_{ij}^{(0)}=F_{ij}$ for all $i \neq j$.
\begin{equation}
\label{eq:2_dec} 
\|A_{SS}-R_{\si}\|_{op} \leq \|A_{SS}-F_{SS}^{(0)}\|_{op} + \|F_{SS}^{(0)}-F_{SS}\|_{op} + \|F_{SS}-R_{\si}\|_2\enspace , 
\end{equation}
using the triangular inequality and 
the fact that the operator norm is smaller than the Frobenius norm. The second term $\|F_{SS}^{(0)}-F_{SS}\|_{op}$ is smaller than $1$ since $F_{SS}-F_{SS}^{(0)}$ is the diagonal matrix with diagonal coefficients $f(x^*_i,x_i^*) \leq 1$, the last inequality coming from $f\in \mathcal{BL}[c_l,c_L,c_e]$.
To control the operator norm of the noise matrix, we shall use the following result \cite[Corollary 4.4.8]{vershynin2018high}. See the same reference for the definition of sub-Gaussian norms $\|.\|_{\psi_2}$.

\begin{lem}[norm of symmetric matrices with sub-gaussian entries]\label{thm:perturbationBound} Let $A$ be an $n_0 \times n_0$ symmetric random matrix whose entries $A_{ij}$ on and above the diagonal are independent mean-zero sub-gaussian random variables. Then, for any $t >0,$ we have$$\|A\|_{op} \leq C K (\sqrt{n}+t)$$ with probability at least $1-4e^{-t^2}.$ Here $K = \textup{max}_{i,j} \|A_{i,j}\|_{\psi_2}.$
\end{lem}

Applying the above lemma with $t = C \sqrt{\log(n)}$ (for a large enough numerical constant $C$) to the difference $A_{SS}-F_{SS}^{(0)}$, we obtain $\|A_{SS}-F_{SS}^{(0)}\|_{op} \lesssim  \sqrt{n}$ with probability higher than $1-1/n^2$.

Together with Lemma \ref{cl:approxRmatrix} and the bound \eqref{eq:2_dec}, we deduce that  $\|A_{SS}-R_{\si}\|_{op} \leq C_{lLea}  \sqrt{n \log(n)}$, so that
\begin{equation*}
  \|Q\hxsp_S - {\bf x}^{**}_S \|_2\leq C_{lLea}  \frac{n \sqrt{\log(n)}}{(\lambda_0 - \lambda_1) \wedge (\lambda_2 - \lambda_{3})} \enspace .
\end{equation*} Then, we deduce from the Cauchy-Schwarz inequality that
\begin{equation*}
    \underset{Q' \in \mathcal{O}}{\textup{min}}\|\hxsp_S - Q'{\bf x}^{**}_S \|_{1}\leq \|Q\hxsp_S - {\bf x}^{**}_S \|_{1} \leq C_{lLea}  \frac{n \sqrt{n\log(n)}}{(\lambda_0 - \lambda_1) \wedge (\lambda_2 - \lambda_{3})} \enspace .
\end{equation*}
The bounds \eqref{bound:spec} and \eqref{ineq:diff:2spec} for $C_{lLea}=4C''_{lLea}$ allow us to replace the above spectral gaps by $(\lambda_0^{*(S)}- \lambda_1^{*(S)})\wedge (\lambda_2^{*(S)}- \lambda_3^{*(S)})$.
Finally, by \eqref{eq:approx} and the equivalence between the distance $d$ in $\Ccal$ and the $\ell^1$-norm in $\mathbb{R}^2$, we have $\underset{Q \in \mathcal{O}}{\textup{min}}\|{\bf x}^{*}_{S} - Q{\bf x}^{**}_S\|_{1}\leq C_a \sqrt{n\log(n)}$. Since $\lambda_0\leq n$ (all the entries of $F_{SS}$ belong to $[0,1]$), we then deduce from the triangular inequality that
\[
    \underset{Q \in \mathcal{O}}{\textup{min}}\|\hxsp_S - Q{\bf x}^{*}_{S} \|_{1}  \leq C'_{lLea} \frac{n \sqrt{n \log(n)}}{(\lambda^{*(S)}_0 - \lambda^{*(S)}_1) \wedge (\lambda^{*(S)}_2 - \lambda^{*(S)}_{3})}\enspace .
\]
Proposition \ref{conj:graphgeo:l1:new} is proved.  \hfill $\square$

\subsubsection{Proofs of technical lemmas}

\begin{proof}[Proof of Lemma \ref{cl:approxRmatrix}] In this proof, we replace the notation $F_{SS}$ by $F_{({\bf x}^*_{S},f)}$ for clarity.
 Since the vector ${\bf x}^{*}_{S}$ satisfies \eqref{GraphGeo:assump:latentPosition:S}, there exists ${\bf x}_S\in \bPi_{n_0}$ such that the following inequality holds
\begin{equation}\label{eq:upper_l_infini}
d_{\infty}({\bf x}^{*}_{S},{\bf x}_S)\leq C_a \sqrt{\log(n)/n}\enspace .
\end{equation}
Combining this with the bi-Lipschitz condition \eqref{cond:lipsch}, we get 
$\max_{ij \in [n_0]}|f(x_i^*,x_j^*) - f(x_i,x_j) | \leq C_{Lea} \sqrt{\log(n)/n}$, that is, 
\begin{equation}\label{ineq:proof:cl:spec:xstarstar}
    \|F_{({\bf x}^{*}_{S},f)} - F_{({\bf x}_S,f)} \|_{\infty} \leq C_{Lea} \sqrt{\frac{\log(n)}{n}} \enspace,
\end{equation}
for the matrices $F_{({\bf x}^*_{S},f)} := [f(x_i^*,x_j^*)]_{i,j\in[n_0]}$ and  $F_{({\bf x}_S,f)} := [f(x_i,x_j)]_{i,j\in[n_0]}$.

Recall that $S=\{1,\ldots,n_0\}$ for the ease of exposition. Let $\tau$ be some permutation that orders  ${\bf x}_S=(x_1,\ldots,x_{n_0})$ on the unit sphere, that is, such that $x_{\tau(1)},\ldots ,x_{\tau(n_0)}$ is ordered. Then, one can observe that the matrix $F_{({\bf x}_S,f), \tau}$ is symmetric circulant since $f$ is a symmetric function which satisfies the geometric condition \eqref{def:geometricSetting} with respect to the geodesic distance $d$ on the unit sphere $\Ccal$.

The matrix $F_{({\bf x}_S,f), \tau}$ is therefore defined by a single vector ${\bf a}$ of size $n_0$ as in Definition~\ref{defi:circulan} of circulant matrices. This vector satisfies $a_s= g(2\pi s/n_0)$ for $s =0,\ldots ,\lfloor n_0/2\rfloor$, where we recall that $g(d(x,y)) =f(x,y)$ in the geometric setting.  From the Lipschitz condition \eqref{cond:lipschLower}, we deduce that ${\bf a}$ satisfies some kind of weak non-increasing condition, that is $a_t \geq a_s \geq 0$ for all $0 \leq t < s \leq \lfloor n_0/2 \rfloor$ such that $s-t \geq C_{le} \sqrt{n\log(n)}$. 

From the bi-Lipschitz condition \eqref{cond:lipsch}, it is easy to see that ${\bf a}$ can be uniformly approximated by a non-increasing vector ${\bf a}'$ such that $\max_j|a_j -a'_j| \leq C_{lLe} \sqrt{\log(n)/n}$. 
Denoting $R$ the circulant circular R-matrix based on the vector ${\bf a}'$, this means that $\max_{ij}|R_{ij}-f(x_{\tau(i)},x_{\tau(j)})| \leq C_{lLe} \sqrt{\log(n)/n}\enspace .$ Hence,
$$\|F_{({\bf x}_S,f)}-R_{\tau^{-1}}\|_{\infty} = \|F_{({\bf x}_S,f), \tau}-R\|_{\infty} \leq C_{lLe} \sqrt{\frac{\log(n)}{n}}\enspace .$$ 
The first result of Lemma \ref{cl:approxRmatrix} is a consequence of  \eqref{ineq:proof:cl:spec:xstarstar} and the last display, setting $\si = \tau^{-1}.$

Next, by definition of $\tau$, the vector  $({\bf x}_{S})_{\sigma^{-1}}= ({\bf x}_S)_{\tau}$ is ordered, and  it therefore equals any other ordered vector in  $\bPi_{n_0}$  up to an orthogonal transformation. Hence, we have  $[({\bf x}_S)_{\sigma^{-1}}]_i= Q({u}^{(1)}_i, {v}^{(1)}_i)$ for some orthogonal transformation $Q$ in $\mathbb{R}^2$, by definition of   ${\bf u}^{(1)}$ and ${\bf v}^{(1)}$. Equivalently, we have $({\bf  x}_S)_i= Q( ({\bf u}_{\sigma}^{(1)})_i, ({\bf  v}_{\sigma}^{(1)})_i)$. Then, we conclude again from  \eqref{eq:upper_l_infini} that the vector ${\bf x}^{**}_S:=({\bf u}_{\sigma}^{(1)}, {\bf v}_{\sigma}^{(1)})$ satisfies the second result of the lemma.  
 
\end{proof}

\begin{proof}[Proof of Lemma \ref{cl:orthonormal:vector}]
Since $\sum_{k=0}^{n_0-1}e^{\iota 4\pi k/n_0}=0$, we have $\sum_{k=0}^{n_0-1} \cos(4\pi k/n_0)= 0$ and $\sum_{k=0}^{n_0-1} \sin(4\pi k/n_0)$ $= 0$. Then, combining with the trigonometric formulas, $\cos(2x)= 2\cos^2(x)-1$, and $\sin(2x)= 2\cos(x)\sin(x)$, we get
\begin{align}\label{cosSum}
    \|{\bf u}^{(1)}\|_2^2 &= \sum_{k=0}^{n_0-1} \cos^2(2\pi k/n_0)= \frac{n_0}{2} \enspace , \\
    \langle {\bf u}^{(1)}, {\bf v}^{(1)}\rangle  &= \sum_{k=0}^{n_0-1} \cos(2\pi k/n_0)\sin(2\pi k/n_0)=  0 \nonumber \enspace .
\end{align}
Besides, $\|{\bf u}^{(1)}\|_2^2 + \|{\bf v}^{(1)}\|_2^2 = n_0$ since $(u^{(1)}_i, v^{(1)}_i)$ for any $i\in [n_0]$ is a point of the unit sphere $\Ccal$. The combination with \eqref{cosSum} leads to $\|{\bf v}^{(1)}\|_2^2 = n_0/2$. 
 \end{proof}

\subsection{Proof of Lemma \ref{lem:fourier}}
Similarly to the proof of Lemma~\ref{cl:approxRmatrix}, we consider a vector ${\bf x}\in \bPi_n$ achieving $d_{\infty}({\bf x}^*,{\bf x})=d_{\infty}({\bf x}^*,\bPi_n)\leq C_a\sqrt{\log(n)/n}$. By the Bi-Lipschitz condition, the matrices $F_{{\bf x^*},f}$ and $F_{{\bf x},f}$ satisfy
\[
\|F_{{\bf x},f}-F_{{\bf x^*},f}\|_{\infty}\leq C_{Lea}\sqrt{\frac{\log(n)}{n}}\enspace . 
\] 
Since $f$ is geometric and ${\bf x}$ belongs to $\bPi_n$, it follows that, up to a permutation, $F_{{\bf x},f}$ is a symmetric circulant matrix associated to the vector $a_j = g(j\frac{2\pi}{n})$ for $j= 0, \ldots ,\lfloor \frac{n}{2} \rfloor$. It then follows from Lemma~\ref{lem:spectrumformula}, that the eigenvalues of $F_{{\bf x},f}$ are equal to the discrete Fourier transform $\Fcal_{m,n}(g)$ of $g$. 

The sequence $(a_j)$, for $j=0, \ldots, \lfloor \frac{n}{2} \rfloor$ is not non-increasing  because the function $g$ is not exactly decreasing with respect to the distance. Still, arguing as in the proof of Lemma~\ref{cl:approxRmatrix}, we can build an non-increasing sequence ${\bf a}'$ satisfying $\max_j|a_j -a'_j| \leq C_{lLe} \sqrt{\log(n)/n}$. The eigenvalues of the corresponding circulant and circular $R$-matrix $R$ are also given by Lemma~\ref{lem:spectrumformula}. We denote them $\alpha_0, \alpha_1,\ldots, \alpha_{\lfloor n/2\rfloor}$. It follows from the definition of the Fourier transform that 
\[
    |\Fcal_{m,n}(g) - \alpha_{m}| \leq C_{lLea} \sqrt{n \log(n)} \enspace, \qquad 0\leq m\leq  \lfloor \frac{n}{2} \rfloor\enspace .    
\]
Hence, the gaps in the Fourier transform $\Phi_{1} = \Fcal_{0,n}(g) - \Fcal_{1,n}(g)$ and 
$\Phi_{2} =  \min_{j =2,\ldots,\lfloor n/2\rfloor} \ \, \Fcal_{1,n}(g) - \Fcal_{j,n}(g)$ satisfy
\begin{equation*}
  \Big|\Phi_1 - (\alpha_0-\alpha_1)\Big| \vee \Big|\Phi_2 - \underset{m =  2,\ldots,\lfloor \frac{n}{2} \rfloor}{\textup{min}} \, (\alpha_1-\alpha_m)\Big| \leq C_{lLea} \sqrt{n \log(n)} \enspace.
\end{equation*}
To conclude, it remains to prove that 
\begin{align} \label{eq:obj2_lem_fourier}  
    \Big|(\alpha_0-\alpha_1)- \Delta_1 \Big|&\leq  C'_{lLea} \sqrt{n \log(n)}\ ;\\ \nonumber
    \Big|\underset{m =  2,\ldots,\lfloor \frac{n}{2} \rfloor}{\textup{min}} \, (\alpha_1-\alpha_m)- \Delta_2 \Big| &\leq C'_{lLea} \sqrt{n \log(n)} \enspace .    
\end{align}
By Lemma~\ref{lem:strictEigenvaluesOptim}, we have $\alpha_0 > \alpha_1 \geq \max_{j=2,\ldots,\lfloor \frac{n}{2} \rfloor }\alpha_j$. Hence, if we denote $\lambda_0\geq \lambda_1\ldots \geq \lambda_{n-1}$ the ordered eigenvalues of $R$ we have 
\begin{equation}\label{eq:transfo_eigen}
    \lambda_0 = \alpha_0 \geq \lambda_1 = \alpha_1 \geq \lambda_2 = \alpha_1 \geq \lambda_3 \enspace ,
\end{equation} where $\lambda_3=\max\{\alpha_2,\ldots,\alpha_{\lfloor n/2 \rfloor}\}$. By definition of ${\bf a}'$ and ${\bf x}$, there exists a permutation $\tau$ of $[n]$ such that 
\[
\|F_{{\bf x}^*,f}-R_{\tau}\|_{\infty}\leq C_{lLea}\sqrt{\frac{\log(n)}{n}}\enspace ,     
\]
which implies that $\|F_{{\bf x}^*,f}-R_{\tau}\|_{op}\leq \|F_{{\bf x}^*,f}-R_{\tau}\|_{2}\leq C_{lLea}\sqrt{n\log(n)}$. Denoting $\lambda_0^*\geq \lambda_1^*\geq \lambda_2^*\geq \ldots$ the ordered eigenvalues of $F_{{\bf x}^*,f}$, we deduce from Weyl's inequality that 
\[
\max_{j=0,\ldots, n-1}|\lambda^*_j-\lambda_j|\leq     C_{lLea}\sqrt{n \log(n)}\enspace .
\]
Together with~\eqref{eq:transfo_eigen}, we deduce that $\Delta_1= \lambda_0^*-\lambda_1^*$ and $\Delta_2=\lambda_2^*-\lambda_3^*$ satisfy~\eqref{eq:obj2_lem_fourier} which concludes the proof.

\subsection{Proof of Corollary \ref{coro:gapspectral} (spectral gap for affine functions)}
\label{proof:specGap}

We will show that, for $n$ large enough,  the gaps in the Fourier    $\Phi_{1} := \Fcal_{0,n}(g) - \Fcal_{1,n}(g)$ and 
$\Phi_{2}: =  \min_{j =2,\ldots,\lfloor n/2\rfloor} \ \, \Fcal_{1,n}(g) - \Fcal_{j,n}(g)$ are at least of the the order of $n$.  Corollary~\ref{coro:gapspectral} will then follow directly from Theorem~\ref{thm:graphgeo:new}  and  Lemma~\ref{lem:fourier}.

Recall that the $m$-th coefficient Fourier transform is defined as
$$\Fcal_{m,n}(g) =  \sum_{j=0}^{n-1} g\left(j \frac{2\pi}{n} \right) \cos\left(j \frac{2\pi m}{n}\right)\enspace . $$
For simplicity, we only consider the case where $n$ is odd -- the case of even $n$ being similar. Let $n=2p+1$ with $p\geq 2$. Using the fact that $ g(x) = 1-x/(2\pi)$, we get
\begin{equation}\label{spectre:fourier:new}
\Fcal_{m,n}(g) = 1  + 2 \sum_{j=1}^p (1- \frac{j}{n}) \cos\left(j \frac{2\pi m}{n}\right)\enspace , \qquad m=0,\ldots, p\enspace . 
\end{equation}
For convenience, $\Fcal_{m,n}(g)$ is denoted by $\alpha_m$ in the sequel. 
Let us show that, for $n$ large enough,  $(\alpha_0-\alpha_1) \gtrsim n$ and $\min_{m\geq 2}(\alpha_1-\alpha_m) \gtrsim n$. For $m=0$,  
\begin{equation}\label{eq:upper_alpha_0}
    \alpha_0  = 1 + 2p - \frac{2}{n} \sum_{j=1}^p  j  = 1+ 2p - \frac{p(p+1)}{n} = n - \frac{(n-1)(n+1)}{4n} = \frac{3n}{4} + \frac{1}{4n} \enspace .
\end{equation}
For $m\geq 1$, we can still work out explicitly $\alpha_m$. 
\begin{align}\label{alphaCoro}
 \alpha_m= - 2\sum_{j=0}^{p}\frac{j}{n} \cos\left(j \frac{2\pi m}{n}\right) &= -\frac{2}{n} Re\left[\sum_{j=1}^p je^{\iota j 2\pi m/n}\right]\\ \nonumber
 &= -\frac{2}{n} Re\left[-\iota f'\left(\frac{2\pi m}{n}\right)\right] 
 \enspace,
\end{align}
where the function $f$ is defined as $f(x):= \sum_{j=0}^p   e^{\iota j x} = \frac{e^{\iota (p+1)x }-1}{e^{\iota x}-1}$ for $x\in (0, \pi)$. We work out $f'(x)$:
\begin{align*}
- \iota f'(x) &= \frac{(p+1)e^{\iota (p+1)x }(e^{\iota x}-1) - e^{\iota x }(e^{\iota (p+1)x }-1)}{(e^{\iota x}-1)^2}\\
&=  -\frac{(p+1)e^{\iota px }(e^{\iota x}-1) - (e^{\iota (p+1)x }-1)}{4\sin^2(x/2)},\\
&= - \iota \frac{(p+1)e^{\iota (p+\frac{1}{2})x }\sin(\frac{x}{2}) - e^{\iota \frac{p+1}{2}x }\sin(\frac{p+1}{2}x)}{ 2\sin^2(x/2) },
\end{align*}
where the second line follows from  $(e^{\iota x}-1)^2 = -4\sin^2(x/2) e^{\iota x}$. Hence,
\begin{align*}
Re\left[-\iota f'(x)\right] = \frac{(p+1)\sin((p+\frac{1}{2})x)\sin(\frac{x}{2}) - \sin^2(\frac{p+1}{2}x)}{ 2\sin^2(x/2) }\enspace.
\end{align*}
Taking $x_m = \frac{2\pi m}{n}$, the first term of the numerator is equal to zero since   $\sin((p+\frac{1}{2})x_m) = \sin(\pi m) =0$. Then, combining the above with \eqref{alphaCoro} yields 
\begin{align*}
\alpha_m=  -\frac{2}{n} Re\left[-\iota f'\left(\frac{2\pi m}{n}\right)\right]  =  \frac{  \sin^2(\frac{p+1}{2}x_m)}{ n\sin^2(\tfrac{\pi m}{n}) }\enspace.
\end{align*}
Since $ \frac{p +1}{2} x_m = m \frac{\pi}{2} + m\frac{\pi}{2n}$, 
$$
\alpha_m = \left\{
    \begin{array}{ll}
        \frac{  \cos^2(\frac{m \pi}{2 n})}{ n\sin^2(\frac{\pi m}{n}) }  & \mbox{if} \, m \textup{ is odd,} \\
          & \mbox{} \\
        \frac{  \sin^2(\frac{m \pi}{2 n})}{ n\sin^2(\frac{\pi m}{n}) }  & \mbox{if} \, m \textup{ is even.}
    \end{array}
\right.
$$
Hence, the sequence of eigenvalues with odd indices is decreasing: $$\alpha_1 > \alpha_3 > \alpha_5 > \ldots > \ldots \enspace ,$$ 
since the fraction $\cos^2(m\pi /2n)/\sin^2(\pi m /n)$ decreases  with $m\in[p]$. In the (remaining) case of even indices, the numerator can be upper bounded as follows:   \, $\sin^2(\frac{m \pi}{2 n}) \leq \sin^2(\pi/4) \leq \textup{min}_{k\in [p]} \cos^2(\frac{k \pi}{2 n})$, which leads to  $\alpha_{2r} \leq \alpha_{2r-1}$ for all $r=1,\ldots,\lfloor p/2 \rfloor$. 

In other words, each eigenvalue of even index is upper bounded by the previous eigenvalue. In light of this, we only need to prove that, for $n$ large enough,
\[
(\alpha_0-\alpha_1) \wedge (\alpha_1-\alpha_2)\wedge (\alpha_1-\alpha_3)\gtrsim  n\enspace .   
\]
From~\eqref{eq:upper_alpha_0}, we deduce that $\alpha_0$ is equivalent to $3n/4$. Besides, we deduce from the explicit form of $\alpha_m$ in the general case that $\alpha_1$, $\alpha_2$, and $\alpha_3$ are respectively equivalent to $\frac{n}{\pi^2}$, $1/(4n)$, and $n/(4\pi^2)$. This completes the proof. 
\hfill $\square$

\subsection{Proof of Proposition \ref{conj:graphgeo:l1} }\label{proof:bidon:spectral:remaniement}

For the first inequality of the proposition, the proof is the same as for Proposition \ref{conj:graphgeo:l1:new}, after replacing $\hxsp_S,{\bf x}^*_{S}$, $A_{SS}$   respectively by $\hxsp,{\bf x}^*$, $A$. The second inequality of the proposition follows from Lemma \ref{lem:fourier}. \hfill $\square$